\documentclass[nohyperref]{article}

\usepackage{amsmath,amsfonts,bm, dsfont}

\usepackage{natbib}

\usepackage{enumitem}

\usepackage[utf8]{inputenc} 
\usepackage[T1]{fontenc}    
\usepackage{booktabs}       
\usepackage{amsfonts}       
\usepackage{nicefrac}       
\usepackage{microtype}      

\usepackage{subfigure}
\usepackage{epsf}
\usepackage{epsfig}
\usepackage{fancyhdr}
\usepackage{graphics}
\usepackage{graphicx}
\usepackage{psfrag}
\usepackage{pdfpages}

\usepackage{url}
\usepackage[colorlinks,linkcolor=magenta,citecolor=blue, pagebackref=true,backref=true]{hyperref}
\renewcommand*{\backrefalt}[4]{%
    \ifcase #1 \footnotesize{(Not cited.)}%
    \or        \footnotesize{(Cited on page~#2.)}%
    \else      \footnotesize{(Cited on pages~#2.)}%
    \fi}
\usepackage{amsthm}
\usepackage{amsmath}
\usepackage{amssymb,bbm}
\usepackage{caption}
\usepackage{multicol}
\usepackage{caption}
\usepackage{color}

\usepackage{changes}
\definechangesauthor[name=Huy Nguyen, color=purple]{HN}

\DeclareMathOperator*{\argmin}{arg\,min}
\DeclareMathOperator{\tr}{tr}
\DeclareMathOperator{\id}{Id}
\DeclareMathOperator{\cov}{Cov}

\usepackage{changes}

\def\Nn{\mathcal{N}}

\def\E{\mathbb{E}}

\def\RR{\mathbb{R}}
\def\dd{\mathrm{d}}

\def\KL{\mathrm{KL}}
\def\zeros{\mathbf{0}}
\def\KLD{\mathrm{KL}^\otimes}
\def\br{\mathbb{R}}

\def\diag{\mathrm{diag}}

\def\aA{\mathbf{a}}
\def\bB{\mathbf{b}}

\def\uU{\mathbf{u}}
\def\vV{\mathbf{v}}

\def\det{\mathrm{det}} 

\usepackage{microtype}
\usepackage{graphicx}
\usepackage{subfigure}
\usepackage{booktabs} 

\usepackage{hyperref}

\usepackage{icml2022}

\usepackage{amsmath}
\usepackage{amssymb}
\usepackage{mathtools}
\usepackage{amsthm}

\usepackage[capitalize,noabbrev]{cleveref}

\theoremstyle{plain}
\newtheorem{theorem}{Theorem}[section]

\newtheorem{lemma}[theorem]{Lemma}

\theoremstyle{definition}
\newtheorem{definition}[theorem]{Definition}
\newtheorem{assumption}[theorem]{Assumption}
\theoremstyle{remark}
\newtheorem{remark}[theorem]{Remark}

\icmltitlerunning{Entropic Gromov-Wasserstein between Gaussian Distributions}

\begin{document}

\twocolumn[
\icmltitle{Entropic Gromov-Wasserstein between Gaussian Distributions}



\icmlsetsymbol{equal}{*}

\begin{icmlauthorlist}
\icmlauthor{Firstname1 Lastname1}{equal,yyy}
\icmlauthor{Firstname2 Lastname2}{equal,yyy,comp}
\icmlauthor{Firstname3 Lastname3}{comp}
\icmlauthor{Firstname4 Lastname4}{sch}
\icmlauthor{Firstname5 Lastname5}{yyy}
\icmlauthor{Firstname6 Lastname6}{sch,yyy,comp}
\icmlauthor{Firstname7 Lastname7}{comp}
\icmlauthor{Firstname8 Lastname8}{sch}
\icmlauthor{Firstname8 Lastname8}{yyy,comp}
\end{icmlauthorlist}

\icmlaffiliation{yyy}{Department of XXX, University of YYY, Location, Country}
\icmlaffiliation{comp}{Company Name, Location, Country}
\icmlaffiliation{sch}{School of ZZZ, Institute of WWW, Location, Country}

\icmlcorrespondingauthor{Firstname1 Lastname1}{first1.last1@xxx.edu}
\icmlcorrespondingauthor{Firstname2 Lastname2}{first2.last2@www.uk}

\icmlkeywords{Machine Learning, ICML}

\vskip 0.3in
]



\printAffiliationsAndNotice{\icmlEqualContribution} 

\begin{abstract}
We study the entropic Gromov-Wasserstein and its unbalanced version between (unbalanced) Gaussian distributions with different dimensions. When the metric is the inner product, which we refer to as inner product Gromov-Wasserstein (IGW), we demonstrate that the optimal transportation plans of entropic IGW and its unbalanced variant are (unbalanced) Gaussian distributions. Via an application of von Neumann's trace inequality, we obtain closed-form expressions for the entropic IGW between these Gaussian distributions. Finally, we consider an entropic inner product Gromov-Wasserstein barycenter of multiple Gaussian distributions. We
prove that the barycenter is a Gaussian distribution when the entropic regularization parameter is small. We further derive a closed-form expression for the covariance matrix of the barycenter.
\end{abstract}

\section{Introduction}
The recent advance in computation of optimal transport~\citep{cuturi2013sinkhorn, bonneel2015sliced, kolouri2016sliced, altschuler2017near, Dvurechensky-2018-Computational, lin2019efficient, deshpande2019max, nguyen2021distributional} has led to a surge of interest in using optimal transport for applications in machine learning and statistics. These applications include generative modeling \citep{arjovsky2017wasserstein,tolstikhin2018wasserstein,salimans2018improving,genevay2018learning,liutkus2019sliced}, unsupervised learning~\citep{ho2017multilevel, Ho_Probabilistic}, domain adaptation~\citep{courty2016optimal, damodaran2018deepjdot, ho2021lamda}, mini-batch computation~\citep{fatras2020learning,Nguyen_bomb}, and other applications~\citep{solomon2016entropic, Nguyen_3D}. 
In these above applications, optimal transport has been used to quantify the discrepancy  of the locations and mass between two probability measures which must be in the same space.

When probability measures lie in different spaces, their locations are not comparable, their inner structures become the main concern. 
In this case, the Gromov-Wasserstein distance is adopted to measure the discrepancy between their inner structures. 
The main idea of Gromov-Wasserstein distance is to consider the transportation of similarity matrices of points which are in the same spaces. However, such formulation of Gromov-Wasserstein distance is computationally expensive as we need to solve a quadratic programming problem. To reduce the computational cost of Gromov-Wasserstein distance,~\cite{peyre2016gromov} propose to regularize the objective function of Gromov-Wasserstein based on the entropy of the transportation plan, which we refer to as \emph{entropic Gromov-Wasserstein}. The entropic regularization idea had also been used earlier in optimal transport and had been shown to improve greatly the computation of optimal transport~\citep{cuturi2013sinkhorn, altschuler2017near, Dvurechensky-2018-Computational, lin2019efficient}. The improvement in approximation of Gromov-Wasserstein distance via the entropic regularization has led to an increasing interest of using that divergence to several applications, including deep generative models~\citep{bunne2019learning}, computer graphics~\citep{solomon2016entropic, xu2019scalable, chen2020graph}, and natural language processing~\citep{Alvarez-2018-Gromov, grave2019unsupervised}. When measures are not probability measures and unbalanced, i.e., they may have different total masses, unbalanced version of Gromov-Wasserstein, named unbalanced Gromov-Wasserstein, via the idea of unbalanced optimal transport~\citep{chizat2018scaling}
had been introduced in the recent work of~\cite{Peyreunbalanced}. The entropic unbalanced Gromov-Wasserstein has been used in robust machine learning applications~\citep{Peyreunbalanced}. Despite their practicalities, theoretical understandings of entropic (unbalanced) Gromov-Wasserstein are still nascent.
The recent work of~\cite{Salmona2021gromov} studied the closed-form expression of Gromov-Wasserstein between Gaussian distributions in different dimensions. However, to the best of our knowledge, the full theoretical analysis of entropic Gromov-Wasserstein and its unbalanced version between (unbalanced) Gaussian distributions in different dimensions has remained poorly understood.

\vspace{0.5 em}
\noindent
\textbf{Our contribution.} In this paper, we present a comprehensive study of the entropic Gromov-Wasserstein and its unbalanced version between (unbalanced) Gaussian distributions when inner product is considered as a cost function, which we refer to as entropic (unbalanced) inner product Gromov-Wasserstein (IGW). We also study its corresponding barycenter problem among multiple Gaussian distributions. Our work also complements the works of \cite{Gelbrich1990}, \cite{janati2020entropic} and ~\cite{Salmona2021gromov} when we give an explicit explanation of the effect of the entropic regularization on the OT and the geometric structure of the optimal transport plan between Gaussians.  More particular, our contribution is four-fold and can be summarized as follows:

\textbf{1. Balanced Gaussian measures:} We first provide a closed-form expression of the entropic inner product Gromov-Wasserstein between Gaussian probability measures  with zero means. We demonstrate that the expression depends mainly on eigenvalues of the covariance matrices of the two measures. Furthermore, an associate optimal transportation plan is shown to be also a zero-mean Gaussian measure. Our analysis hinges upon an application of von Neuman's trace inequality \citep{kristof1969neumann,horn_johnson_1991} for singular values of matrices.

\textbf{2. Unbalanced Gaussian measures:} Second, by relaxing marginal constraints via Kullback-Leibler (KL) divergence, we further study the entropic unbalanced IGW between two unbalanced Gaussian measures with zero means. The main challenge comes from the two $\KL$ terms that add another level of constraint in the objective function. To overcome that challenge, we show that the objective function can be broken down into several sub-problems that can be solved explicitly through some cubic and quadratic equations. That leads to an almost closed-form formulation of the optimal Gaussian transport plan of the unbalanced IGW between the unbalanced Gaussian measures.  

\textbf{3. Barycenter problem:} Finally, we investigate the (entropic) Gromov-Wasserstein barycenter problems with inner product cost, which we refer to as (entropic) IGW barycenter, among zero-mean balanced Gaussian measures. We prove that the barycenter of zero-mean Gaussian distributions is also a zero-mean Gaussian distribution with a diagonal covariance matrix. Reposing on that Gaussian barycenter, we can directly obtain a closed-form expression for the barycenter problem when the regularized parameter is sufficiently small, which is also the setting people widely use in practice.

\textbf{4. Revisiting entropic optimal transport between Gaussian distributions.} We note in passing that based on the new proof technique introduced for computing the entropic Gromov-Wasserstein between Gaussian distributions, we revisit the entropic Wasserstein between (unbalanced) Gaussian distributions in Appendix~\ref{sec:revisit_entropic_Gaussian} and provide a simpler proof than that in~\citep{janati2020entropic} to derive the closed-form expression for that problem.

\vspace{0.5 em}
\noindent
\textbf{Paper organization.} The paper is organized as follows. In Section \ref{sec:prelim}, we   provide backgrounds for the (unbalanced) inner product Gromov-Wasserstein and its corresponding entropic version. Next, we establish a closed-form expression of the entropic IGW between two balanced Gaussian distributions in Section~\ref{sec:entropic_GW} while proving that the entropic unbalanced IGW between unbalanced Gaussian measures also has a closed form in Section~\ref{sec:unbalanced_GW}. Subsequently, we present our study of the (entropic) IGW barycenter problem among multiple Gaussian distributions in Section~\ref{sec:barycenter_GW}. In Section~\ref{sec:empirical_study}, we carry out empirical studies to investigate the behavior of algorithms solving the entropic Gromov-Wasserstein before concluding with a few discussions in Section~\ref{sec:conclusion}. Additional proofs and auxiliary results are presented in the supplementary material.

\vspace{0.5 em}
\noindent
\textbf{Notation.}
We use the following notations throughout the paper. For a non-negative definite matrix $A$, if we do not specify the spectral decomposition of $A$, then generally, $\lambda_{a,i}$ or $\lambda_i(A)$ denotes the $i$-th largest eigenvalues of $A$. For random vectors  $X = (X_1,\ldots,X_m)$ and $Y= (Y_1,\ldots,Y_n)$, let $K_{xy}$ be the covariance matrix between $X$ and $Y$, which means that  $\big(K_{xy}\big)_{i,j}= \cov(X_i,Y_j)$. For square matrices $A$ and $B$, we write $A\succeq B$ or $B\preceq A$ if $A-B$ is non-negative definite. For a positive integer $n$, we denote by $\id_n$ the identity matrix of size $n\times n$ while $[n]$ stands for the set $\{1,2,\ldots,n\}$. For any real number $a$, $[a]^+$ indicates $\max\{a,0\}$. Lastly, we denote by $\mathcal{M}^+(\mathcal{X})$ the set of all positive measures in a space $\mathcal{X}$ and by $\mathcal{N}(\gamma,\Sigma)$ the Gaussian distribution in $\br^d$ with mean $\gamma$ and variance $\Sigma$ where $d$ will be specified in each case.

\section{Preliminaries}\label{sec:prelim}
In this section, we first provide background for the Gromov-Wasserstein distance between two probability distributions and discuss its properties when using $\ell_2$-norm as a cost function. Then, we present the formulation of  an inner product Gromov-Wasserstein problem between two (unbalanced) Gaussian measures with a discussion about its geometric properties. 
\subsection{Gromov-Wasserstein distance}
Let $\mu$ and $\nu$ be two probability measures on two Polish spaces $(\mathcal{X},d_\mathcal{X})$ and $(\mathcal{Y},d_\mathcal{Y})$, and let $c_\mathcal{X}:\mathcal{X}\times\mathcal{X}\to\br$ and $c_\mathcal{Y}:\mathcal{Y}\times\mathcal{Y}\to\br$ be two measurable functions. Then, with $p\geq 1$, the $p$-Gromov-Wasserstein distance between $\mu$ and $\nu$, denoted by $\mathsf{GW}_p(\mu,\nu)$, is defined as follows:
\begin{align}\label{def:gw_distance}
    \left(\inf_{\pi\in\Pi(\mu,\nu)}\mathbb{E}_{\pi\otimes\pi}\big|c_\mathcal{X}(X,X^\prime)-c_\mathcal{Y}(Y,Y^\prime)\big|^p\right)^{\frac{1}{p}},
\end{align}
where $(X,Y)$ and $(X^{\prime}, Y^{\prime})$ are independent and  identically distributed with probability distribution $\pi$ that belongs to the set of joint probability distributions $\Pi(\mu,\nu)$, in which their marginal distributions are the corresponding distributions $\mu$ and $\nu$.


Although the above problem always admits an optimal solution \citep{vayer2020contribution}, it is computationally expensive as we need to solve a quadratic optimization problem. For example, when the dimensions of $\mathcal{X}$ and $\mathcal{Y}$ are $m$ and $n$, respectively, the complexity of solving the Gromov-Wasserstein minimization problem is $\mathcal{O}(m^2n+n^2m)$ \citep{peyre2016gromov,vincentcuaz2022semirelaxed}.


\subsection{Gromov-Wasserstein distance  using $\ell_2$-norm}\label{section:l2_GW} 
Similar to the Wasserstein distance, there is no closed-form expression for the Gromov-Wasserstein distance between general distributions. Even in the important and specific cases when the input distributions are Gaussian, we still do not know if the optimal transport plan needs to be Gaussian or not \cite{Salmona2021gromov}. However, there are some certain understandings  of the Gromov-Wasserstein distance when the cost functions are Euclidean norms \citep{Salmona2021gromov, vayer2020contribution, Gelbrich1990},
\begin{align}\label{eq:GW_ell_2_formulation}
    \left(\inf_{\pi\in\Pi(\mu,\nu)}\mathbb{E}_{\pi\otimes\pi}\big|\|X-X^\prime\|^2_m-\|Y-Y^\prime\|_n\big|^2\right)^{\frac{1}{2}},
\end{align}
where $\|\cdot\|_m$ and $\|\cdot\|_n$ are Euclidean norms on $\br^m$ and $\br^n$, respectively. 
\begin{theorem}\cite{Salmona2021gromov}
Let $\mu \sim \mathcal{N}(m_0,\Sigma_0)$ and $\nu \sim\mathcal{N}(m_1,\Sigma_1)$ where $m_0\in \mathbb{R}^m$,$m_1\in \mathbb{R}^n$ are two vector means and $\Sigma_0\in \mathbb{R}^{m\times m}$ and $\Sigma_1\in \mathbb{R}^{n\times n}$ are the corresponding non-degenerated covariance matrices. Let $P_0,D_0$ and $P_1,D_1$ be respective diagonalizations of $\Sigma_0$ and $\Sigma_1$. Let us define $T_0:x\in \mathbb{R}^m \rightarrow P_0^{\top}(x-m_0) $ and $T_1:y\in \mathbb{R}^n\rightarrow P_1^{\top}(y-m_1)$. The Gromov-Wasserstein problem in equation~\eqref{eq:GW_ell_2_formulation} is equivalent to 
\begin{align*}
    \sup_{X\sim T_0\#\mu,T_1\#\nu} \sum_{i,j}\mathsf{Cov}(X_i^2,Y_j^2) + 2\big\|\mathsf{Cov}(X,Y) \big\|_F^2
\end{align*}
where $X=(X_1,\ldots,X_m)^{\top}, Y=(Y_1,\ldots,Y_n)^{\top}$ and $\|.\|_F$ is the Frobenius norm. 
\end{theorem}
This theorem shows that the $\ell_2$-norm Gromov-Wasserstein problem is equivalent to first shifting two input distributions to have mean zeros and then comparing covariance structures of the shifted distributions. After the shifting step, the translation invariant property is no longer needed and the problem is reduced to finding a match between the covariance structures. 


\subsection{Inner product Gromov-Wasserstein}
\label{sec:problem_settings}

We now reintroduce the Gromov-Wasserstein distance using the inner product  as a \textsl{cost} function and its unbalanced variant. 
\begin{definition}[\textbf{Inner product Gromov-Wasserstein}] \cite{Salmona2021gromov}
Let $X,X'\sim \mu$ be i.i.d random multivariates in $\mathbb{R}^m$ while $Y,Y'\sim \nu$ are i.i.d random multivariates in $\mathbb{R}^n$. Next, we assume that both $(X,Y)$ and $(X',Y')$ are jointly distributed as $\pi$.  
Then, the inner product Gromov-Wasserstein (IGW) problem is defined as
\begin{align*}
     \mathsf{IGW}(\mu,\nu):=\min_{\pi\in \Pi(\mu,\nu)}  \E_{\pi\otimes\pi} \Big\{ \big[ \langle X,X^{\prime}\rangle - \langle Y,Y^{\prime}\rangle \big]^2 \Big\}.
\end{align*}
\end{definition}
By adding the relative entropy term to the above objective function, we have the following definition, named \emph{entropic IGW}, is given by
\begin{align}
     \mathsf{IGW}_{\varepsilon}(\mu,\nu) := \min_{\pi\in \Pi(\mu,\nu) } &\E_{\pi\otimes\pi} \Big\{ \big[ \langle X,X^{\prime}\rangle - \langle Y,Y^{\prime}\rangle \big]^2 \Big\} \nonumber \\
     &\qquad+ \varepsilon \KL\big(\pi \| \mu\otimes \nu \big), \label{eq:entropic_IGW}
\end{align}
where $\varepsilon>0$ is a regularized parameter and $\KL(\alpha||\beta):=\int\log\Big(\frac{d\alpha}{d\beta}\Big)d\alpha+\int(d\beta-d\alpha)$ for any positive measures $\alpha$ and $\beta$.

\vspace{0.5 em}
\textbf{Entropic unbalanced IGW problem:} When $\mu$ and $\nu$ are unbalanced Gaussian measures which are not probability distributions, the entropic IGW formulation in equation~\eqref{eq:entropic_IGW} is no longer valid. One solution to deal with this issue is by regularizing the marginal constraints via KL divergence~\citep{chizat2018scaling}. The entropic IGW problem between unbalanced measures, which we refer to as \emph{entropic unbalanced IGW}, admits the following form:
\begin{align}
    \label{eq:UIGW_formulation}
     & \mathsf{UIGW}_{\varepsilon,\tau}(\mu,\nu):=\min_{\pi} \E_{\pi\otimes\pi} \Big\{ \big[ \langle X,X^{\prime}\rangle - \langle Y,Y^{\prime}\rangle \big]^2 \Big\}\nonumber \\
     &+ \tau \KL^{\otimes}(\pi_{x} \|\mu) + \tau \KL^{\otimes}(\pi_y \|\nu) +   
    \varepsilon \KL^{\otimes}\big(\pi \| \mu\otimes \nu \big),
\end{align}
where $\tau>0$, $\pi \in\mathcal{M}^+(\mathbb{R}^m\times\mathbb{R}^n)$ in the minimum, $\pi_x$ and $\pi_y$ are the projections of $\pi$ on $\mathbb{R}^m$ and $\mathbb{R}^n$, respectively, and $\KL^\otimes(\alpha||\beta):=\KL(\alpha\otimes\alpha||\beta\otimes\beta)$. Using the quadratic-KL divergence $\KL^\otimes$ \citep{sejourne2021unbalanced} in place of the standard $\KL$ will result in $\mathsf{UIGW}_{\varepsilon,\tau}$ being 2-homogeneous, i.e., if $\mu$ and $\nu$ are multiplied by $\theta\geq 0$, then the value of $\mathsf{UIGW}_{\varepsilon,\tau}(\mu,\nu)$ is multiplied by $\theta^2$. 

\textbf{Invariant property:} Unlike the $\ell_2$-norm Gromov Wasserstein in equation~\eqref{eq:GW_ell_2_formulation}, which is invariant to both translation and rotation \citep{Salmona2021gromov}, our $\mathsf{IGW}_{\varepsilon}$ and $\mathsf{UIGW}_{\varepsilon,\tau}$  only keeps unchanged under the latter transformation, which is shown in the following lemma with a note that the KL divergence is both rotation and translation invariant.
\begin{lemma}[\textbf{Invariant to rotation}]
\label{lemma:IGW_rotation}
Let $\mu$ and $\nu$ be two probability measures on $\br^m$ and $\br^n$. Let $T_m:x\to O_mx$ and $T_n:y\to O_ny$ be two functions in which $O_m$ and $O_n$ are orthogonal matrices of size $m\times m$ and $n\times n$, respectively. Then, we have
\begin{align*}
    \mathsf{IGW}_\varepsilon(T_m\#\mu,T_n\#\nu)=\mathsf{IGW}_\varepsilon(\mu,\nu),
\end{align*}
where $\#$ denotes the push-forward operator. 
\end{lemma}
The proof of Lemma~\ref{lemma:IGW_rotation} can be found in Appendix~\ref{appendix:IGW_rotation}.

\textbf{Mean-zero assumption:} As being discussed above, the $\mathsf{IGW}_{\varepsilon}$ and $\mathsf{UIGW}_{\varepsilon,\tau}$ are not invariant to translation, because the $\mathsf{IGW}$ depends on the means of distributions, which is apparently undesirable. To get around this cumbersome, we assume that the means of two distributions are equal to zero, which is equivalent to the shifting step as shown in Section~\ref{section:l2_GW}. Then, the translation invariant property is no longer a concern, and the problem is reduced to working with the covariance structures. By putting the mean-zero assumption, we still follow the main idea of the Gromov-Wasserstein distance, which is comparing the inner structures of two input distributions.  


\section{Entropic Gromov-Wasserstein between balanced Gaussians}
\label{sec:entropic_GW}
We present in this section a closed-form expression of the entropic inner product Gromov-Wasserstein between two Gaussian measures and a formula of a corresponding optimal transport plan in Theorem \ref{theorem:entropic_gw:inner_product:closed_form}. 



Let $\mu=\mathcal{N}(\mathbf{0},\Sigma_\mu)$ and $\nu=\mathcal{N}(\mathbf{0},\Sigma_\nu)$ be two Gaussian measures in $\mathbb{R}^m$ and $\mathbb{R}^n$, respectively. To ease the ensuing presentation, in the entropic $\mathsf{IGW}$ problem, we denote the covariance matrix of the joint distribution $\pi$ of $\mu$ and $\nu$ by
\begin{align}
    \Sigma_{\pi} = \begin{pmatrix}\Sigma_{\mu} & K_{\mu\nu} \\
    K_{\mu\nu}^{\top} & \Sigma_{\nu}
    \end{pmatrix}. \label{eq:plan_structure_balanced}
\end{align}
In addition, the rotation invariant property of entropic IGW in Lemma~\ref{lemma:IGW_rotation} allows us to assume without loss of generality that $\Sigma_{\mu}$ and $\Sigma_{\nu}$ are diagonal matrices with positive diagonal entries sorted in descending order. In particular,
\begin{align}
    \Sigma_{\mu} &= \diag\big(\lambda_{\mu,1},\ldots, \lambda_{\mu,m} \big), \nonumber \\
    \Sigma_{\nu} &= \diag\big(\lambda_{\nu,1},\ldots,\lambda_{\nu,n} \big). \label{eq:covariance_structures}
\end{align}
This step simplifies significantly the unnecessarily complicated form of the general covariance, since the Gromov-Wasserstein depends on the eigenvalues of the covariance matrix.  
The following result gives the closed-form expression for the entropic IGW between two balanced Gaussian distributions.
\begin{theorem}[\textbf{Entropic IGW has a closed form}]
\label{theorem:entropic_gw:inner_product:closed_form}
Suppose without loss of generality that $m\geq n$. Let $\mu=\mathcal{N}(\mathbf{0},\Sigma_\mu)$ and $\nu=\mathcal{N}(\mathbf{0},\Sigma_\nu)$ be two zero-mean Gaussian measures in $\mathbb{R}^m$ and $\mathbb{R}^n$, respectively, where $\Sigma_\mu$ and $\Sigma_\nu$ are diagonal matrices given in equation~\eqref{eq:covariance_structures}. The entropic inner product Gromov-Wasserstein between $\mu$ and $\nu$ then equals to
\begin{align*}
    \mathsf{IGW}_\varepsilon(\mu,\nu) &= \tr(\Sigma_{\mu}^2) +  \tr(\Sigma_{\nu}^2) -  2 \sum_{k=1}^{n} \Big[\lambda_{\mu,k} \lambda_{\nu,k} - \frac{\varepsilon}{4}\Big]^+ \\
    & + \frac{\varepsilon}{2} \sum_{k=1}^{n}\Big[\log(\lambda_{\mu,k} \lambda_{\nu,k}) - \log\Big(\frac{\varepsilon}{4}\Big)  \Big]^+.
\end{align*}
Furthermore, the optimal transportation plan of the entropic IGW problem is a zero-mean Gaussian measure $\pi^*=\mathcal{N}(\mathbf{0},\Sigma_{\pi^*})$, where
\begin{align}\label{eq:IGW_transport_plan}
    \Sigma_{\pi^*} = \begin{pmatrix}
    \Sigma_{\mu} & K^*_{\mu \nu} \\
    (K^*_{\mu \nu})^{\top} & \Sigma_{\nu}
    \end{pmatrix},
\end{align}
with $K^*_{\mu\nu} = \diag\Big( \Big[\lambda_{\mu,k} \lambda_{\nu,k}\big[1- \frac{\varepsilon}{4\lambda_{\mu,k}\lambda_{\nu,k}} \big]^+\Big]^{\frac{1}{2}}\Big)_{k=1}^n$ is an $m\times n$ matrix.
\end{theorem}
\begin{remark}
It can be seen that the quantity $\mathsf{IGW}_\varepsilon(\mu,\nu)$ depends only on the eigenvalues of the covariance matrices of $\mu$ and $\nu$. In addition, as $\varepsilon\to 0$, $\mathsf{IGW}_\varepsilon(\mu,\nu)$ converges to $\sum_{k=1}^n(\lambda_{\mu,k}-\lambda_{\nu,k})^2+\sum_{k=n+1}^m\lambda_{\mu,k}^2$, which is the squared Euclidean distance between covariance matrices $\Sigma_\mu$ and $\Sigma_\nu$. In other words, $\mathsf{IGW}_\varepsilon(\mu,\nu)$ reflects the difference in the inner structures of the two input measures. Regarding the optimal transport plan, when $\Sigma_\mu$ and $\Sigma_\nu$ are not diagonal, the covariance matrix $K^*_{\mu\nu}$ in equation~\eqref{eq:IGW_transport_plan} turns into
\begin{align*}
    K^{*,\text{new}}_{\mu\nu}=P_{\mu}D^{\frac{1}{2}}_{\mu}\Lambda_{\mu\nu}^{\frac{1}{2}}D^{\frac{1}{2}}_{\nu}P^{\top}_{\nu}=P_{\mu}K^*_{\mu\nu}P^{\top}_{\nu}.
\end{align*}
where $\Lambda_{\mu\nu}^{\frac{1}{2}}=\diag\Big( \Big\{ \Big[1 - \frac{\varepsilon}{4\lambda_{\mu,k}\lambda_{\nu,k}} \Big]^+\Big\}^{\frac{1}{2}}\Big)_{k=1}^n$, $P_{\mu},D_{\mu}$ and $P_{\nu},D_{\nu}$ are the orthogonal diagonalizations of $\Sigma_\mu(=P_\mu D_\mu P^{\top}_\mu)$ and $\Sigma_\nu(=P_\nu D_\nu P^{\top}_\nu)$, respectively (see Appendix~\ref{appendix:entropic_gw:inner_product:closed_form} for the derivation).
\end{remark}
\begin{proof}[Proof Sketch of Theorem~\ref{theorem:entropic_gw:inner_product:closed_form}]
The full proof of this theorem is in Appendix~\ref{appendix:entropic_gw:inner_product:closed_form}. First of all, we note  that the objective function depends  only on the covariance matrix of the joint distribution and the KL term. Given a fixed covariance matrix, the entropic term in equation~\eqref{eq:entropic_IGW} forces the optimal transport plan for the entropic IGW problem to be a Gaussian distribution (see Lemma~\ref{lemma:min_KL_Gaussian} below). Subsequently, Lemma~\ref{lemma:entropic_gw:inner_product:closed_form:supporting} (cf. the end of this proof) 
is used to attack the the determinant and the trace involving the covariance matrix $K_{\mu\nu}$ between $X$ and $Y$. Then, the entropic IGW is reduced to a minimization problem of single variables in a particular interval, which has a closed-form solution. 
\end{proof}
\begin{lemma}[\textbf{KL divergence minimum}] \label{lemma:min_KL_Gaussian}
Let $Q_{v}$ be the Gaussian distribution in $\br^d$ with mean $v$ and variance $\Sigma_{v}$. Denote by $\Pi_{u,\Sigma_{u}}$ the family of all probability distributions in $\br^d$ which have mean $u$ and variance $\Sigma_{u}$. Then,
\begin{align*}
   \mathcal{N}\big(u,\Sigma_{u}\big) =\argmin_{P_u\in\Pi_{u,\Sigma_u}} \KL(P_{u}\|Q_{v}).
\end{align*}
\end{lemma}
The proof of this lemma is deferred to Appendix~\ref{appendix:min_KL_Gaussian}
\begin{lemma}[\textbf{Eigenvalues, determinant and trace}]
\label{lemma:entropic_gw:inner_product:closed_form:supporting}
Let $U_{\mu\nu} \Lambda_{\mu\nu}^{\frac{1}{2}} V_{\mu\nu}^{\top}$ be the SVD of  $\Sigma_{\mu}^{-\frac{1}{2}} K_{\mu\nu}\Sigma^{-\frac{1}{2}}_{\nu}$, where  $\Lambda_{\mu\nu}^{\frac{1}{2}} = \diag(\kappa_{\mu\nu,k}^{\frac{1}{2}})_{k=1}^n$ with $\kappa^{\frac{1}{2}}_{\mu\nu,k}$ is the $k$-th largest singular value of $\Sigma_{\mu}^{-\frac{1}{2}} K_{\mu\nu}\Sigma^{-\frac{1}{2}}_{\nu}$. Here, $\Lambda_{\mu\nu}^{\frac{1}{2}}$ is an $m\times n$ matrix, while $U_{\mu\nu}$ and $V_{\mu \nu}$ are unitary matrices of sizes $m\times m$ and $n\times n$, respectively. Then,
\begin{itemize}
\item[(a)] The values of $\kappa_{\mu\nu,1},\ldots,\kappa_{\mu\nu,n}$ lie between $0$ and $1$.
\item[(b)] The determinant of matrix $\Sigma_\pi$ could be computed as $\det(\Sigma_\pi) = \det(\Sigma_{\mu}) \det(\Sigma_{\nu}) \prod_{j=1}^{n} \big(1-\kappa_{\mu\nu,j}\big)$.
\item[(c)]  We have $\tr\big(K_{\mu\nu}^{\top} K_{\mu\nu}\big) \leq \sum_{j=1}^{n} \lambda_{\mu,j} \lambda_{\nu,j} \kappa_{\mu\nu,j}$.
The equality occurs when $U_{\mu\nu}$ and $V_{\mu\nu}$ are identity matrices of sizes $m$ and $n$, respectively.
\end{itemize}
\end{lemma}
The proof of Lemma~\ref{lemma:entropic_gw:inner_product:closed_form:supporting} can be found in Appendix~\ref{appendix:entropic_gw:inner_product:closed_form:supporting}. 

 Part (a) of Lemma \ref{lemma:entropic_gw:inner_product:closed_form:supporting} shows a representation of the covariance matrix $K_{\mu\nu}$ through its singular values, which is useful for calculating its determinant in part (b). 
Our key technique to derive the closed-form expression for the entropic IGW lies in part (c) of Lemma \ref{lemma:entropic_gw:inner_product:closed_form:supporting}. In that part, we utilize the von Neumann's trace inequality which says that the sum of singular values of the product of two matrices is maximized when their SVDs share the same orthogonal bases and their corresponding singular values match that of each other by their magnitudes.
In comparison with previous works such as ~\citep{janati2020entropic} and ~\cite{Salmona2021gromov}, we have a direct attack to the problem by figuring out a detailed formula for the covariance matrix and utilising the von Neumann's inequality to avoid solving a complicated system of equations of the first derivatives of the objective function.


\section{Entropic Unbalanced Gromov-Wasserstein between unbalanced Gaussians}
\label{sec:unbalanced_GW}
In this section, we consider a more general setting when the two distributions $\mu$ and $\nu$ are unbalanced Gaussian measures in $\br^m$ and $\br^n$, respectively. Specifically,
\begin{align}
    \mu&=m_\mu\mathcal{N}(\mathbf{0}_m,\Sigma_\mu), \nonumber \\
    \nu&=m_\nu\mathcal{N}(\mathbf{0}_n,\Sigma_\nu), \label{eq:unbalanced_Gauss}
\end{align}
where $m\geq n$; $m_\mu,m_\nu>0$ are their masses and $\Sigma_\mu$, $\Sigma_\nu$ are given in equation~\eqref{eq:covariance_structures}. Here, $\mu$ and $\nu$ do not necessarily have the same mass, i.e., $m_\mu$ could be different from $m_\nu$. For any positive measure $\alpha$, we denote by $\overline{\alpha}$ the normalized measure of $\alpha$. Thus, $\overline{\mu} = \mathcal{N}(\mathbf{0}_m,\Sigma_{\mu})$ and $\overline{\nu} = \mathcal{N}(\mathbf{0}_n,\Sigma_{\nu}) $.

Under the setting of entropic unbalanced IGW, we denote the covariance matrix of a transportation plan $\pi$ by
\begin{align}
\Sigma_{\pi} = \begin{pmatrix}
\Sigma_x & K_{xy}\\
K_{xy}^{\top} & \Sigma_y
\end{pmatrix}, \label{eq:transport_plan_unbalanced}
\end{align}
where $\Sigma_x$ and $\Sigma_y$ are covariance matrices of distributions $\pi_x$ and $\pi_y$, respectively. 

Note that the objective function of entropic unbalanced IGW in equation~\eqref{eq:UIGW_formulation} involves two new KL divergence terms compare to that of the entropic IGW problem. Thus, before presenting the the main theorem of this section, we introduce a lower bound for the KL divergence between two Gaussian distributions in the following lemma.
 \begin{lemma}[\textbf{Lower bound for KL divergence}]\label{lemma:lower_bound_KL}
 Let $\pi$ be given in equation~\eqref{eq:transport_plan_unbalanced}, $\lambda_x = (\lambda_{x,i})_{i=1}^m$ be the eigenvalues of $\Sigma_x$, and $\lambda_y = (\lambda_{y,j})_{j=1}^n$ be the eigenvalues of $\Sigma_y$. Similarly, we define $\lambda_{\mu} $ for $\Sigma_{\mu}$ and $\lambda_{\nu}$ for $\Sigma_{\nu}$. Then, we find that
\begin{itemize}
    \item[(a)]  $\KL(\overline{\pi}_x\|\overline{\mu}) \geq \frac{1}{2} \sum_{i=1}^m \Psi\big(\lambda_{x,i} \lambda_{\mu,i}^{-1}\big)$;
    \item[(b)] $\KL(\overline{\pi}_y\|\overline{\nu}) \geq \frac{1}{2} \sum_{j=1}^n \Psi\big(\lambda_{y,j}\lambda_{\nu,j}^{-1}\big)$; 
    \item[(c)]
    $\KL(\overline{\pi}\|\overline{\mu}\otimes \overline{\nu}) =  \KL(\overline{\pi}_x\|\overline{\mu}) +\KL(\overline{\pi}_y\|\overline{\nu})\\
    \textbf{}\qquad\qquad\textbf{} \qquad+\frac{1}{2}\sum_{k=1}^n \log(1- \kappa_{xy,k})$,
\end{itemize}
where $\Psi(x) = x - \log(x) - 1$, and $(\kappa_{xy,k}^{\frac{1}{2}})_{k=1}^n$ are singular values of matrix $\Sigma_{x}^{-\frac{1}{2}} K_{xy} \Sigma_y^{-\frac{1}{2}}$.
 \end{lemma}
The proof of Lemma~\ref{lemma:lower_bound_KL} is in Appendix~\ref{appendix:lower_bound_KL}. 
Next, we define some functions and quantities that will be used in our analysis. Given $\tau,\varepsilon>0$, let
\begin{align*}
 &g_{\varepsilon,\tau,+}(x,y;a,b):= x^2 + y^2+ \varepsilon \left[\log(xy)- \log \frac{\varepsilon}{2} \right]^+  \\
 &\quad + (\tau+ \varepsilon) \left[ \frac{x}{a}+ \frac{y}{b} - \log\left(\frac{xy}{ab}\right)-2 \right]- 2 \left[xy - \frac{\varepsilon}{2} \right]^+; \\
 &h_{\varepsilon,\tau}(x;a):= x^2 + (\tau+\varepsilon) \left[ \frac{x}{a}-\log\left(\frac{x}{a}\right)-1\right]; \\
 &(\varphi_k,\phi_k):= \argmin_{x,y>0}~g_{\varepsilon,\tau,+}(x,y;\lambda_{\mu,k}\lambda_{\nu,k}),~ k \in[n]; \\
 &\varphi_k:= \argmin_{x>0}~h_{\varepsilon,\tau}(x,\lambda_{\mu,k}),~ k=n+1,\ldots,m;
\end{align*}
for any $x,y,a,b>0$. The detailed calculation of $(\varphi_k)_{k=1}^m$ and $(\phi_k)_{k=1}^n$ can be found in Appendix~\ref{sec:upsilon_calculation}.

Now, we are ready to state our main result regarding the entropic unbalanced IGW between two unbalanced Gaussian distributions.
\begin{theorem}[\textbf{Entropic unbalanced IGW has a closed form}]
 \label{theorem:UGW}
 Let $\mu$ and $\nu$ be two unbalanced Gaussian distributions given in equation~\eqref{eq:unbalanced_Gauss}. Then, the entropic unbalanced IGW between $\mu$ and $\nu$ can be computed as follows:
 \begin{align}\label{eq:UIGW_closed-form}
     \mathsf{UIGW}_{\varepsilon,\tau}&(\mu,\nu) = m_{\pi^*}^2 \Upsilon^{*}+ \varepsilon\KL(m_{\pi^*}^2\|m_{\mu}^2 m_{\nu}^2)\nonumber \\
     &+ \tau \Big\{\KL(m_{\pi^*}^2\|m_{\mu}^2) + \KL(m_{\pi^*}^2\|m_{\nu}^2) \Big\},
 \end{align}
 where 
 \begin{align*}
     m_{\pi^*} &:= (m_{\mu}m_{\nu})^{\frac{\tau+\varepsilon}{2\tau + \varepsilon}}\exp\Big\{\frac{-\Upsilon^{*}}{2(2\tau + \varepsilon)} \Big\},
     \\
     \Upsilon^{*} &:= \sum_{k=1}^n g_{\varepsilon,\tau,+}(\varphi_k,\phi_k; \lambda_{\mu,k},\lambda_{\nu,k}) \\
     & \hspace{8 em} + \sum_{k=n+1}^m h_{\varepsilon,\tau}(\varphi_k;\lambda_{\mu,k}).
 \end{align*}
Furthermore, the optimal transportation plan is an unbalanced Gaussian measure $\pi^* = m_{\pi^*} \mathcal{N}(\zeros,\Sigma_{\pi^*})$, where 
 \begin{align*}
     \Sigma_{\pi^*} = \begin{pmatrix} \Sigma_x^{*} &K_{xy}^{*}\\
     (K_{xy}^{*})^{\top} &\Sigma_y^{*}
     \end{pmatrix},
\end{align*}
 in which $\Sigma_x^{*} = \diag\big(\varphi_k \big)_{k=1}^m$, $\Sigma_y^{*} = \diag\big(\phi_k \big)_{k=1}^n$, and $K_{xy}^{*} = \diag\big(\psi_k \big)_{k=1}^n$ is an $m\times n$ matrix with  $\psi_k := \Big\{\big[1 - \frac{\varepsilon}{2\varphi_k \phi_k}\big]^+ \varphi_k \phi_k\Big\}^{\frac{1}{2}}$ for all $k\in[n]$.
\end{theorem}
\begin{remark}
Similar to $\mathsf{IGW}_\varepsilon(\mu,\nu)$, equation~\eqref{eq:UIGW_closed-form} shows that $\mathsf{UIGW}_\varepsilon(\mu,\nu)$ still depends on the eigenvalues of the covariance matrices of $\mu$ and $\nu$ via the quantity $\Upsilon^*$, the detailed calculation of which is deferred to Appendix~\ref{sec:upsilon_calculation}. In addition, as $\mu$ and $\nu$ are assumed to be unbalanced in this case, the masses of these two measures are also included in the closed-form of $\mathsf{UIGW}_\varepsilon(\mu,\nu)$. 
\end{remark}
\begin{proof}[Proof Sketch of Theorem~\ref{theorem:UGW}]
The full proof of this theorem can be found in Appendix~\ref{appendix:UGW}. 
Recall that feasible transport plans $\pi$ in the problem \eqref{eq:UIGW_formulation} are not necessarily probability measures. Therefore, we aim to optimize over the shape $\bar{\pi}$ and the masses $m_\pi$ of $\pi$. To do so, we  separate those two factors in the objective function \eqref{eq:UIGW_formulation} as follows:
\begin{align}\label{eq:resulting_obj_function}
    m_{\pi}^2 \Upsilon + \varepsilon&\KL(m_{\pi}^2\|m_{\mu}^2 m_{\nu}^2)\nonumber\\
    &+ \tau \Big\{\KL(m_{\pi}^2\|m_{\mu}^2) + \KL(m_{\pi}^2\|m_{\nu}^2) \Big\},
\end{align}
where
\begin{align}\label{eq:shape_term}
    \Upsilon&:= \E_{\bar{\pi}\otimes\bar{\pi}}\big[ \langle X,X^{\prime}\rangle - \langle Y,Y^{\prime}\rangle \big]^2 + 2\varepsilon\KL(\bar{\pi}\|\bar{\mu}\otimes\bar{\nu})\nonumber\\
    &+2(\varepsilon + \tau) \Big\{ \KL(\bar{\pi}_x \| \bar{\mu}) + \KL(\bar{\pi}_y \| \bar{\nu}) \Big\}.
\end{align}
\paragraph{Shape optimization.} In a similar fashion to the proof of Theorem~\ref{theorem:entropic_gw:inner_product:closed_form}, we show that the sum of the expectation and entropic terms in equation~\eqref{eq:shape_term} depend only the covariance matrix $\Sigma_\pi$ of $\pi$ as below:
\begin{align*}
    \|\lambda_x\|_2^2 + \|\lambda_y\|_2^2 - &2\sum_{k=1}^n \Big(\lambda_{x,k}\lambda_{y,k} - \frac{\varepsilon}{2}\Big)^+\\
    &+ \varepsilon \sum_{k=1}^n   \Big[\log(\lambda_{x,k}\lambda_{y,k}) - \log\frac{\varepsilon}{2} \Big]^+,
\end{align*}
where $\lambda_x,\lambda_y,\lambda_\mu$ and $\lambda_\nu$ are defined as in Lemma~\ref{lemma:lower_bound_KL}.
Putting this result together with parts (a) and (b) of Lemma~\ref{lemma:lower_bound_KL}, the problem \eqref{eq:shape_term} is reduced to
\begin{align*}
    \min_{\lambda_x,\lambda_y>0}\Big\{\sum_{k=1}^n g_{\varepsilon,\tau, +}&(\lambda_{x,k},\lambda_{y,k}; \lambda_{\mu,k},\lambda_{\nu,k})\\
    &+\sum_{k=n+1}^m h_{\varepsilon,\tau}(\lambda_{x,k};\lambda_{\mu,k})\Big\}.
\end{align*}
Finally, Lemma~\ref{lemma:order_eigenvalues} (in Appendix~\ref{sec:auxiliary_results}) allows us to optimize each summation term in the above problem independently but still preserving the decreasing order of eigenvalue sequences $(\lambda_{x,k})_{k=1}^n$. Eventually, we obtain the optimal value $\Upsilon^*$, which will be calculated in detail in Appendix~\ref{sec:upsilon_calculation}.
\paragraph{Mass optimization.} Based on the equation \eqref{eq:resulting_obj_function}, the optimal mass $m_{\pi^*}$ is the square root of the minimizer of the function
\begin{align*}
    f(x):=x\Upsilon^*+ &\varepsilon\KL(x\|m_{\mu}^2 m_{\nu}^2)\\
    &+ \tau \Big\{\KL(x\|m_{\mu}^2) + \KL(x\|m_{\nu}^2) \Big\},
\end{align*}
which is obtained by solving the equation $f^\prime(x)=0$.
\end{proof}
\section{(Entropic) Gromov-Wasserstein Barycenter of Gaussian Distributions}
\label{sec:barycenter_GW}
In this section, we consider the problem of computing the (entropic) Gromov-Wasserstein barycenter of $T$ Gaussian measures $\mu_1,\ldots,\mu_T$ defined in spaces of various dimensions. To tackle this problem, we need to fix the desired dimension of a barycenter, e.g, $d$, and choose beforehand the positive weighting coefficients $\alpha_1,\ldots,\alpha_T$ such that $\sum_{\ell=1}^T \alpha_{\ell}=1$ associated with the measures $\mu_1,\ldots,\mu_T$. Here we allow the dimension of the barycenter to be flexible, since the given Gaussian measures could be in different dimensional spaces. Then, the barycenter of $T$ positive measures under the inner-product Gromov-Wasserstein distance is defined as follows:
\begin{definition}[\textbf{Inner product Gromov-Wasserstein barycenter}]
\label{def:IGW_barycenter}
Let $\mu_{\ell}$ be a positive measure in $\mathbb{R}^{m_{\ell}}$ for any $\ell\in [T]$. Let $\alpha_1,\ldots,\alpha_T$ be positive constants such that $\sum_{\ell=1}^{T} \alpha_{\ell} = 1$. The inner product Gromov-Wasserstein (IGW) barycenter of $\{\mu_{\ell},\alpha_{\ell}\}_{\ell=1}^{T}$ is defined as the probability distribution of the random vector $Y$ in $\mathbb{R}^d$ which is a solution of the following problem
\begin{align}
   \argmin_{\substack{X_{\ell}\sim \mu_{\ell};~\dim(Y) = d ;\\ (X_{\ell},Y),(X_{\ell}^{\prime},Y)\sim \pi_{\ell,y}}} \sum_{\ell=1}^T \alpha_{\ell} \mathbb{E}_{\pi_{\ell,y}\otimes\pi_{\ell,y}} \Big\{  \big[\langle X_{\ell},X_{\ell}^{\prime} \rangle - \langle Y, Y^{\prime} \rangle \big]^2 \Big\}, \label{definition:barycenter}
\end{align}
where in the above minimum, $(X_{\ell},Y)$ and $(X_{\ell}^{\prime},Y^{\prime})$ are i.i.d random vectors for each $\ell=1,\ldots, T$. 
\end{definition}
Based on this definition, the following theorem provides a closed-form expression for the IGW barycenter of $T$ Gaussian measures.

\begin{theorem}[\textbf{Inner product Gromov-Wasserstein barycenter has a closed form}]
\label{theorem:barycenter}
Let $\mu_{\ell}= \mathcal{N}(\mathbf{0}, \Sigma_{\ell})$ be Gaussian distribution in $\mathbb{R}^{m_\ell}$ for all $\ell\in[T]$, where $\Sigma_{\ell}=\diag\big(\lambda_{\ell,i} \big)_{i=1}^{m_\ell}$ is an $m_{\ell}\times m_{\ell}$ positive definite matrix.
Let $d$ be a positive integer and assume that $d\leq \max_{\ell\in[T]} m_{\ell}$, the IGW barycenter of the formulation \eqref{definition:barycenter} has the form: $\mu^* = \mathcal{N}(\mathbf{0},\Sigma_{\mu^*})$, where
\begin{align*}
    \Sigma_{\mu^*} = \diag\big(\lambda_{\mu^*,j} \big)_{j=1}^d,
\end{align*}
in which $\lambda_{\mu^*,j} = \sum_{\ell=1}^{T} \alpha_{\ell} \lambda_{\ell,j} \mathbf{1}_{j\leq m_{\ell}}$ for all $j\in[d]$.
\end{theorem} 
The proof of Theorem~\ref{theorem:barycenter} can be found in Appendix~\ref{appendix:barycenter}. Subsequently, we define the formulation of the barycenter of positive measures using the entropic IGW.
\begin{definition}[\textbf{Entropic inner product Gromov-Wasserstein barycenter}]
Let $\mu_{\ell}$ be a positive measure in $\mathbb{R}^{m_{\ell}}$ for any $\ell\in [T]$. Let $\alpha_1,\ldots,\alpha_T$ be positive constants such that $\sum_{\ell=1}^{T} \alpha_{\ell} = 1$. The entropic inner-product Gromov-Wasserstein barycenter of $\{\mu_{\ell},\alpha_{\ell}\}_{\ell=1}^{T}$ is defined as the probability distribution of the random vector $Y$ in $\mathbb{R}^d$ which is a solution of the following problem
\begin{align}
\argmin_{\substack{X_{\ell}\sim \mu_{\ell};Y \in \RR^d; Y\sim \mu;\\ (X_{\ell},Y),(X_{\ell}^{\prime},Y)\sim \pi_{\ell,y}}} &\sum_{\ell=1}^{T} \alpha_{\ell} \Big\{ \mathbb{E}_{\pi_{\ell,y}\otimes\pi_{\ell,y}}  \big[\langle X_{\ell},X_{\ell}^{\prime} \rangle - \langle Y, Y^{\prime} \rangle \big]^2\nonumber\\
   &+ \varepsilon \KL(\pi_{\ell,y} \| \mu_{\ell}\otimes \mu)\Big\}, \label{definition:entropic_barycenter}
\end{align}
where in the above minimum, $\{X_{\ell}^{\prime}\}$ and $\{X_{\ell}\}$ are i.i.d. random vectors and $Y$ and $Y^{\prime}$ are i.i.d. random vectors.
\end{definition}
From that definition, we have the following result for the entropic IGW barycenter of $T$ Gaussian distributions. 
\begin{theorem}[\textbf{Entropic inner product Gromov-Wasserstein barycenter has a closed form}]
\label{theorem:entropic_barycenter}
Let $\mu_{\ell}= \mathcal{N}(\mathbf{0}, \Sigma_{\ell})$ be Gaussian distribution in $\mathbb{R}^{m_\ell}$ for all $\ell\in[T]$, where $\Sigma_{\ell}=\diag\big(\lambda_{\ell,i} \big)_{i=1}^{m_\ell}$ is an $m_{\ell}\times m_{\ell}$ positive definite matrix.
Let $d$ be a positive integer such that $d\leq \max_{\ell\in[T]} m_{\ell}$. Define $d_{\ell} = \min \{d,m_{\ell} \}$. Let $\varepsilon$ be a positive constant satisfying 
\begin{align}
    \varepsilon \leq 2 \lambda_{\ell,j} \big\{A_j + (A_j^2 - \varepsilon B_j)^{\frac{1}{2}} \big\}, \label{condition:varepsilon_barycenter}
\end{align}
and $A_j^2 \geq \varepsilon B_j$ for any $\ell\in[T]$ and $j\in[d]$, where $A_j = \sum_{\ell=1}^{T} \alpha_{\ell} \lambda_{\ell,j} \mathbf{1}_{j\leq d_{\ell}}$ and $B_j = \sum_{\ell=1}^{T} \alpha_{\ell} \mathbf{1}_{j\leq d_{\ell}}$. Then, the barycenter of the formulation \eqref{definition:entropic_barycenter} admits a Gaussian solution which has the form: $\mu^* = \mathcal{N}(\mathbf{0},\Sigma_{\mu^*})$, in which $\Sigma_{\mu^*}=\diag(\lambda_{\mu^*,j})_{j=1}^d$ with  
\begin{align*}
    \lambda_{\mu^*,j} = \sum_{\ell=1}^{T} \alpha_{\ell} \lambda_{\ell,j} \mathbf{1}_{j\leq d_{\ell}} \left\{1 - \frac{\varepsilon}{\lambda_{\ell,j} \big[A_j + (A_j^2 - \varepsilon B_j)^{\frac{1}{2}} \big]} \right\}.
\end{align*}
\end{theorem}
The proof of Theorem~\ref{theorem:entropic_barycenter} is in Appendix~\ref{appendix:entropic_barycenter}.
\begin{remark} 
In the above theorem, we need a set of conditions for $\varepsilon$ that could be satisfied when $\varepsilon$ is small, which is often chosen in practice. When $\varepsilon$ is not small, then the readers could follow the guideline in  Lemma \ref{lemma:minimization_square_log} in Appendix~\ref{sec:auxiliary_results} to find the covariance matrix $\Sigma_{\mu^*}$. The proof of Theorem~\ref{theorem:entropic_barycenter} needs the below lemma to show the existence of a Gaussian minimizer.
\end{remark}
\begin{lemma}\label{lemma:KL_conditional_gaussian_min}
Let $Q_X=\mathcal{N}(\gamma_x,\Sigma_x)$ be a Gaussian distribution in $\br^m$ and $P_Y$ be a probability distribution with mean $\gamma_y\in\br^n$. Denote by $\Pi(Q_X,P_Y)$ the family of all probability distributions in $\mathbb{R}^{m+n}$ which have marginal distributions $Q_{X}$ and $P_Y$, and variance $\Sigma_{X,Y}=$\scalebox{0.75}{$\begin{pmatrix}\Sigma_x & \Sigma_{xy} \\ \Sigma_{xy}^{\top} & \Sigma_y\end{pmatrix}$}. Here, $\Sigma_x$ and $\Sigma_y$ are non-degenerate covariance matrices of size $m\times m$ and $n\times n$, respectively. Then
\begin{align*}
    \mathcal{N}\big([\gamma_x, \gamma_y]^{\top},\Sigma_{X,Y}\big) \in \argmin_{\substack{P_{X,Y}\in \Pi(Q_X,P_Y)\\ }} \mathsf{KL}\big(P_{X,Y} \|Q_X \otimes P_Y\big).
\end{align*}
\end{lemma}
The proof of Lemma \ref{lemma:KL_conditional_gaussian_min} is in Appendix~\ref{appendix:KL_conditional_gaussian_min}. 

\section{Empirical Studies}\label{sec:empirical_study}
In this section, we will use the derived closed-forms to inspect the behavior of algorithms solving entropic Gromov-Wasserstein problem, in particular those studied in \cite{peyre2016gromov}. We use the implementation of these algorithms in Python Optimal Transport library \citep{flamary2017pot}.
\paragraph{Projected mirror descent for $\mathsf{IGW}_\varepsilon$ (Figure \ref{figure:convergence}).} The algorithm in \cite{peyre2016gromov} reads
\begin{align*}
    C^{(t)} &= -D_\aA P^{(t)} D_\bB \\
    P^{(t + 1)} &= \min_{P \in \Pi(\aA, \bB)} \langle P, C^{(t)} \rangle - \varepsilon H(P),
\end{align*}
in which $D_\aA, D_\bB$ correspond to inner product cost matrices on supports of $\mu_N$ and $\nu_N$, $\Pi(\aA, \bB)$ is the set of discrete couplings $\{P \in \RR^{N \times N}: P \mathbf{1}_N = \aA, P^\top \mathbf{1}_N = \bB\}$, and $\langle P, C^{(t)} \rangle = \sum_{i, j \in [N]} P_{ij} C^{(t)}_{ij}$. Here, we set $m = 2, n = 3$ and sample $N$ (between $10$ and $2000$) data points from $\Nn(\zeros_2, \Sigma_\mu)$ and $\Nn(\zeros_3, \Sigma_\nu)$, in which $\Sigma_\mu$ and $\Sigma_\nu$ are diagonal matrices whose diagonal values are uniformly samples from the interval $[0, 2]$. The regularization parameter $\varepsilon$ is chosen from $\{0.5, 1, 5\}$. We plot means and one standard-deviation areas over 20 independent runs for each $N$. As expected, with more samples we can approximate $\mathsf{IGW}_\varepsilon$ more accurately. However, since this algorithm only converges to a stationary point, there might still be a gap between the converged value and the optimal value.
\begin{figure}[!t]
    \centering
    \includegraphics[width=0.49\textwidth]{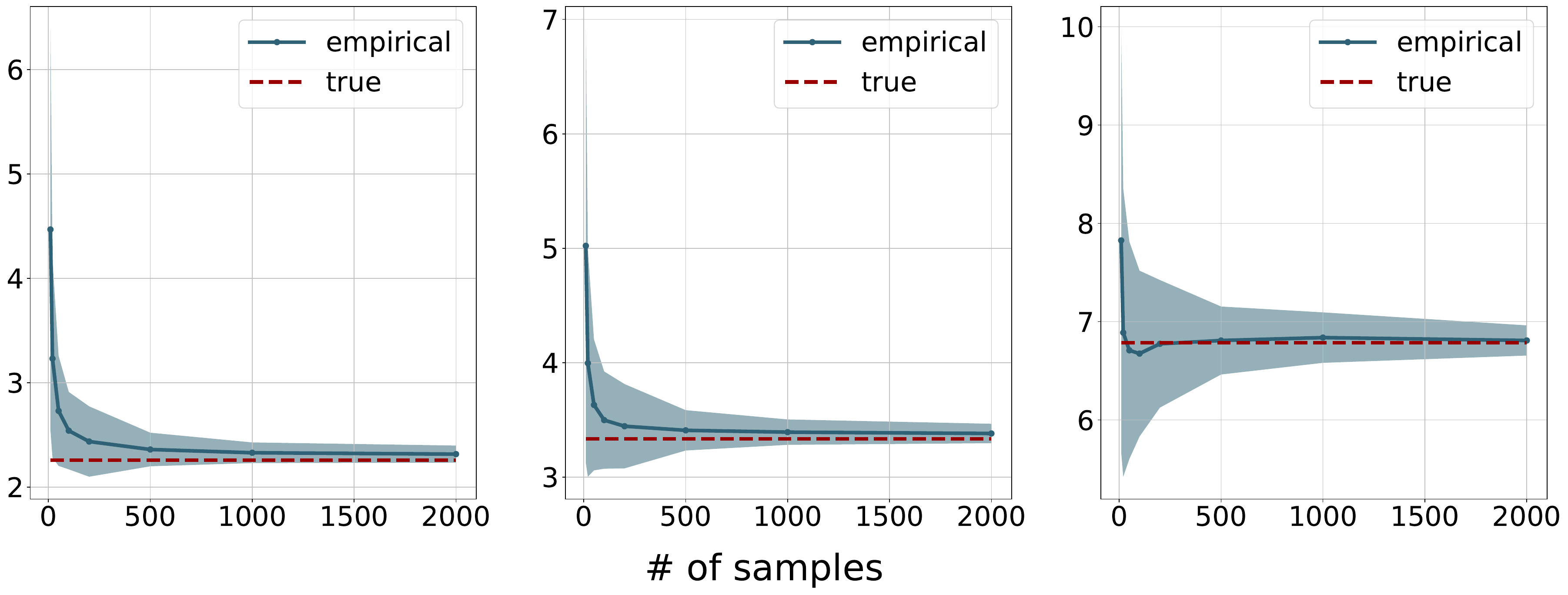}  
    \caption{Empirical convergence for \cite{peyre2016gromov} in computing $\mathsf{IGW}_\varepsilon$. From left to right: $\varepsilon=0.5$, $\varepsilon=1$ and $\varepsilon=5$. The red dashed lines correspond to the theoretical optimal values from Theorem \ref{theorem:entropic_gw:inner_product:closed_form}, while the blue lines and shaded regions are the means and standard variations of the objective values computed according to \cite{peyre2016gromov}, respectively.}
    \label{figure:convergence}
\end{figure}

\paragraph{Visualizing optimal transportation plans of for $n = m = 1$ (Figure \ref{figure:transportation_plan}).} Now we take a look at the transportation plans computed by the above algorithm for different values of $\varepsilon$. Specifically, we create $1000$-bin histograms of $\Nn(0, \alpha)$ and $\Nn(0, \beta)$ (which will be $\mu_N$ and $\nu_N$ with $N = 1000$) in which $\alpha = 2, \beta = 10$, and set $\varepsilon \in \{0.1, 1, 20, 40, 80, 100\}$ (values in this set are chosen based on $\alpha \beta$). The algorithm is run till convergence (with tolerance $10^{-9}$) and the plans are plotted in Figure \ref{figure:transportation_plan}. It is apparent that the plans returned by the projected mirror descent algorithm are Gaussians and resemble our theoretical ones (except for the case $\varepsilon = 1$, but noting that when $n = m = 1$, if $\Nn(\zeros, (\begin{smallmatrix} \alpha & \gamma \\ \gamma & \beta \end{smallmatrix}))$ is an optimal plan for $\mathsf{IGW}_\varepsilon$, so is $\Nn(\zeros, (\begin{smallmatrix} \alpha & -\gamma \\ -\gamma & \beta \end{smallmatrix})$).
\begin{figure}[!t]
    \centering
    \includegraphics[width=0.49\textwidth]{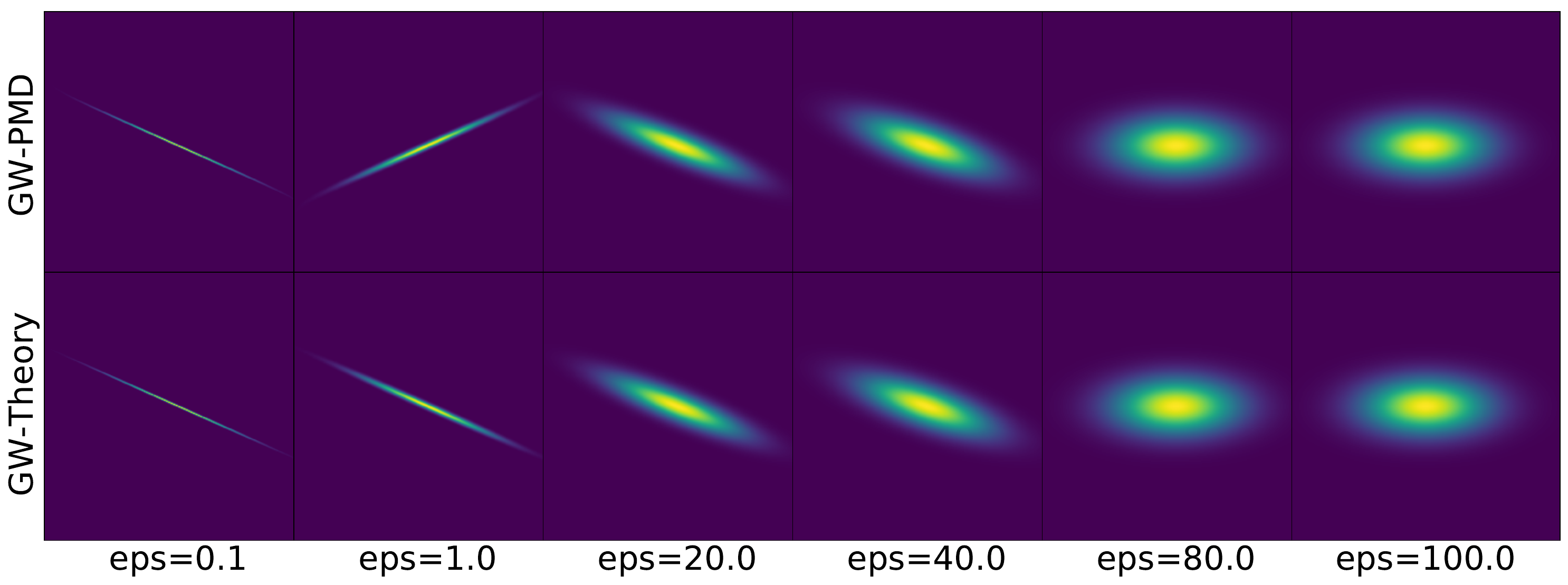}  
    \caption{Empirical transportation plans between $1$-dimensional Gaussians $\Nn(0, 2)$ and $\Nn(0, 10)$. From left to right:
    $\varepsilon = 0.1, \varepsilon = 1, \varepsilon = 20, \varepsilon = 40, \varepsilon = 80, \varepsilon = 100$. The first rows present plans returned by \cite{peyre2016gromov} and the second row corresponds to theoretical plans in Theorem \ref{theorem:entropic_gw:inner_product:closed_form}.}
    \label{figure:transportation_plan}
\end{figure}

\paragraph{Alternating minimization for the barycenter problem (Figure \ref{figure:barycenter}).} Given $T$ discrete measures $\{\mu^{(i)}_N\}_{i = 1}^T$, a set of weights $\{\alpha_i\}_{i = 1}^T$, cost matrices $C^{(i)}$ corresponding to measure $\mu^{(i)}_N$ and a probability vector $p$, \cite{peyre2016gromov} proposes an alternating minimization scheme to find a $d$-dimensional barycenter $\mu$ of $\{\mu^{(i)}_N\}_{i = 1}^T$ with probability mass $p$, i.e., finding the cost matrix $C$ on the support of this barycenter. In our case of inner product metric, $C$ is a Gram matrix, and we can recover the supports of $\mu$ via the Cholesky decomposition $C = L L^\top$ where $B \in \RR^{N \times N}$ and choosing the first $d$ columns of $L$. In our setting, for $i \in [T]$, we choose $\mu^{(i)}_N = \sum_{j = 1}^N \frac{1}{N} \delta_{X_j}$ where $X_j \sim \Nn(0, \Sigma^{(i)})$, $C^{(i)}_{kl} = \langle X_k, X_l \rangle$ for $k, l \in [N]$ and $p = \mathbf{1}_N / N$. Specifically, we consider $N = 1000, T = 3, \Sigma^{(1)} = 5, \Sigma^{(2)} = \mathrm{diag}(10, 1), \Sigma^{(3)} = \mathrm{diag}(2, 1, 1)$, $\alpha = (0.3, 0.6, 0.1)$ and $d = 2$. The regularization parameter $\varepsilon$ is set to $0.1$, which satisfies all the constraints in Theorem \ref{theorem:entropic_barycenter}. After finding $N$ support points for $\mu$ by the discussed method, we compute the sample covariance $\hat{\Sigma}$, perform orthogonal diagonalization, and apply the corresponding transformation to the support. We then compute its sample mean and sample covariance, and compare the Gaussian with these statistics with our theoretical optimal plan (see Figure \ref{figure:barycenter}). It is worth noting that any orthogonally-transformed version of the found barycenter is also a barycenter of $\{\mu^{(i)}_N\}_{i = 1}^T$ under Gromov-Wasserstein setting.

\begin{figure}[!t]
    \centering
    \includegraphics[width=0.49\textwidth]{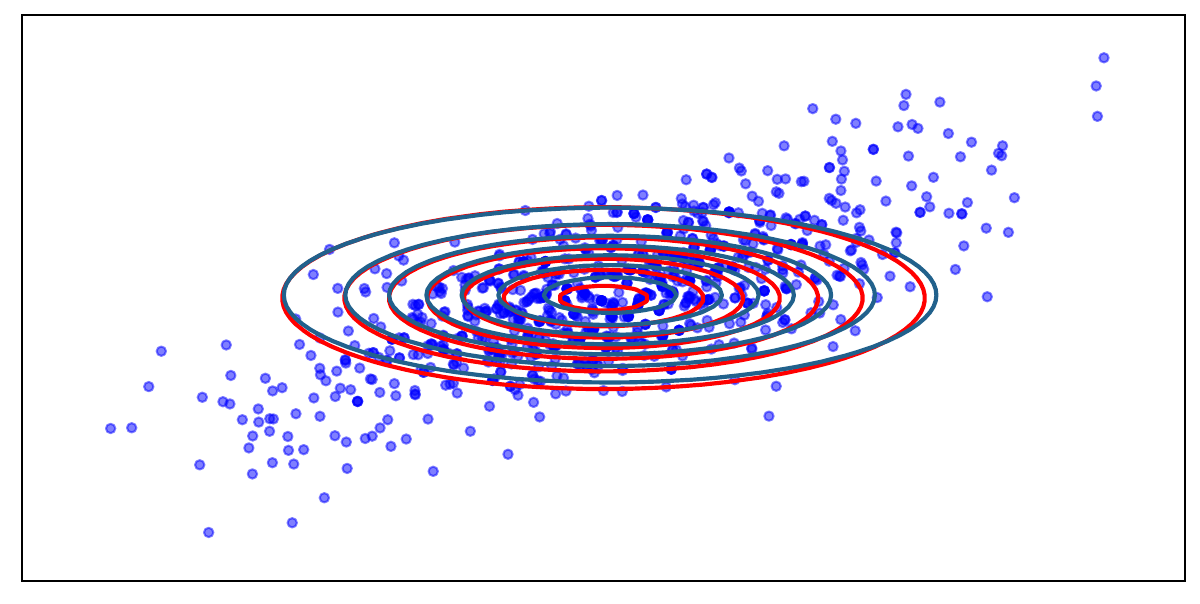}
    \caption{Visualizing the barycenter. The blue points correspond to one possible set of support for the computed barycenter by \cite{peyre2016gromov} (see discussion). The blue contours are the rotated version of the Gaussian fitted to this support set (for comparison). The red contours represent the theoretical barycenter whose form is presented in Theorem \ref{theorem:entropic_barycenter}.}
    \label{figure:barycenter}
\end{figure}

\section{Conclusion}
\label{sec:conclusion}
In this paper, we provide a comprehensive study of the entropic (unbalanced) inner product Gromov-Wasserstein (IGW) between (unbalanced) Gaussian distributions. We demonstrate that the optimal transportation plan is (unbalanced) Gaussian distribution. Based on that result and a novel application of von Neumann's trace inequality, we derive the closed-form expression for the entropic (unbalanced) IGW between these distributions. Finally, we also consider the (entropic) Gromov-Wasserstein barycenter problem of multiple Gaussian measures. We prove that the barycenter problem 
admits a Gaussian minimizer  and obtain the closed-form expression for the covariance matrix of the barycenter when the entropic regularization parameter is small. 

\bibliography{example_paper}
\bibliographystyle{icml2022}

\newpage
\appendix
\onecolumn
\begin{center}
\textbf{\Large{Supplementary Materials for ``Entropic Gromov-Wasserstein between Gaussian Distributions''}}
\end{center}
In this supplement, we first provide proofs of remaining lemmas and theorems in Appendix~\ref{sec:proofs_remaining_lemmas}. Then, we present how to derive a closed-form formulation for $\Upsilon^{*}$ in Theorem~\ref{theorem:UGW} in Appendix~\ref{sec:upsilon_calculation}. Auxiliary results are presented in Appendix~\ref{sec:auxiliary_results} while another proof of Lemma~\ref{lemma:order_eigenvalues} is in Appendix~\ref{sec:another_proof}. We highlight the issue of obtaining closed-form expression for entropic inner product Gromov-Wasserstein between non-zero means Gaussian distributions in Appendix~\ref{sec:nonzero_means}. Finally, we revisit the entropic optimal transport between (unbalanced) Gaussian distributions in Appendix~\ref{sec:revisit_entropic_Gaussian} and provide a simpler proof than that in~\citep{janati2020entropic} to derive the closed-form expression for that problem.
\section{Proofs of remaining results}
\label{sec:proofs_remaining_lemmas}
This appendix is devoted to provide the proofs of lemmas and theorems presented in the paper.
\subsection{Proof of Lemma~\ref{lemma:IGW_rotation}}
\label{appendix:IGW_rotation}
Firstly, note that when $\pi^\prime\in\Pi(T_m\#\mu,T_n\#\nu)$, there exists $\pi\in\Pi(\mu,\nu)$ such that $(T_m,T_n)\#\pi=\pi^\prime$. Therefore, we have
\begin{align}
\mathbb{E}_{\pi^\prime\otimes\pi^\prime}[\langle X,X^\prime\rangle-\langle Y,Y^\prime\rangle]^2&=\mathbb{E}_{(T_m,T_n)\#\pi\otimes(T_m,T_n)\#\pi}[\langle T_mX,T_mX^\prime\rangle-\langle T_nY,T_nY^\prime\rangle]^2\nonumber\\
&=\mathbb{E}_{\pi\otimes\pi}[\langle X,X^\prime\rangle-\langle Y,Y^\prime\rangle]^2.\label{eq:expectation_unchanged}
\end{align}
Furthermore, since $d\pi^\prime(x,y)=d\pi(T_m^{-1}(x),T_n^{-1}(y))$, $d(T_m\#\mu)(x)=d\mu(T_m^{-1}(x))$ and $d(T_n\#\nu)(y)=d\nu(T_n^{-1}(y))$, by changing of variables $x\mapsto T_m(u)$ and $y\mapsto T_n(v)$, we get 
\begin{align}
    \mathrm{KL}(\pi^\prime\|(T_m\#\mu)\otimes(T_n\#\nu))&=\int\log\left(\dfrac{d\pi'(x,y)}{d(T_m\#\mu)(x)d(T_n\#\nu)(y)}\right)\nonumber\\
    &=\int\log\left(\dfrac{d\pi'(T_m(u),T_n(v)))}{d(T_m\#\nu)(T_m(u))d(T_n\#\nu)(T_n(v))}\right)|\det(T'_m(u))\det(T'_n(v))|\nonumber\\
    &= \int\log\left(\dfrac{d\pi(u,v)}{d\mu(u)d\nu(v)}\right)\nonumber\\
    &=\mathrm{KL}(\pi\|\mu\otimes\nu))\label{eq:entropy_unchanged},
\end{align}
with a note that $|\det(T_m'(u)|=|\det(O_m)|=1$ and $|\det(T_n'(v))|=|\det(O_n)|=1$ since $O_m$ and $O_n$ are orthogonal matrices. Putting the results in equations \eqref{eq:expectation_unchanged} and \eqref{eq:entropy_unchanged} together, we obtain the conclusion.
\subsection{Proof of Theorem~\ref{theorem:entropic_gw:inner_product:closed_form}}
\label{appendix:entropic_gw:inner_product:closed_form}
Firstly, we will show that
\begin{align}
\label{eq:first_equality_balanced}
\E_{\pi\otimes\pi} \Big\{ \big[ \langle X,X^{\prime}\rangle - \langle Y,Y^{\prime}\rangle \big]^2 \Big\}  = \tr(\Sigma_{\mu}^2) + \tr(\Sigma_{\nu}^2)- 2\tr\big(K_{\mu\nu}^{\top}K_{\mu\nu}\big). 
\end{align}
Indeed, we have
\begin{align*}
    \E_{\pi\otimes\pi} \Big\{ \big[ \langle X,X^{\prime}\rangle - \langle Y,Y^{\prime}\rangle \big]^2 \Big\}  = \E_{\pi\otimes\pi} \langle X,X^{\prime}\rangle^2 + \E_{\pi\otimes\pi} \langle Y, Y^{\prime}\rangle^2 - 2 \E_{\pi\otimes\pi} \langle X, X^{\prime}\rangle \langle Y, Y^{\prime}\rangle.
\end{align*}
As $\pi$ is a coupling of two zero-mean distributions $\mu$ and $\nu$, $\pi$ also has mean zero. Combining this result with the independence between $X$ and $X^{\prime}$ and $Y$ and $Y^{\prime}$ leads to
\begin{align*}
    \E_{\pi\otimes\pi} \langle X,X^{\prime}\rangle^2 = \tr(\Sigma_{\mu}^2) \quad \text{ and } \quad \E_{\pi\otimes\pi} \langle Y,Y^{\prime}\rangle^2 = \tr(\Sigma_{\nu}^2).
\end{align*}
Meanwhile, we have
\begin{align*}
    \E_{\pi\otimes\pi} [\langle X, X^{\prime}\rangle \langle Y, Y^{\prime}\rangle] = \E_{\pi\otimes\pi} \Big[\sum_{i=1}^{m} \sum_{j=1}^n X_iX^{\prime}_{i}Y_jY^{\prime}_{j}\Big]= \sum_{i=1}^m \sum_{j=1}^n \E_{\pi\otimes\pi}[X_i Y_j]\E_{\pi}[X^\prime_i Y^\prime_j]= \tr\big(K_{\mu\nu}^{\top}K_{\mu\nu}\big).
\end{align*}
Putting the above results together, we obtain the desired equality~\eqref{eq:first_equality_balanced}. It indicates that we can rewrite the formulation of $\mathsf{IGW}_\varepsilon(\mu,\nu)$ as follows:
\begin{align}
    \label{eq:IGW_equality}
    \mathsf{IGW}_\varepsilon&(\mu,\nu) =  \tr(\Sigma_{\mu}^2) + \tr(\Sigma_{\nu}^2) +\min_{\pi\in\Pi(\mu,\nu)} \Big\{\varepsilon \KL(\pi\|\mu \otimes \nu)- 2 \tr\big(K_{\mu\nu}^{\top}K_{\mu\nu}\big)\Big\}.
\end{align}

To solve the minimization problem \eqref{eq:IGW_equality}, we firstly fix the matrix $K_{\mu\nu}$, therefore, the covariance matrix of $\pi$ is also fixed due to equation~\eqref{eq:plan_structure_balanced}. By Lemma~\ref{lemma:min_KL_Gaussian}, the optimal transport plan of this problem needs to be a Gaussian distribution.
Thus, according to Lemma~\ref{lemma:KL_calculate} in Appendix~\ref{sec:auxiliary_results} and part (b) of Lemma \ref{lemma:entropic_gw:inner_product:closed_form:supporting}, the entropic term in the objective function reads
\begin{align*}
   \varepsilon \KL(\pi\|\mu \otimes \nu)& = \frac{1}{2}\varepsilon\Big\{\tr\big(\Sigma_{\pi} \Sigma^{-1}_{\mu\otimes \nu}\big) - (m+n) +\log\Big( \frac{\det(\Sigma_{\mu\otimes \nu})}{\det(\Sigma_{\pi})}\Big) \Big\} \\
   & = -\frac{\varepsilon}{2}  \sum_{i=1}^{n} \log(1-\kappa_{\mu\nu,i}),
\end{align*}
where $([\kappa_{\mu \nu,i}]^\frac{1}{2})_{i=1}^n$ are singular values (in descending order) of matrix $\Sigma^{-\frac{1}{2}}_\mu K_{\mu\nu}\Sigma^{-\frac{1}{2}}_\nu$.
By applying part (c) of Lemma~\ref{lemma:entropic_gw:inner_product:closed_form:supporting}, the optimal value of the term $\tr\big(K_{\mu\nu}^{\top}K_{\mu\nu}\big)$ is achieved at $K_{\mu\nu}=\Sigma_{\mu}^{\frac{1}{2}}\diag([\kappa_{\mu\nu,k}]^{\frac{1}{2}})_{k=1}^n\Sigma_{\nu}^{\frac{1}{2}}$. Therefore, the optimization problem \eqref{eq:IGW_equality} reduces to
\begin{align*}
      \mathsf{IGW}_\varepsilon(\mu,\nu)=\tr(\Sigma_{\mu}^2) + \tr(\Sigma_{\nu}^2&)  - \max_{\kappa_{\mu\nu,k}\in[0,1)}\Big\{2 \sum_{k=1}^{n} \lambda_{\mu,k}\lambda_{\nu,k}\kappa_{\mu\nu,k}+ \frac{\varepsilon}{2} \sum_{k=1}^n\log(1-\kappa_{\mu\nu,k})\Big\}.
\end{align*}
The function $f(x) = ax + \frac{\varepsilon}{4} \log(1-x)$ for $a>0$ determined in the interval $[0,1)$  attains its maximum at $x = \big[ 1-\frac{\varepsilon}{4a} \big]^+ $. Thus, $\kappa_{\mu\nu,k}^{*} = \big[1 - \frac{\varepsilon}{4\lambda_{\mu,k}\lambda_{\nu,k}} \big]^+$ for all $k\in[n]$, and
\begin{align*}
    \mathsf{IGW}_\varepsilon(\mu,\nu)=\tr(\Sigma_{\mu}^2) + \tr(\Sigma_{\nu}^2) - 2 \sum_{k=1}^{n} \Big(\lambda_{\mu,k} \lambda_{\nu,k} - \frac{\varepsilon}{4}\Big)^+ + \frac{\varepsilon}{2} \sum_{k=1}^{n}\Big[\log(\lambda_{\mu,k} \lambda_{\nu,k}) - \log\Big(\frac{\varepsilon}{4}\Big)  \Big]^+.
\end{align*}
Additionally, we have
\begin{align*}
K^*_{\mu \nu} = \diag\Big( \Big\{\lambda_{\mu,k} \lambda_{\nu,k}  \Big[1 - \frac{\varepsilon}{4\lambda_{\mu,k}\lambda_{\nu,k}} \Big]^+\Big\}^{\frac{1}{2}}\Big)_{k=1}^n.
\end{align*}
As a consequence, the proof is completed.
\paragraph{When $\Sigma_\mu$ and $\Sigma_\nu$ are not diagonal:} The value of $\mathsf{IGW}_\varepsilon(\mu,\nu)$ still remains according to Lemma~\ref{lemma:IGW_rotation} but the optimal transport plan does not. More specifically, the equality conditions mentioned in part (c) of Lemma~\ref{lemma:entropic_gw:inner_product:closed_form:supporting} 
will be no longer valid. Instead, let $P_{\mu},D_{\mu}$ and $P_{\nu},D_{\nu}$ are the orthogonal diagonalizations of $\Sigma_\mu(=P_\mu D_\mu P^{\top}_\mu)$ and $\Sigma_\nu(=P_\nu D_\nu P^{\top}_\nu)$, respectively. By using the same approach as in Lemma~\ref{lemma:entropic_gw:inner_product:closed_form:supporting}, let $U_{\mu\nu}\Lambda^{\frac{1}{2}}_{\mu\nu}V^{\top}_{\mu\nu}$ be the SVD of matrix $\Sigma^{-\frac{1}{2}}_\mu K_{\mu\nu}\Sigma^{-\frac{1}{2}}_\nu$ where $\Lambda^{\frac{1}{2}}_{\mu\nu}=\diag([\kappa_{\mu\nu,k}]^{\frac{1}{2}})_{k=1}^n$, we have $K_{\mu\nu}=\Sigma^{\frac{1}{2}}_{\mu}U_{\mu\nu}\Lambda^{\frac{1}{2}}_{\mu\nu}V^{\top}_{\mu\nu}\Sigma^{\frac{1}{2}}_{\nu}$. Therefore,
\begin{align*}
    \tr(K^{\top}_{\mu\nu}K_{\mu\nu})&=\tr\big(\Sigma_{\nu}^{\frac{1}{2}} V_{\mu\nu} (\Lambda_{\mu\nu}^{\frac{1}{2}})^{\top} U_{\mu\nu}^{\top} \Sigma_{\mu} U_{\mu\nu} \Lambda_{\mu\nu}^{\frac{1}{2}} V_{\mu\nu}^{\top} \Sigma_{\nu}^{\frac{1}{2}} \big)\\
    &=\tr\big(V_{\mu\nu} (\Lambda_{\mu\nu}^{\frac{1}{2}})^{\top} U_{\mu\nu}^{\top} \Sigma_{\mu} U_{\mu\nu} \Lambda_{\mu\nu}^{\frac{1}{2}} V_{\mu\nu}^{\top} \Sigma_{\nu}\big)\\
    &=\tr\big(V_{\mu\nu} (\Lambda_{\mu\nu}^{\frac{1}{2}})^{\top} U_{\mu\nu}^{\top} P_{\mu}D_{\mu}P^{\top}_{\mu} U_{\mu\nu} \Lambda_{\mu\nu}^{\frac{1}{2}} V_{\mu\nu}^{\top}P_{\nu}D_{\nu}P^{\top}_{\nu} \big)\\
    &=\tr\big([P^{\top}_{\nu}V_{\mu\nu} (\Lambda_{\mu\nu}^{\frac{1}{2}})^{\top}] [U_{\mu\nu}^{\top} P_{\mu}D_{\mu}][P^{\top}_{\mu} U_{\mu\nu} \Lambda_{\mu\nu}^{\frac{1}{2}}] [V_{\mu\nu}^{\top}P_{\nu}D_{\nu}] \big)\\
    &\leq \sum_{i=1}^{n} \lambda_{\mu,i} \lambda_{\nu,i}\kappa_{\mu\nu,i}.
\end{align*}
The equality occurs when $U_{\mu\nu}=P_{\mu}$ and $V_{\mu\nu}=P_{\nu}$. Hence, the covariance matrix $K^*_{\mu\nu}$ in this case turns into
\begin{align*}
    K^{*,\text{new}}_{\mu\nu}=\Sigma^{\frac{1}{2}}_{\mu}P_{\mu}\Lambda^{\frac{1}{2}}_{\mu\nu}P^{\top}_{\nu}\Sigma^{\frac{1}{2}}_{\nu}=P_{\mu}D^{\frac{1}{2}}_{\mu}\Lambda^{\frac{1}{2}}_{\mu\nu}D^{\frac{1}{2}}_{\nu}P_{\nu}=P_{\mu}K^*_{\mu\nu}P^{\top}_{\nu}.
\end{align*}
\subsection{Proof of Lemma~\ref{lemma:min_KL_Gaussian}}
\label{appendix:min_KL_Gaussian}
Let $q_v$ and $p_u$ be the probability density functions of distributions $Q_v$ and $P_u$, respectively, whereas $q_u$ be the probability density function of the Gaussian distribution $Q_u$ in $\br^d$ which has mean $u$ and variance $\Sigma_u$. Then, we have
\begin{align*}
    \KL(P_{u}\|Q_{v}) &= \int_{\br^d}p_u \log\frac{p_{u}}{q_{v}}  \dd x
    = \int_{\br^d} p_{u} \log\frac{p_{u}}{q_u} \dd x + \int_{\br^d} p_u \log \frac{q_u}{q_{v}} \dd x \\
    &= \KL(P_u\|Q_u) + \int_{\br^d} p_u \log\frac{q_u}{q_v} \dd x \\
    &\geq  \int_{\br^d} p_u \log\frac{q_u}{q_v} \dd x.
\end{align*}
The equality happens when $P_u = Q_u$. Now we prove the lower bound is constant. 

Note that
\begin{align}
    \int_x p_u \log q_v &= \E_{p_u}\big[\log(q_v)\big] = \E_{p_u} \Big[ -\frac{1}{2}  \log \big[(2 \pi)^d \det (\Sigma_v) \big] -\frac{1}{2} (x - v)^\top \Sigma^{-1}_v (x - v) \big) \Big] \nonumber\\
    &= -\frac{1}{2}\E_{p_u} \Big[ (x-v)^{\top}\Sigma_v^{-1}(x-v)\Big] + \mathsf{const} \nonumber\\
    &=-\frac{1}{2}\E_{p_u}\Big[(x-u + u-v)^{\top} \Sigma_v^{-1}(x-u+u-v) \Big] + \mathsf{const} \nonumber\\
    &= -\frac{1}{2}\E_{p_u}\Big[(x-u)^{\top}\Sigma_v^{-1}(x-u) + 2(x-u)^{\top}\Sigma_{v}^{-1}(u-v) + (u-v)^{\top}\Sigma_v^{-1}(u-v) \Big] + \mathsf{const}\nonumber\\
    &= -\frac{1}{2}\E_{p_u}\Big[(x-u)^{\top}\Sigma_v^{-1}(x-u)\Big] + \mathsf{const}\nonumber\\
    &= -\frac{1}{2}\E_{p_u}\Big[(x-u)^{\top}\Sigma_v^{-\frac{1}{2}}\Sigma_v^{-\frac{1}{2}}(x-u)\Big] + \mathsf{const}\nonumber\\
    &=-\frac{1}{2}\E_{p_u}\Big[\tr\big(\Sigma_v^{-\frac{1}{2}}(x-u)(x-u)^{\top}\Sigma_v^{-\frac{1}{2}} \big) \Big] + \mathsf{const}\nonumber\\
    & =-\frac{1}{2} \tr\big(\Sigma_v^{-\frac{1}{2}}\Sigma_u\Sigma_v^{-\frac{1}{2}} \big) + \mathsf{const}\label{eq:expectation_log}.
\end{align}
Similarly, we have $\E_{p_u}\big[\log(q_u) \big]$ is a constant. Hence, the proof is completed.

\subsection{Proof of Lemma~\ref{lemma:entropic_gw:inner_product:closed_form:supporting}}
\label{appendix:entropic_gw:inner_product:closed_form:supporting}
\textbf{(a)} We start with decomposing the matrix $K_{\mu\nu}$ as $K_{\mu\nu} = \Sigma_{\mu}^{\frac{1}{2}}U_{\mu\nu}\Lambda_{\mu\nu}^{\frac{1}{2}}V_{\mu\nu}^{\top}\Sigma_{\nu}^{\frac{1}{2}}$. It follows from the fact $\Sigma_{\pi}$ is non-negative definite that
\begin{align*}
    \Sigma_{\nu} &\succeq K_{\mu\nu}^{\top}\Sigma_{\mu}^{-1} K_{\mu\nu} \\
    \Leftrightarrow \Sigma_{\nu}  &\succeq \Sigma_{\nu}^{\frac{1}{2}} V_{\mu\nu} (\Lambda_{\mu\nu}^{\frac{1}{2}})^{\top} U_{\mu\nu}^{\top} U_{\mu\nu} \Lambda_{\mu\nu}^{\frac{1}{2}} V_{\mu\nu}^{\top}\Sigma_{\nu}^{\frac{1}{2}} \\
    \Leftrightarrow \Sigma_{\nu} &\succeq \Sigma_{\nu}^{\frac{1}{2}}V_{\mu\nu}\Lambda_{\mu\nu} V_{\mu\nu}^{\top} \Sigma_{\nu}^{\frac{1}{2}}\\
    \Leftrightarrow \id_n & \succeq V_{\mu\nu}\Lambda_{\mu\nu} V_{\mu\nu}^{\top},
\end{align*}
where $\Lambda_{\mu\nu}=\diag\big(\kappa_{\mu\nu,k}\big)_{k=1}^n$ is an $n\times n$ matrix.
Since $V_{\mu\nu}$ is an unitary matrix, all eigenvalues of matrix $\Lambda_{\mu\nu}$, which are $(\kappa_{\mu\nu,k})_{k=1}^n$, belong to the interval $[0,1]$.

\noindent
\textbf{(b)} By the definition of matrix $\Sigma_\pi$, we have
\begin{align*}
    \det(\Sigma_\pi) &= \det(\Sigma_{\mu}) \det\big(\Sigma_{\nu} - K_{\mu\nu}^{\top}\Sigma_{\mu}^{-1} K_{\mu\nu} \big)  \\
    &= \det(\Sigma_{\mu}) \det(\Sigma_{\nu}) \det\Big(\id_n - V_{\mu\nu}\Lambda_{\mu\nu} V_{\mu\nu}^{\top} \Big) \\
    &= \det(\Sigma_{\mu}) \det(\Sigma_{\nu}) \prod_{k=1}^{n} (1-\kappa_{\mu\nu,k}).
\end{align*}
The third equality results from the fact that all eigenvalues of matrix $\id_n - V_{\mu\nu}\Lambda_{\mu\nu} V_{\mu\nu}^{\top} $ are $1- \kappa_{\mu\nu,k}$ for $k\in[n]$.

\noindent
\textbf{(c)} Using the decomposition of matrix $K_{\mu\nu}$ in part (a), $\tr\big(K_{\mu\nu}^{\top}K_{\mu\nu}\big)$ can be rewritten as
\begin{align*}
    \tr\big(\Sigma_{\nu}^{\frac{1}{2}} V_{\mu\nu} (\Lambda_{\mu\nu}^{\frac{1}{2}})^{\top} U_{\mu\nu}^{\top} \Sigma_{\mu} U_{\mu\nu} \Lambda_{\mu\nu}^{\frac{1}{2}} V_{\mu\nu}^{\top} \Sigma_{\nu}^{\frac{1}{2}} \big)
    = \tr\Big(\big[V_{\mu\nu} (\Lambda_{\mu\nu}^{\frac{1}{2}})^{\top}\big]\big[ U_{\mu\nu}^{\top} \Sigma_{\mu}\big] \big[ U_{\mu\nu} \Lambda_{\mu\nu}^{\frac{1}{2}}\big] \big[ V_{\mu\nu}^{\top} \Sigma_{\nu}\big] \Big).
\end{align*}
Applying the generalization of the von Neumann's inequality \citep{kristof1969neumann,horn_johnson_1991} for singular values with a note that the equality happens when $U_{\mu \nu}$ and $V_{\mu\nu}^{\top}$ are identity matrices, we obtain
\begin{align*}
    \tr\big(K_{\mu\nu}^{\top}K_{\mu\nu}\big) \leq \sum_{i=1}^{n} \lambda_{\mu,i} \lambda_{\nu,i}\kappa_{\mu\nu,i}.
\end{align*}
Hence, we obtain the conclusion of this lemma.
\subsection{Proof of Lemma~\ref{lemma:lower_bound_KL}}
\label{appendix:lower_bound_KL}
By the formula for KL divergence between Gaussians in Lemma~\ref{lemma:KL_calculate} in Appendix~\ref{sec:auxiliary_results} and applying the von Neumann's trace inequality \citep{kristof1969neumann,horn_johnson_1991}, we have
 \begin{align*}
     \KL(\overline{\pi}_x\|\overline{\mu}) &= \frac{1}{2}\left\{\tr\big(\Sigma_x\Sigma_{\mu}^{-1}  \big) - m + \log\Big(\frac{\det(\Sigma_{\mu})}{\det(\Sigma_x)}\Big)\right\} \\
     &\geq \frac{1}{2}  \sum_{i=1}^m \left\{ \frac{\lambda_{x,i}}{\lambda_{\mu,i}} -\log\Big(\frac{\lambda_{x,i}}{\lambda_{\mu,i}}\Big)-1  \right\}.
 \end{align*}
 Similarly, we find that
 \begin{align*}
     \KL(\overline{\pi}_y\|\overline{\nu}) \geq \frac{1}{2} \sum_{j=1}^n\left\{ \frac{\lambda_{y,j}}{\lambda_{\nu,j}}  - \log\Big(\frac{\lambda_{y,j}}{\lambda_{\nu,j}}\Big)-1\right\}.
 \end{align*}
 For the last equality, $\KL(\overline{\pi}\|\overline{\mu}\otimes \overline{\nu})$ is equal to 
 \begin{align*}
     & \frac{1}{2}\left\{ \tr\big(\Sigma_{\pi} \Sigma^{-1}_{\mu\otimes \nu} \big)- (m+n) - \log \Big(\frac{\det(\Sigma_{\pi})}{\det(\Sigma_{\mu\otimes \nu})} \Big) \right\} \\
     =~& \frac{1}{2} \Big\{\tr\big(\Sigma_x \Sigma_{\mu}^{-1}\big) + \tr\big(\Sigma_y\Sigma_{\nu}^{-1} \big) - (m+n)- \log \Big(\frac{\det(\Sigma_x)\det(\Sigma_y)}{\det(\Sigma_{\mu}) \det(\Sigma_{\nu}) } \Big) - \sum_{k=1}^{n}\log(1-\kappa_{xy,k})\Big\} \\
     =~& \KL(\overline{\pi}_x\|\overline{\mu}) + \KL(\overline{\pi}_y\|\overline{\nu}) - \frac{1}{2}\sum_{k=1}^n \log(1- \kappa_{xy,k}).
 \end{align*}
The first equality results from the form of the matrix $\Sigma_{\mu\otimes\nu}=\begin{pmatrix}
\Sigma_\mu & \zeros_{m\times n}\\
\zeros_{n\times m} & \Sigma_\nu
\end{pmatrix}$.
\subsection{Proof of Theorem~\ref{theorem:UGW}}
\label{appendix:UGW}
For the sake of computing, we will firstly show that all quadratic KL terms in the objective function~\eqref{eq:UIGW_formulation} can be transformed to their normal versions as follows:
\begin{align}
\KLD(\pi_x \| \mu)&= 2m_{\pi}^2 \KL(\bar{\pi}_x \| \bar{\mu}) + \KL(m_{\pi}^2\|m_{\mu}^2),\label{eq:kl_pi_x}\\
\KLD(\pi_y\|\nu) &= 2m_{\pi}^2 \KL(\bar{\pi}_y \| \bar{\nu}) + \KL(m_{\pi}^2\|m_{\nu}^2), \label{eq:kl_pi_y}\\
\KLD(\pi\|\mu\otimes \nu)&= 2m_{\pi}^2 \KL(\bar{\pi} \|\overline{\mu} \otimes \overline{\nu}) + \KL\big(m_{\pi}^2 \|m_{\mu}^2m_{\nu}^2 \big)\label{eq:kl_pi}.
\end{align}
By applying parts (ii) and (i) of Lemma~\ref{lemma:kl_divergence} in that order, we get the equation~\eqref{eq:kl_pi_x}
\begin{align*}
    \KLD(\pi_x\|\mu) &= \KLD(m_{\pi}\bar{\pi}_x\|m_\mu\bar{\mu})\\
    &= m^2_{\pi}\KLD(\bar{\pi}_x\|\bar{\mu}) + m^2_{\pi}\log\left(\frac{m^2_\pi}{m^2_\mu}\right)+(m^2_\mu-m^2_\pi)\\
    &= 2m^2_\pi\KL(\bar{\pi}_x\|\bar{\mu})+\KL(m^2_\pi\|m^2_\mu).
\end{align*}
The equations~\eqref{eq:kl_pi_y} and \eqref{eq:kl_pi} are obtained similarly. 
As a result, the objective function in equation \eqref{eq:UIGW_formulation} can be rewritten as
\begin{align}
 &m_{\pi}^2 \E_{\bar{\pi}\otimes\bar{\pi}} \Big\{ \big[ \langle X,X^{\prime}\rangle - \langle Y,Y^{\prime}\rangle \big]^2 \Big\}+  \varepsilon\Big\{ 2m_{\pi}^2 \KL(\bar{\pi} \|\overline{\mu} \otimes \overline{\nu}) + \KL\big(m_{\pi}^2 \|m_{\mu}^2m_{\nu}^2 \big)\Big\}\nonumber \\
 &\qquad\qquad+\tau \Big\{  2m_{\pi}^2 \KL(\bar{\pi}_x \| \bar{\mu}) + \KL(m_{\pi}^2\|m_{\mu}^2)+ 2m_{\pi}^2 \KL(\bar{\pi}_y \| \bar{\nu}) + \KL(m_{\pi}^2\|m_{\nu}^2) \Big\}\nonumber\\
 = & m_{\pi}^2 \Upsilon + \varepsilon\KL(m_{\pi}^2\|m_{\mu}^2 m_{\nu}^2)+ \tau \Big\{\KL(m_{\pi}^2\|m_{\mu}^2) + \KL(m_{\pi}^2\|m_{\nu}^2) \Big\}, \label{eq:optimize_mass}
\end{align}
in which $\Upsilon$ is defined as 
\begin{align}\label{eq:S_formulation}
    \Upsilon:=\E_{\bar{\pi}\otimes\bar{\pi}} \Big\{ \big[ \langle X,X^{\prime}\rangle - \langle Y,Y^{\prime}\rangle \big]^2 \Big\} +2\varepsilon\KL(\bar{\pi}\|\bar{\mu}\otimes\bar{\nu})+2\tau\Big\{\KL(\bar{\pi}_x\|\bar{\mu})+\KL(\bar{\pi}_y\|\bar{\nu})\Big\}.
\end{align}
Now, we will divide the rest of the proof into two parts: shape optimization and mass optimization.
\paragraph{Shape optimization.} The problem of interest now is $\min_{\bar{\pi}}\Upsilon$. We will first prove that the optimal shape $\bar{\pi}$ is a zero-mean Gaussian measure. 
Denote by $(u,v)$ and $\Sigma_\pi$ the mean vector and covariance matrix of $\bar{\pi}$ where $u\in\br^m$ and $v\in\br^n$. Let $Z=X-u, Z^{\prime}=X^{\prime}-u$ and $T=Y-v,T^{\prime}=Y^{\prime}-v$. According to Lemma~\ref{lemma:expectation_factorize}, we have
\begin{align*}
    &\E_{\bar{\pi}\otimes\bar{\pi}} \Big\{ \big[ \langle X,X^{\prime}\rangle - \langle Y,Y^{\prime}\rangle \big]^2 \Big\}\\
    =&\E_{\bar{\pi}\otimes\bar{\pi}} \Big\{ \big[ \langle Z,Z^{\prime}\rangle - \langle T,T^{\prime}\rangle \big]^2 \Big\}+2u^{\top}\Sigma_x u +2u^{\top}\Sigma_{xy}v+2v^{\top}\Sigma_y v + \Big[\|u\|^2-\|v\|^2\Big]^2\\
    =&\tr(\Sigma^2_x)+\tr(\Sigma^2_y) - 2\tr(K^{\top}_{xy}K_{xy})+2u^{\top}\Sigma_x u+ 2u^{\top}\Sigma_{xy}v+2v^{\top}\Sigma_y v + \Big[\|u\|^2-\|v\|^2\Big]^2.
\end{align*}
The second equality is obtained by using the same arguments for deriving equation~\eqref{eq:first_equality_balanced}. Let us fix the covariance matrix and mean vector of $\bar{\pi}$, therefore, the value of the expectation term in equation~\eqref{eq:S_formulation} is also fixed. Then, Lemma~\ref{lemma:min_KL_Gaussian} indicates that $\bar{\pi}$ needs to be a Gaussian measure.
while Lemma~\ref{lemma:KL_calculate} forces its mean to be zero. 

Consequently, the problem $\min_{\bar{\pi}}\Upsilon$ is reduced to
\begin{align}\label{eq:optimize_covariance}
    \Upsilon^*:=\min_{\Sigma_\pi}\Big\{\tr(\Sigma^2_x)+\tr(\Sigma^2_y) - 2\tr(K^{\top}_{xy}K_{xy})+2\varepsilon\KL(\bar{\pi}\|\bar{\mu}\otimes\bar{\nu})+2\tau\KL(\bar{\pi}_x\|\bar{\mu})+2\tau\KL(\bar{\pi}_y\|\bar{\nu})\Big\}
\end{align}
Using the same arguments as in Theorem \ref{theorem:entropic_gw:inner_product:closed_form}, we obtain that the minimum value of
\begin{align*}
     \tr(\Sigma_{x}^2)+ \tr(\Sigma_y^2)- 2 \tr(K^{\top}_{xy} K_{xy}) +2\varepsilon\KL(\bar{\pi}\|\bar{\mu}\otimes\bar{\nu})
\end{align*}
is equal to
 \begin{align*}
     \|\lambda_x\|_2^2 + \|\lambda_y\|_2^2 - 2\sum_{k=1}^n \Big(\lambda_{x,k}\lambda_{y,k} - \frac{\varepsilon}{2}\Big)^+ + \varepsilon \sum_{k=1}^n   \Big[\log(\lambda_{x,k}\lambda_{y,k}) - \log\frac{\varepsilon}{2} \Big]^+.
 \end{align*}
Combining this result with parts (a) and (b) of Lemma~\ref{lemma:lower_bound_KL} implies that the problem~\eqref{eq:optimize_covariance} is equivalent to 
\begin{align*}
    \Upsilon^*=\min_{\lambda_x,\lambda_y>0}~\Big\{\sum_{k=1}^n g_{\varepsilon,\tau, +}(\lambda_{x,k},\lambda_{y,k}; \lambda_{\mu,k},\lambda_{\nu,k}) + \sum_{k=n+1}^m h_{\varepsilon,\tau}(\lambda_{x,k};\lambda_{\mu,k})\Big\},
\end{align*}
with note that $(\lambda_{x,k})_{k=1}^m$ and $(\lambda_{y,k})_{k=1}^n$ are  decreasing sequences. By Lemma \ref{lemma:order_eigenvalues}, each summation term of the above problem can be optimized independently but still preserving the order of $(\lambda_{x,k})_{k=1}^n$ as follows:
 \begin{align*}
      \Upsilon^*=\sum_{k=1}^n \min_{\lambda_{x,k},\lambda_{y,k} >0}&~ g_{\varepsilon,\tau,+}\big(\lambda_{x,k},\lambda_{y,k};\lambda_{\mu,i},\lambda_{\nu,j}\big)+\sum_{k=n+1}^m  \min_{\lambda_{x,k}>0}~h_{\varepsilon,\tau}\big(\lambda_{x,k},\lambda_{\mu,k}\big). 
 \end{align*}
A detailed calculation of~$\Upsilon^{*}$ is deferred to Appendix~\ref{sec:upsilon_calculation}. 
\paragraph{Mass optimization.} Finally, based on equation \eqref{eq:optimize_mass}, the optimal mass is obtained by taking the square root of the minimizer of the following problem: 
\begin{align*}
    m^2_{\pi^*}=\argmin_{x>0}~f(x):=\Upsilon^* x + \varepsilon \KL(x\|m^2_{\mu}m^2_{\nu})+\tau \KL(x\|m^2_{\mu}) + \tau \KL(x\|m^2_{\nu}).
\end{align*}
By part (c) of Lemma~\ref{lemma:maximizer_linear_log}, we obtain $m_{\pi^*}=(x^*)^{\frac{1}{2}}=(m_\mu m_\nu)^{\frac{\tau+\varepsilon}{2\tau + \varepsilon}}\exp\Big\{\frac{-\Upsilon^*}{2(2\tau + \varepsilon)} \Big\}$. Hence, the proof is completed.
\subsection{Proof of Theorem~\ref{theorem:barycenter}}
\label{appendix:barycenter}
Firstly, let $K_{\ell y}$ be the covariance matrix between $X_{\ell}$ and $Y$ while $v$ and $\Sigma_{y}$ be the mean and covariance matrix of $Y$, respectively. By Lemma~\ref{lemma:expectation_factorize}, the objective function in equation \eqref{definition:barycenter} is rewritten as 
\begin{align*}
    \sum_{\ell=1}^T \alpha_{\ell} \mathbb{E}_{\pi_{\ell,y}\otimes\pi_{\ell,y}} \Big\{  \big[\langle X_{\ell},X_{\ell}^{\prime} \rangle - \langle Y, Y^{\prime} \rangle \big]^2 \Big\}=
    \sum_{\ell=1}^{T} \alpha_{\ell} \Big\{\E_{\pi_{\ell,y}\otimes\pi_{\ell,y}}\big[\langle X_\ell, X^{\prime}_{\ell}\rangle - \langle Y-v,Y^{\prime}-v\rangle\big]^2+2v^{\top}\Sigma_y v + \|v\|^4\Big\}.
\end{align*}
As $\Sigma_y$ is a positive definite matrix, we have $2v^{\top}\Sigma_y v + \|v\|^4>0$ when $v\neq \zeros_d$. Therefore, the barycenter $\mu^*$ must have mean zero. Next, we denote by $\lambda_{\ell,i}$ and $\lambda_{y,i}$ the $i$th largest eigenvalues of $\Sigma_{\ell}$ and $\Sigma_{y}$, respectively. By using the same arguments for deriving equation~\eqref{eq:first_equality_balanced}, we get
\begin{align*}
    \sum_{\ell=1}^T \alpha_{\ell} \mathbb{E}_{\pi_{\ell,y}\otimes\pi_{\ell,y}} \Big\{  \big[\langle X_{\ell},X_{\ell}^{\prime} \rangle - \langle Y, Y^{\prime} \rangle \big]^2 \Big\}=\sum_{\ell=1}^{T} \alpha_{\ell} \Big\{\sum_{i=1}^{d} \lambda_{y,i}^2 - 2 \tr\big(K_{\ell y}^{\top}K_{\ell y} \big)+ \sum_{j=1}^{m_{\ell}} \lambda_{\ell,j}^2\Big\}.
\end{align*}
According to parts (a) and (c) of Lemma \ref{lemma:entropic_gw:inner_product:closed_form:supporting}, we have
\begin{align*}
    \tr(K_{\ell y}^{\top} K_{\ell y}) \leq   \sum_{j=1}^{\min\{d,m_{\ell}\}} \lambda_{y,j} \lambda_{\ell,j},
\end{align*}
where the equality happens when $(X_{\ell})_j = \sqrt{\lambda_{\ell,j}} Z_j$ and 
$ Y_{j} = \sqrt{\lambda_{y,j}} Z_j $ with $(Z_j)_j$ are i.i.d.  standard normal random variables. Thus, the problem now is reduced to minimize 
\begin{align*}
    \sum_{\ell =1}^{T} \alpha_{\ell} \Big\{\sum_{j=1}^d \lambda_{y,j}^2 - 2\sum_{j=1} ^{\min\{d, m_{\ell}\}} \lambda_{y,j} \lambda_{\ell,j}  \Big\} =\sum_{j=1}^d \Big\{ \lambda_{y,j}^2 - 2\lambda_{y,j} \sum_{\ell=1}^{T} \alpha_{\ell}\lambda_{\ell,j} \mathbf{1}_{j\leq m_{\ell}}\Big\}.
\end{align*}
This quantity has a quadratic form, so it attains its minimum at
\begin{align*}
    \lambda_{y,j} = \sum_{\ell=1}^{T} \alpha_{\ell} \lambda_{\ell,j} \mathbf{1}_{j\leq m_{\ell}}; \quad j\in[d].
\end{align*}
Hence, we obtain the conclusion of the theorem.
\subsection{Proof of Theorem~\ref{theorem:entropic_barycenter}}
\label{appendix:entropic_barycenter}
Firstly, we will show that the barycenter $\mu^*$ has mean zero. Let $K_{\ell y}$ be the covariance matrix between $X_{\ell}$ and $Y$ while $v$ and $\Sigma_{y}$ be the mean and covariance matrix of $Y$, respectively. By Lemma~\ref{lemma:expectation_factorize}, the objective function in equation~\eqref{definition:entropic_barycenter} can be rewritten as
\begin{align*}
    &\sum_{\ell=1}^{T} \alpha_{\ell} \Big\{ \mathbb{E}_{\pi_{\ell,y}\otimes\pi_{\ell,y}}  \big[\langle X_{\ell},X_{\ell}^{\prime} \rangle - \langle Y, Y^{\prime} \rangle \big]^2+ \varepsilon \KL(\pi_{\ell,y} \| \mu_{\ell}\otimes \mu)\Big\}\\
    =&\sum_{\ell=1}^{T} \alpha_{\ell} \Big\{\E_{\pi_{\ell,y}\otimes\pi_{\ell,y}}\big[\langle X_\ell, X^{\prime}_{\ell}\rangle - \langle Y-v,Y^{\prime}-v\rangle\big]^2+2v^{\top}\Sigma_y v + \|v\|^4\Big\}+\varepsilon \KL(\pi_{\ell,y} \| \mu_{\ell}\otimes \mu).
\end{align*}
Note that $\pi_{\ell,y}$ and $\mu_\ell\otimes\mu$ share the same mean and $\KL$ divergence is invariant to translation, which make the quantity $\KL(\pi_{\ell,y}\|\mu_{\ell}\otimes\mu)$ become independent of the mean vector of $\pi_{\ell,y}$. Additionally, as $\Sigma_y$ is a positive definite matrix, we have $2v^{\top}\Sigma_y v + \|v\|^4>0$ when $v\neq \zeros_d$. Therefore, the barycenter $\mu*$ must have mean zero. Moreover, according to Lemma~\ref{lemma:KL_conditional_gaussian_min}, there exists a Gaussian solution, says $\mu^*$. Then, the objective function is reduced to
\begin{align*}
    \sum_{\ell=1}^{T} \alpha_{\ell} \Big\{\sum_{i=1}^{d} \lambda_{y,i}^2 - 2 \tr\big(K_{\ell y}^{\top}K_{\ell y} \big)+ \sum_{j=1}^{m_{\ell}} \lambda_{\ell,j}^2\Big\}+\varepsilon \KL(\pi_{\ell,y} \| \mu_{\ell}\otimes \mu).
\end{align*}
Denote by $\Sigma_{\ell,y}$ and $\Sigma_{\ell \otimes y}$ the covariance matrices of $\pi_{\ell,y}$ and $ \mu_{\ell}\otimes \mu$, respectively. The entropic term in the above equation is then equal to
\begin{align*}
    \frac{\varepsilon}{2} \left\{ \tr(\Sigma_{\ell,y} \Sigma_{\ell\otimes y}^{-1}) + \log \Big(\frac{\det(\Sigma_{\ell\otimes y})}{\det (\Sigma_{\ell,y})} \Big) - (m_\ell + d) \right\} = \frac{\varepsilon}{2}  \log\left(\frac{\det (\Sigma_{\ell\otimes y})}{\det (\Sigma_{\ell,y})} \right).
\end{align*}
The problem is thus reduced to minimize
\begin{align}
    \sum_{j=1}^d \lambda_{y,j}^2 -2\sum_{\ell=1}^{T}  \alpha_{\ell}\tr\big(K_{\ell y}^{\top}K_{\ell y}\big)+\frac{1}{2} \sum_{\ell=1}^{T} \alpha_{\ell}\varepsilon \log\left(\frac{\det (\Sigma_{\ell \otimes y})}{\det (\Sigma_{\ell,y})} \right).\label{eq:entropic_IGW_p1}
\end{align}
Use the same approach of Lemma \ref{lemma:entropic_gw:inner_product:closed_form:supporting}, let $U_{\ell y} \Lambda_{\ell y}^{\frac{1}{2}} V_{\ell y}^{\top}$ be the SVD of matrix $\Sigma^{-\frac{1}{2}}_{\ell}K_{\ell y}\Sigma^{-\frac{1}{2}}_{y}$, or equivalently,
\begin{align*}
    K_{\ell y} = \Sigma^{\frac{1}{2}}_{\ell} U_{\ell y} \Lambda_{\ell y}^{\frac{1}{2}} V_{\ell y}^{\top} \Sigma^{\frac{1}{2}}_y.
\end{align*}
Denote by $\kappa^{\frac{1}{2}}_{\ell y,j}$ the $j$-th largest singular value of $\Sigma^{-\frac{1}{2}}_{\ell}K_{\ell y}\Sigma^{-\frac{1}{2}}_{y}$. Then, we find that 
\begin{align*}
  \log\left(\frac{\det (\Sigma_{\ell\otimes y})}{\det (\Sigma_{\ell,y})} \right) = - \sum_{j=1}^{d_{\ell}} \log(1-\kappa_{\ell y,j}).
\end{align*}
By applying the von Neumann's trace inequality \citep{kristof1969neumann,horn_johnson_1991}, we have
\begin{align*}
    \tr(K_{\ell y}^{\top}K_{\ell y}) =& \tr\Big(\Sigma^{\frac{1}{2}}_y V_{\ell y} (\Lambda_{\ell y}^{\frac{1}{2}})^{\top} U_{\ell y}^{\top} \Sigma_{\ell} U_{\ell y} \Lambda_{\ell y} V_{\ell y}^{\top}   \Sigma^{\frac{1}{2}}_y \Big) 
    = \tr\Big( [V_{\ell y} (\Lambda_{\ell y}^{\frac{1}{2}})^{\top}]\times [U_{\ell y}^{\top} \Sigma_{\ell }]\times  [U_{\ell y} \Lambda_{\ell y}^{\frac{1}{2}}] \times [V_{\ell y}^{\top} \Sigma_y] \Big) \\
    \leq &\sum_{j=1}^{d_{\ell}} \lambda_{\ell,j} \lambda_{y,j} \kappa_{\ell y,j},
\end{align*}
with a note that all $\Lambda_{\ell y}$, $\Sigma_y$ and $\Sigma_{\ell}$ are diagonal matrices. The equality occurs when $V_{\ell y}$ and $U_{\ell y}$ are identity matrices. 

Put the above results together, the problem in equation \eqref{eq:entropic_IGW_p1} is reduced to minimizing
\begin{align*}
    \sum_{j=1}^d \lambda_{y,j}^2 - \sum_{\ell=1}^{T} 2\alpha_{\ell} \sum_{t=1}^{d_\ell} \lambda_{\ell,t}\kappa_{\ell y,t} \lambda_{y,t} - \frac{\varepsilon}{2} \sum_{\ell=1}^{T} \alpha_{\ell} \sum_{t=1}^{d_\ell} \log(1-\kappa_{\ell y,t}), 
\end{align*}
or equivalently,
\begin{align*}
    \sum_{j=1}^d \Big\{ \lambda_{y,j}^2 - 2\lambda_{y,j} \sum_{\ell=1}^{T} \alpha_{\ell} \lambda_{\ell,j} \kappa_{\ell y,j} \mathbf{1}_{j\leq d_{\ell}}\Big\}- \frac{\varepsilon}{2} \sum_{\ell=1}^{T} \alpha_{\ell} \sum_{t=1}^{d_\ell}\log(1 - \kappa_{\ell y,t}).
\end{align*}
By fixing the values of $\kappa_{\ell y,t}\in[0,1]$ for all $t\in[d_\ell]$, this quantity attains its minimum at 
\begin{align*}
    \lambda_{y,j} = \sum_{\ell=1}^{T} \alpha_{\ell} \lambda_{\ell,j} \kappa_{\ell y,j} \mathbf{1}_{j\leq d_{\ell}}, \quad j\in [d].
\end{align*}
Therefore, the problem of entropic IGW barycenter is equivalent to finding the maximum of 
\begin{align*}
   & \sum_{j=1}^d \Big[\sum_{\ell=1}^{T} \alpha_{\ell} \lambda_{\ell,j} \kappa_{\ell y,j} \mathbf{1}_{j\leq d_{\ell}}\Big]^2 + \frac{\varepsilon}{2}\sum_{\ell=1}^{T}  \sum_{j=1}^{d_{\ell}} \alpha_{\ell} \log(1 - \kappa_{\ell y,j})\\
&    = \sum_{j=1}^d \Big\{\Big[\sum_{\ell=1}^{T} \alpha_{\ell} \lambda_{\ell,j} \kappa_{\ell y,j} \mathbf{1}_{j\leq d_{\ell}}\Big]^2+ \frac{\varepsilon}{2} \sum_{\ell=1}^{T} \alpha_{\ell} \log(1-\kappa_{\ell y,j}) \mathbf{1}_{j\leq d_{\ell}}\Big\}.
\end{align*}
We need to maximize each term, which is 
\begin{align*}
     \Big[\sum_{\ell=1}^{T} \alpha_{\ell} \lambda_{\ell,j} \kappa_{\ell y,j} \mathbf{1}_{j\leq d_{\ell}}\Big]^2 +  \sum_{\ell=1}^{T} \frac{\varepsilon}{2} \alpha_{\ell} \log(1 - \kappa_{\ell y,j}) \mathbf{1}_{j\leq d_{\ell}},
\end{align*}
under the constraints that $0\leq \kappa_{\ell y,j}\leq 1$, for all $j\in[d_\ell]$. By Lemma \ref{lemma:minimization_square_log} in Appendix~\ref{sec:auxiliary_results}, and when $\varepsilon$ satisfying the condition~\eqref{condition:varepsilon_barycenter} for $j\leq \min \{d,m_{\ell}\}$, we have
\begin{align*}
    \kappa_{\ell y,j} = 1 - \frac{\varepsilon  }{2 \lambda_{\ell,j} \Big\{ A_{j} + \sqrt{ A_{j}^2 - \varepsilon  B_j }\Big\} },
\end{align*}
where $A_{j} = \sum_{\ell=1}^{T} \alpha_{\ell} \lambda_{\ell,j} \mathbf{1}_{j\leq d_{\ell}}$ and $B_j = \sum_{\ell=1}^{T} \alpha_{\ell} \mathbf{1}_{j\leq d_{\ell}}$.
Hence, we reach the conclusion of the theorem.
\subsection{Proof of Lemma \ref{lemma:KL_conditional_gaussian_min}}
\label{appendix:KL_conditional_gaussian_min}
Let $Q_{X,Y}=\mathcal{N}([\gamma_x,\gamma_y]^{\top},\Sigma_{X,Y})$ and $Q_Y=\mathcal{N}(\gamma_y,\Sigma_y)$ be Gaussian distributions in $\br^{m+n}$ and $\br^n$, respectively. Denote by $q_X, q_Y, q_{X,Y}, p_Y$ and $p_{X,Y}$ the probability density functions of distributions $Q_X, Q_Y, Q_{X,Y}, P_Y$ and $P_{X,Y}$, respectively, we need to show that
\begin{align*}
    \mathsf{KL}\big(P_{X,Y} \|Q_X \otimes P_Y \big) \geq \mathsf{KL} \big(Q_{X,Y}\|Q_X\otimes Q_Y \big).
\end{align*}
Expanding the KL divergence, we need to prove
\begin{align*}
    \int p_{X,Y}(x,y) \big[\log p_{X,Y}(x,y) - \log q_{X}(x) - \log p_Y(y)\big] dx dy \geq \int q_{X,Y}(x,y) \big[\log q_{X,Y}(x,y)- \log q_X(x) -  \log q_Y(y) \big]  dx dy.
\end{align*}
Similar to the derivation for equation~\eqref{eq:expectation_log} in the proof Lemma \ref{lemma:min_KL_Gaussian} with a note that $q_{X,Y}$, $q_X$ and $q_Y$ are Gaussian distributions, we have
\begin{align*}
    \int p_{X,Y}(x,y) \log q_X(x) dx dy &= \int q_{X,Y}(x,y) \log q_X(x) dx dy;\\
    \int p_{X,Y}(x,y) \log q_Y(y) dx dy &= \int q_{X,Y}(x,y) \log q_Y(y) dx dy;\\
    \int p_{X,Y}(x,y) \log q_{X,Y}(x,y) dx dy &= \int q_{X,Y}(x,y) \log q_{X,Y}(x,y) dx dy,
\end{align*}
since both sides are equal to  the same function of the covariance matrices of $P_{X,Y}$ and $Q_{X,Y}$, which are both equal to $\Sigma_{X,Y}$. Hence, the inequality is equivalent to
\begin{align*}
    \int p_{X,Y}(x,y) \big[\log p_{X,Y}(x,y) - \log p_Y(y)\big]dxdy &\geq \int p_{X,Y}(x,y) \big[\log q_{X,Y}(x,y) - \log q_{Y}(y) \big]dx dy\\
    \Leftrightarrow \int p_{X,Y}(x,y) \log \frac{p_{X,Y}(x,y)/p_Y(y)}{q_{X,Y}(x,y)/q_{Y}(y)} dx dy &\geq 0.
\end{align*}
By using the formula for conditional distributions, we get 
\begin{align*}
    p_{X,Y}(x,y) &= p_Y(y) p_{X|Y}(x|y);\\
    q_{X,Y}(x,y) & = q_Y(y) q_{X|Y}(x|y).
\end{align*}
Then the left hand side is equal to
\begin{align*}
    \int p_Y(y) p_{X|Y}(x|y) \log \frac{p_{X|Y}(x|y)}{q_{X|Y}(x|y)} dx dy &= \int_y  p_Y(y) \Big[\int_x p_{X|Y}(x|y) \log \frac{p_{X|Y}(x|y)}{q_{X|Y}(x|y)} dx\Big] dy \\
    &=\int_y p_Y(y)\mathsf{KL}\big(P_{X|Y} \| Q_{X|Y}\big) dy
\end{align*}
which is not less than zero, since the KL term is non-negative. The inequality becomes equality when $P_{X|Y} = Q_{X|Y}$. 
\section{On detailed calculation of $\Upsilon^*$}
\label{sec:upsilon_calculation}
In this appendix, we discuss how to derive a closed-form formulation for $\Upsilon^{*}$ in Theorem~\ref{theorem:UGW}. It is equivalent to finding the minimizers of functions $g_{\varepsilon,\tau,+}$ and $h_{\varepsilon,\tau}$ in Lemma~\ref{lemma:minimizer_g_plus} and Lemma~\ref{lemma:minimizer_h}, respectively. 
\begin{lemma}[Minimizer of $g_{\varepsilon,\tau,+}$]\label{lemma:minimizer_g_plus}
Let 
\begin{align*}
    g_{\varepsilon,\tau,+}(x,y;a,b) := x^2 + y^2  + (\tau + \varepsilon) \Big(\frac{x}{a} + \frac{y}{b} - \log\Big(\frac{xy}{ab} \Big) - 2 \Big) - 2\Big[ xy - \frac{\varepsilon}{2}\Big]^+  + \varepsilon \Big[\log(xy) - \log\Big(\frac{\varepsilon}{2}\Big) \Big]^+.
\end{align*}
We define 
\begin{align*}
    g_{\varepsilon,\tau,-1}(x,y;a,b) &:= x^2 + y^2 + (\tau+\varepsilon) \Big[\frac{x}{a} + \frac{y}{b} - \log\Big(\frac{xy}{ab}\Big)-2  \Big], \\
    g_{\varepsilon,\tau,1}(x,y;a,b) &:= x^2 + y^2 + (\tau+\varepsilon) \Big[\frac{x}{a} + \frac{y}{b} - \log\Big(\frac{xy}{ab}\Big)-2 \Big] - 2(xy- \frac{\varepsilon}{2}) + \varepsilon\Big[\log(xy) - \log\Big(\frac{\varepsilon}{2}\Big) \Big], \\
    (\widetilde{x},\widetilde{y}) &:= \argmin_{x,y>0}~ g_{\varepsilon,\tau,-1}(x,y;a,b), \\
    (\widehat{x},\widehat{y})&:= \argmin_{x,y>0}~ g_{\varepsilon,\tau,1}(x,y;a,b), \\
    (x^*,y^*) &:=\argmin_{x,y>0}~g_{\varepsilon,\tau,+}(x,y,;a,b).
\end{align*}
Then if $\widetilde{x} \widetilde{y} < \frac{\varepsilon}{2}$, then $(x^*,y^*) =(\widetilde{x},\widetilde{y}) $ where solutions are in Lemma \ref{lemma:minimizer_g_neg1}, and if otherwise, then $(x^*,y^*) = (\widehat{x},\widehat{y})$, where solutions are in Lemma \ref{lemma:minimizer_g_plus1}.
\end{lemma}
\begin{proof}[Proof of Lemma \ref{lemma:minimizer_g_plus}]
 By definition of functions $g_{\varepsilon,\tau,1}$, $g_{\varepsilon,\tau,-1}$ and $g_{\varepsilon,\tau,+}$, we have
\begin{align*}
    g_{\varepsilon,\tau,+}(x,y;a,b) = g_{\varepsilon,\tau,-1}(x,y;a,b) \mathbf{1}_{\{xy < \frac{\varepsilon}{2}\}} + g_{\varepsilon,\tau,1}(x,y;a,b) \mathbf{1}_{\{xy \geq \frac{\varepsilon}{2}\}}.
\end{align*}
We observe that both functions $g_{\varepsilon,\tau,1}$ and $g_{\varepsilon,\tau,-1}$ are convex functions. Moreover, we have
\begin{align*}
    g_{\varepsilon,\tau,-1}(x,y;a,b) - g_{\varepsilon,\tau,1}(x,y;a,b) &= 2\Big(xy - \frac{\varepsilon}{2}\Big) - \varepsilon\Big[ \log(xy) - \log\Big(\frac{\varepsilon}{2}\Big)\Big]\\
    &= \varepsilon \Big\{ \frac{xy}{\varepsilon /2} - 1 - \log\Big(\frac{xy}{\varepsilon/2}\Big) \Big\} \geq 0.
\end{align*}
It means that $g_{\varepsilon,\tau,-1}(x,y;a,b) \geq g_{\varepsilon,\tau,1}(x,y;a,b)$.

Next, we prove that   $\frac{\varepsilon}{2}$ cannot lie between  $\widehat{x} \widehat{y}$ and  $\widetilde{x} \widetilde{y}$. Assume the contradictory, then the segment connecting $(\widetilde{x},\widetilde{y})$ to $(\widehat{x},\widehat{y})$ cuts the hyperpole $xy = \frac{\varepsilon}{2}$ at a point with coordinates denoted by $(\overline{x},\overline{y})$. Since $g_{\varepsilon,\tau,1}$ is a convex function, we have 
\begin{align}
    \max \Big\{g_{\varepsilon,\tau,1}(\widehat{x},\widehat{y};a,b), g_{\varepsilon,\tau,1}(\widetilde{x}, \widetilde{y};a,b) \Big\} > g_{\varepsilon,\tau,1}(\overline{x},\overline{y};a,b). \label{inequality:g}
\end{align}
Moreover, we find that
\begin{align*}
    g_{\varepsilon,\tau,1}(\overline{x},\overline{y};a,b) &\geq g_{\varepsilon,\tau,1}(\widehat{x},\widehat{y};a,b) \\
    g_{\varepsilon,\tau,1}(\overline{x},\overline{y};a,b) &= g_{\varepsilon,\tau,-1}(\overline{x},\overline{y};a,b)\geq g_{\varepsilon,\tau,-1}(\widetilde{x},\widetilde{y};a,b)\geq g_{\varepsilon,\tau,1}(\widetilde{x},\widetilde{y};a,b).
\end{align*}
It contradicts inequality \eqref{inequality:g}. Thus, both $\widehat{x} \widehat{y}$ and $\widetilde{x} \widetilde{y}$ have to be greater or smaller than  $\frac{\varepsilon}{2}$ at the same time. Given that, we have two separate cases:

\vspace{0.5 em}
\noindent
\textbf{Case 1}: They are both greater than $\frac{\varepsilon}{2}$, then 
\begin{align*}
    g_{\varepsilon,\tau,-1}(x,y;a,b) \geq g_{\varepsilon,\tau,-1}(\widetilde{x},\widetilde{y};a,b) \geq g_{\varepsilon,\tau,1}(\widetilde{x},\widetilde{y};a,b) \geq g_{\varepsilon,\tau,1}(\widehat{x},\widehat{y};a,b).
\end{align*}
It means that $\widehat{x} = x^*$ and $\widehat{y} = y^*$.

\vspace{0.5 em}
\noindent
\textbf{Case 2}: Both of them are smaller than $\frac{\varepsilon}{2}$. For any $(x,y)$ such that $xy \geq \frac{\varepsilon}{2}$, there exists $\big(\overline{x},\overline{y}\big)$ lies in the segment connecting $(x,y)$ to $(\widehat{x},\widehat{y})$ such that $\overline{x}~\overline{y} = \frac{\varepsilon}{2}$. Since the function $g_{\varepsilon,\tau,1}$ is convex, we find that
\begin{align*}
    t g_{\varepsilon,\tau,1}(\widehat{x},\widehat{y};a,b) + (1-t) g_{\varepsilon,\tau,1}(x,y;a,b) \geq g_{\varepsilon,\tau,1}(\overline{x},\overline{y};a,b),
\end{align*}
where $t = \frac{\|(\widehat{x},\widehat{y}) - (\overline{x},\overline{y})\|_2}{\|(x,y) - (\widehat{x},\widehat{y})\|_2}$.
However, $g_{\varepsilon,\tau,1}(\widehat{x},\widehat{y};a,b) \leq g_{\varepsilon,\tau,1}(\overline{x},\overline{y};a,b)$, which implies that
\begin{align*}
    g_{\varepsilon,\tau,1}(x,y;a,b) \geq g_{\varepsilon,\tau,1}(\overline{x},\overline{y};a,b) = g_{\varepsilon,\tau,-1}(\overline{x},\overline{y};a,b) \geq g_{\varepsilon,\tau,-1}(\widetilde{x},\widetilde{y};a,b).
\end{align*}
It means that $\widetilde{x} = x^*$ and $\widetilde{y} = y^*$. As a consequence, we obtain the conclusion of the lemma.
\end{proof}

\begin{lemma}[Minimizer of function $g_{\varepsilon,\tau,-1}$] \label{lemma:minimizer_g_neg1}
\label{lemma:ugw:property_of_f^2_j}
With $a,b> 0$, let
\begin{equation*}
    g_{\varepsilon,\tau,-1}(x,y;a,b) = x^2+y^2+
    (\tau+\varepsilon) \left(\dfrac{x}{a}+\dfrac{y}{b} - \log(xy)-\log(ab)\right)
\end{equation*}
subjected to $x,y > 0$. Let $(\widetilde{x},\widetilde{y}) = \argmin_{x,y>0} g_{\varepsilon,\tau,-1}(x,y;a,b)$. Then, 
\begin{equation*}
\widetilde{x} = -\dfrac{\tau+\varepsilon}{4a} + \dfrac{1}{2}\sqrt{2(\tau + \varepsilon)+ \dfrac{(\tau+\varepsilon)^2}{4a^2}}, \quad \widetilde{y} = -\dfrac{\tau+\varepsilon}{4b} + \dfrac{1}{2}\sqrt{2(\tau +\varepsilon)+ \dfrac{(\tau+\varepsilon)^2}{4b^2}}.
\end{equation*}
\end{lemma}

\begin{proof}[Proof of Lemma \ref{lemma:ugw:property_of_f^2_j}]
1. Taking the derivatives of $g_{\varepsilon,\tau,-1}$ with respect to $x$ and $y$, we get
\begin{equation*}
    \dfrac{\partial g_{\varepsilon,\tau,-1}}{\partial x} = 2x + \dfrac{\tau+\varepsilon}{a} - \dfrac{\tau+\varepsilon}{x}, \quad \dfrac{\partial g_{\varepsilon,\tau,-1}}{\partial y} = 2y + \dfrac{\tau+\varepsilon}{b} - \dfrac{\tau+\varepsilon}{y}.
\end{equation*}
It follows that
\begin{equation*}
    \dfrac{\partial^2 g_{\varepsilon,\tau,-1}}{\partial x^2} = 2 + \dfrac{\tau+\varepsilon}{x^2}, \quad \dfrac{\partial^2 g_{\varepsilon,\tau,-1}}{\partial x\partial y} = 0, \quad  \dfrac{\partial^2 g_{\varepsilon,\tau,-1}}{\partial y^2} = 2 + \dfrac{\tau+\varepsilon}{y^2}.
\end{equation*}
The Hessian matrix of function $g_{\varepsilon,\tau,-1}$ is positive definite, so $g_{\varepsilon,\tau,-1}$ is a convex function. 

2. Solving the equations $\dfrac{\partial g_{\varepsilon,\tau,-1}}{\partial x} = 0$ and $\dfrac{\partial g_{\varepsilon,\tau,-1}}{\partial y} = 0$, we have
\begin{equation*}
\widetilde{x} = -\dfrac{\tau+\varepsilon}{4a} + \dfrac{1}{2}\sqrt{2(\tau +\varepsilon) + \dfrac{(\tau+\varepsilon)^2}{4a^2}}, \quad \tilde{y} = -\dfrac{\tau+\varepsilon}{4b} + \dfrac{1}{2}\sqrt{2(\tau +\varepsilon) + \dfrac{(\tau+\varepsilon)^2}{4b^2}}.
\end{equation*}
Therefore, we reach the conclusion of the lemma.
\end{proof}

\begin{lemma}[Minimizer of function $g_{\varepsilon,\tau,1}$]
\label{lemma:minimizer_g_plus1}
For $a,b>0$, let
\begin{equation*}
    g_{\varepsilon,\tau,1}(x,y;a,b) = x^2+y^2+\varepsilon\Big[\log(xy)-\log\Big(\frac{\varepsilon}{2}\Big)\Big]  -2\Big(xy-\frac{\varepsilon}{2}\Big)
    +(\tau+\varepsilon) \left[\dfrac{x}{a}+\dfrac{y}{b} - \log\Big(\frac{xy}{ab}\Big)-2\right],
\end{equation*}
subjected to $x,y > 0$. Define $(\widehat{x},\widehat{y}) = \argmin_{x,y>0} g_{\varepsilon,\tau,1}(x,y;a,b) $. Then, the value of $(\widehat{x},\widehat{y})$ is given in the following equations.
    \begin{equation*}
    \widehat{x} = -\dfrac{\tau+\varepsilon}{4a} +\dfrac{1}{2} \sqrt{\dfrac{(\tau+\varepsilon)^2}{4a^2}+4\tau \Big(\frac{1}{2} + \frac{1}{\widehat{z}} \Big) }; \quad \widehat{y} = -\dfrac{\tau+\varepsilon}{4b} +\dfrac{1}{2} \sqrt{\dfrac{(\tau+\varepsilon)^2}{4b^2}+4\tau \Big(\frac{1}{2} + \frac{1}{\widehat{z}} \Big) },
    \end{equation*}
    where $\widehat{z}$ is the unique positive root of the equation 
    \begin{equation*}
        \tau z^3 + \Big[ 8 \tau - \frac{(\tau + \varepsilon)^2}{ab} \Big]z^2 + \Big[16\tau - 2(\tau+\varepsilon)^2 \Big(\frac{1}{a} + \frac{1}{b} \Big)^2 \Big] z - 4 (\tau+\varepsilon)^2 \Big(\frac{1}{a} +\frac{1}{b} \Big)^2 = 0.
    \end{equation*}
\end{lemma}
\begin{proof}[Proof of Lemma \ref{lemma:minimizer_g_plus1}]
 Taking the first derivatives of $g_{\varepsilon,\tau,1}$ with respect to $x$ and $y$ , we have
\begin{equation*}
    \dfrac{\partial g_{\varepsilon,\tau,1}}{\partial x} = 2x - 2y + \dfrac{\tau+\varepsilon }{a} - \dfrac{ \tau}{x}, \quad \dfrac{\partial g_{\varepsilon,\tau,1}}{\partial y} = 2y - 2x + \dfrac{\tau+\varepsilon}{b} - \dfrac{\tau}{y}.
\end{equation*}
It follows that
\begin{equation*}
    \dfrac{\partial^2 g_{\varepsilon,\tau,1}}{\partial x^2} = 2 + \dfrac{\tau}{2x^2}, \quad \dfrac{\partial^2 g_{\varepsilon,\tau,1}}{\partial x\partial y} = -2, \quad  \dfrac{\partial^2 g_{\varepsilon,\tau,1}}{\partial y^2} = 2 + \dfrac{\tau}{y^2}.
\end{equation*}
Thus, the Hessian matrix of this function is positive definite, which implies that $g_{\varepsilon,\tau,1}$ is a convex function. 

The equations of the stationary point give us
\begin{equation*}
    2\widehat{x} - 2\widehat{y} + \frac{\tau + \varepsilon}{a} - \frac{\tau}{\widehat{x}} = 0;  \qquad 2\widehat{y} - 2\widehat{x} + \frac{\tau+\varepsilon}{b} - \frac{\tau}{\widehat{y}} =0.
\end{equation*}
Taking the sum of both equations leads to
\begin{equation*}
    \dfrac{\tau+\varepsilon}{a}+\dfrac{\tau+\varepsilon}{b} = \dfrac{\tau}{\widehat{x}}+\dfrac{\tau}{\widehat{y}} \Leftrightarrow \tau \left(\dfrac{1}{\widehat{x}}+\dfrac{1}{\widehat{y}} -\dfrac{1}{a} - \dfrac{1}{b}\right) = \dfrac{\varepsilon}{a}+\dfrac{\varepsilon}{b}.
\end{equation*}

\begin{equation}
\label{eqn:ugw:system_of_worst_case}
    \begin{cases}
     2(\widehat{x} - \widehat{y}) + \dfrac{\tau+\varepsilon}{a} = \dfrac{\tau}{\widehat{x}}\\
     \text{}\\
     2(\widehat{y} - \widehat{x}) + \dfrac{\tau+\varepsilon}{b} = \dfrac{\tau}{\widehat{y}}.
    \end{cases}
\end{equation}
Taking the difference between two equations in system~\eqref{eqn:ugw:system_of_worst_case}, we get
\begin{equation*}
    (\widehat{x}-\widehat{y})\Big(4+\dfrac{\tau}{\widehat{x} \widehat{y}}\Big) = \dfrac{\tau+\varepsilon}{b} - \dfrac{\tau+\varepsilon}{a}.
\end{equation*}
Let $z = \dfrac{\tau}{\widehat{x} \widehat{y}}$, we have $\widehat{x}- \widehat{y} = (\tau+ \varepsilon) \dfrac{1/b-1/a}{4+z}$. Take the product of two equations in system~\eqref{eqn:ugw:system_of_worst_case}, we have
\begin{equation*}
    \Big[2(\widehat{x}- \widehat{y})+\dfrac{\tau+\varepsilon}{a}\Big]\Big[2(\widehat{y}-\widehat{x})+\dfrac{\tau+\varepsilon}{b}\Big] = \dfrac{\tau^2}{\widehat{x}\widehat{y}}.
\end{equation*}
Substitute $\widehat{x}- \widehat{y} = (\tau+\varepsilon)\frac{1/b-1/a}{4+z}$ to the above equation, we obtain
\begin{equation*}
    (\tau+\varepsilon)^2\Big[2 \Big(\frac{1}{a} + \frac{1}{b} \Big) + \frac{z}{a} \Big]\Big[2\Big(\frac{1}{a} + \frac{1}{b} \Big) + \frac{z}{b} \Big] = \tau z(z+4)^2,
\end{equation*}
which is equivalent to the following cubic equation
\begin{align*}
  t(z):=\tau z^3 + \Big[ 8 \tau - \frac{(\tau + \varepsilon)^2}{ab} \Big]z^2 + \Big[16\tau - 2(\tau+\varepsilon)^2 \Big(\frac{1}{a} + \frac{1}{b} \Big)^2 \Big] z - 4 (\tau+\varepsilon)^2 \Big(\frac{1}{a} +\frac{1}{b} \Big)^2 = 0.
\end{align*}
Since it is a cubic equation, it has either three real roots or one real root. In the first case, we know that it has at least one positive root, since the highest coefficient of the equation is positive. Assume that it has other two positive roots. By applying Viete's theorem, we get 
\begin{align*}
    8\tau - \frac{(\tau+\varepsilon)^2}{ab} <0; \qquad 16\tau - (\tau+\varepsilon)^2\Big(\frac{1}{a} + \frac{1}{b}\Big)^2 >0,
\end{align*}
which implies that $\big(\frac{1}{a} + \frac{1}{b}\big)^2 < \frac{2}{ab}$, it is contradiction by Cauchy-Schwarz's inequality. Thus, the equation $t(z)=0$ has unique positive root which is denoted by $\widehat{z}$.

With $\widehat{z} = \dfrac{\tau}{\widehat{x}\widehat{y}}$, we have $\widehat{x}\widehat{y} = \dfrac{\tau}{\widehat{z}}$. Substituting it to the system of equations \eqref{eqn:ugw:system_of_worst_case}, we have
\begin{equation*}
    \begin{cases}
     \widehat{x}^2+\dfrac{\tau+\varepsilon}{2a} \widehat{x} = \dfrac{\tau}{\widehat{z}} +\dfrac{\tau}{2}\\
     \text{}\\
     \widehat{y}^2+\dfrac{\tau+\varepsilon}{2b} \widehat{x} = \dfrac{\tau}{\widehat{z}} +\dfrac{\tau}{2}.
    \end{cases}
\end{equation*}
This implies that
\begin{equation*}
    \widehat{x} = -\dfrac{\tau+\varepsilon}{4a} +\dfrac{1}{2} \sqrt{\dfrac{(\tau+\varepsilon)^2}{4a^2}+4\tau \Big(\frac{1}{2} + \frac{1}{\widehat{z}} \Big) }, \quad \widehat{y} = -\dfrac{\tau+\varepsilon}{4b} +\dfrac{1}{2} \sqrt{\dfrac{(\tau+\varepsilon)^2}{4b^2}+4\tau \Big(\frac{1}{2} + \frac{1}{\widehat{z}} \Big) }.
\end{equation*}
As a consequence, we obtain the conclusion of the lemma.
\end{proof}
\begin{lemma}[Minimizer of function $h_{\varepsilon,\tau}$]
    \label{lemma:minimizer_h}
    Given a constant $a>0$, the function
    \begin{equation*}
        h_{\varepsilon,\tau}(x;a) = x^2+(\tau + \varepsilon) \Big[\frac{x}{a} - \log \Big(\frac{x}{a}\Big) -1 \Big]
    \end{equation*}
    attains its minimum at
    $x^* = -\dfrac{\tau+\varepsilon}{4a} + \sqrt{\dfrac{(\tau+\varepsilon)^2}{16a^2}+\dfrac{\tau+\varepsilon}{2}}$.
\end{lemma}

\begin{proof}[Proof of Lemma \ref{lemma:minimizer_h}]
Taking the first and second derivatives of function $ h_{\tau,\varepsilon}$, we have
\begin{equation*}
    \dfrac{\partial h_{\tau,\varepsilon}}{\partial x} = 2x+\dfrac{\tau+\varepsilon}{a}-\dfrac{\tau+\varepsilon}{x}, \quad \dfrac{\partial^2 h_{\tau,\varepsilon}}{\partial x^2} = 2 + \dfrac{\tau+\varepsilon}{x^2}.
\end{equation*}
Note that the second derivative of $h_{\tau,\varepsilon}$ is positive, therefore, the positive minimizer of this function is exactly the positive solution of equation $\dfrac{\partial h_{\tau,\varepsilon}}{\partial x} = 0$, which is
\begin{equation*}
    x^* = -\dfrac{\tau+\varepsilon}{4a} + \sqrt{\dfrac{(\tau+\varepsilon)^2}{16a^2}+\dfrac{\tau+\varepsilon}{2}}.
\end{equation*}
\end{proof}
\section{Auxiliary results}
\label{sec:auxiliary_results}
In this appendix, we provide additional lemmas that are used to derive the closed-form expressions of entropic (unbalanced) IGW between (unbalanced) Gaussian distributions.
\begin{lemma}[KL divergence between Gaussian measures]
 \label{lemma:kl_divergence}
 Let $\alpha = m_{\alpha}\mathcal{N}_k(\mathbf{0},\Sigma_{\alpha})$ and $\beta=m_{\beta} \mathcal{N}_k(\mathbf{0},\Sigma_{\beta})$ be scaled Gaussian measures while $\overline{\alpha}=\mathcal{N}_k(\zeros,\Sigma_\alpha)$ and $\overline{\beta}=\mathcal{N}_k(\zeros,\Sigma_\beta)$ be their normalized versions. Then, the generalized KL divergence between $\alpha$ and $\beta$ is
 \begin{align*}
     \KL(\alpha\|\beta) &= m_{\alpha} \KL(\overline{\alpha}\|\overline{\beta}) + \KL(m_{\alpha}\|m_{\beta}),
 \end{align*}
 where the KL divergence between $\overline{\alpha}$ and $\overline{\beta}$ is
 \begin{align*}
     \frac{1}{2}\left\{ \tr\big(\Sigma_{\alpha}\Sigma_{\beta}^{-1}\big)- k + \log \Big(\frac{\det (\Sigma_{\beta})}{\det (\Sigma_{\alpha})} \Big)\right\}.
 \end{align*}
 \end{lemma}
 \begin{proof}[Proof of Lemma~\ref{lemma:kl_divergence}]
 Let $f(x,\Sigma_{\alpha})$ and $f(x,\Sigma_{\beta})$ be the distribution functions of $\mathcal{N}_k(\mathbf{0},\Sigma_{\alpha})$ and $\mathcal{N}_k(\mathbf{0},\Sigma_{\beta})$, respectively. 
 \begin{align*}
     \KL(\alpha\|\beta) &= \int_{\mathbb{R}^k} \log\Big(\frac{m_{\alpha} f(x,\Sigma_{\alpha})}{m_{\beta} f(x,\Sigma_{\beta}} \Big) d\alpha(x) - m_{\alpha} + m_{\beta} \\
     &= \int_{\mathbb{R}^k} \log\Big( \frac{f(x,\Sigma_{\alpha})}{f(x,\Sigma_{\beta})}\Big) m_{\alpha} d\overline{\alpha}(x)  + \log\Big(\frac{m_{\alpha}}{m_{\beta}} \Big) m_{\alpha} - m_{\alpha} + m_{\beta} \\
     &= m_{\alpha} \KL(\overline{\alpha}\|\overline{\beta}) + \KL(m_{\alpha}\|m_{\beta}).
 \end{align*}
The expression for $\KL(\overline{\alpha}\|\overline{\beta})$ results from Lemma~\ref{lemma:KL_calculate}. Hence, the proof is completed.
 \end{proof}
\begin{lemma}[Double integrals]\label{lemma:double_integrals}
For the sake of simple computations, we derive some relations between the quadratic-KL and the standard KL as follows: 
 \begin{align*}
 \textit{(i) } &\KLD(\alpha \| \beta) = 2 m_\alpha \KL(\alpha \| \beta) + (m_\alpha - m_\beta)^2;\\
 \textit{(ii) } &\KLD(t\alpha \| r \beta) = t^2 \KLD(\alpha \| \beta) + t^2 \log\left(\frac{t^2}{r^2}\right) m_\alpha^2 + (r^2 - t^2) m_\beta^2,\quad t,r>0,
 \end{align*}
 for any positive measures $\alpha$ and $\beta$.
 \end{lemma}
 \begin{proof}[Proof of Lemma~\ref{lemma:double_integrals}]
Let $p_\alpha$ and $p_\beta$ be the Radon-Nikodym derivatives of $\alpha$ and $\beta$ with respect to the Lebesgue measure. Then, we have\\
\textit{Part (i)} 
\begin{align*}
    \KLD(\alpha\|\beta)&=\int\int\log\left(\frac{p_\alpha(x)p_\alpha(x')}{p_\beta(x)p_\beta(x')}\right)p_\alpha(x)p_\alpha(x')dxdx' - m^2_\alpha + m^2_\beta\\
    &=m_\alpha\int\log\left(\frac{p_\alpha(x)}{p_\beta(x)}\right)p_\alpha(x)dx+m_\alpha\int\log\left(\frac{p_\alpha(x')}{p_\beta(x')}\right)p_\alpha(x')dx'- m^2_\alpha + m^2_\beta\\
    &= 2m_\alpha\int\log\left(\frac{p_\alpha(x)}{p_\beta(x)}\right)p_\alpha(x)dx - m^2_\alpha + m^2_\beta\\
    &= 2m_\alpha[\KL(\alpha\|\beta)+m_\alpha-m_\beta]-m^2_\alpha + m^2_\beta\\
    &= 2m_\alpha\KL(\alpha\|\beta) + (m_\alpha-m_\beta)^2.
\end{align*}
\textit{Part (ii)}
\begin{align*}
    \KLD(t\alpha \| r \beta) 
    &= \iint  \log \left( \frac{t^2 p_\alpha(x) p_\alpha(x')}{r^2 p_\beta(x) p_\beta(x')} \right) t^2 p_\alpha(x) p_\alpha(x') \dd x \dd x' + r^2 m_\beta^2 - t^2 m_\alpha^2 \\
    &= t^2 \left[ \iint \log \left( \frac{p_\alpha(x) p_\alpha(x')}{p_\beta(x) p_\beta(x')} \right) p_\alpha(x) p_\alpha(x') \dd x \dd x' + \log \left(\frac{t^2}{r^2}\right) m_\alpha^2 \right] + r^2 m_\beta^2 - t^2 m_\alpha^2 \\
    &= t^2 \left[ \iint \log \left( \frac{p_\alpha(x) p_\alpha(x')}{p_\beta(x) p_\beta(x')} \right) p_\alpha(x) p_\alpha(x') \dd x \dd x' + m_\beta^2 - m_\alpha^2 \right]  + t^2 \log\left(\frac{t^2}{r^2}\right) m_\alpha^2 + (r^2 - t^2) m_\beta^2 \\
    &= t^2 \KLD(\alpha \| \beta) + t^2 \log\left(\frac{t^2}{r^2}\right) m_\alpha^2 + (r^2 - t^2) m_\beta^2.
\end{align*}
As a consequence, we obtain the conclusion of the lemma.
\end{proof}
\begin{lemma}[KL divergence between Gaussians]
\label{lemma:KL_calculate}
The Kullback-Leibler divergence between Gaussian measures $\alpha=\Nn(\mu_\alpha, \Sigma_\alpha)$ and $\beta=\Nn(\mu_\beta, \Sigma_\beta)$ on $\RR^{k}$ is given by
\begin{align*}
    \KL\big(\alpha\|\beta \big) = \frac{1}{2} \left[ \tr(\Sigma_\beta^{-1} \Sigma_\alpha) + (\mu_\beta - \mu_\alpha)^\top \Sigma_\beta^{-1} (\mu_\beta - \mu_\alpha) - k + \log \left( \frac{\det (\Sigma_\beta)}{\det (\Sigma_\alpha)} \right) \right].
\end{align*}
As a result, when $\Sigma_\beta,\Sigma_\alpha$ are fixed, $\KL(\alpha\|\beta)$ achieves its minimum value when $\mu_\alpha = \mu_\beta$.
\end{lemma}
\begin{proof}[Proof of Lemma~\ref{lemma:KL_calculate}]
Let $p_\alpha(x)$ and $p_\beta(x)$ be the probability density functions of $\alpha$ and $\beta$, respectively. By the formula for KL divergence, we have
\begin{align}\label{eq:kl_formula}
    \KL(\alpha\|\beta)=\int_{\br^k}\log\left(\dfrac{p_\alpha(x)}{p_\beta(x)}\right)p_\alpha(x)dx,
\end{align}
in which
\begin{align*}
    \log(p_\alpha(x))&=-\dfrac{1}{2}[k\log(2\pi)+\log(\det(\Sigma_\alpha))+(x-\mu_\alpha)^{\top}\Sigma^{-1}_\alpha(x-\mu_\alpha)],\\
    \log(p_\beta(x))&=-\dfrac{1}{2}[k\log(2\pi)+\log(\det(\Sigma_\beta))+(x-\mu_\beta)^{\top}\Sigma^{-1}_\beta(x-\mu_\beta)],
\end{align*}
where $\pi$ denotes a constant number \textit{pi} only in this lemma. Plugging this result in equation~\eqref{eq:kl_formula}, we get
\begin{align*}
    \KL(\alpha\|\beta)&=\dfrac{1}{2}\left[\log\left(\dfrac{\det(\Sigma_\beta)}{\det(\Sigma_\alpha)}\right)+\int_{\br^k}[(x-\beta)^{\top}\Sigma^{-1}_\beta(x-\beta)-(x-\alpha)^{\top}\Sigma^{-1}_\alpha(x-\alpha)]p_\alpha(x)dx\right]\\
    &=\dfrac{1}{2}\left[\log\left(\dfrac{\det(\Sigma_\beta)}{\det(\Sigma_\alpha)}\right)+(\mu_\beta-\mu_\alpha)^{\top}\Sigma^{-1}_\beta(\mu_\beta-\mu_\alpha)+\int_{\br^k}[(x-\mu_\alpha)^{\top}(\Sigma^{-1}_\beta-\Sigma^{-1}_\alpha)(x-\mu_\alpha)]p_\alpha(x)dx\right].
\end{align*}
Note that
\begin{align*}
    \int_{\br^k}[(x-\mu_\alpha)^{\top}(\Sigma^{-1}_\beta-\Sigma^{-1}_\alpha)(x-\mu_\alpha)]p_\alpha(x)dx&=\int_{\br^k}\tr((x-\mu_\alpha)(x-\mu_\alpha)^{\top}(\Sigma^{-1}_\beta-\Sigma^{-1}_\alpha))p_\alpha(x)dx\\
    &=\tr\left(\int_{\br^k}[(x-\mu_\alpha)(x-\mu_\alpha)^{\top}(\Sigma^{-1}_\beta-\Sigma^{-1}_\alpha)]p_\alpha(x)dx\right)\\
    &=\tr\left(\Sigma_\alpha(\Sigma^{-1}_\beta-\Sigma^{-1}_\alpha)\right)\\
    &=\tr(\Sigma_\alpha\Sigma^{-1}_\beta)-k.
\end{align*}
Putting the above results together, we obtain the conclusion of this lemma.
\end{proof}
\begin{lemma}[Order solutions] \label{lemma:order_eigenvalues}
Let $\{a_i\}_{i=1}^m$ and $\{b_j\}_{j=1}^m$ be positive   decreasing sequences. Let $\{\widetilde{x}_i\}_{i=1}^m$ and $\{\widetilde{y}_j\}_{j=1}^n$ be the minimizer of 
\begin{align*}
    \min_{x_i> 0, y_j >0}~\mathsf{D}\big(\{x_i\},\{y_j\}; \{a_i\}, \{b_j\} \big):=\sum_{i=1}^m x_i^2 + \sum_{j=1}^n y_j^2+ &(\tau+\varepsilon)\Big\{\sum_{i=1}^m \big[\frac{x_i}{a_i} -  \log(x_i)\big] + \sum_{j=1}^n \big[\frac{y_j}{b_j} -\log(y_j)\big]\Big\} \\
    &- 2\sum_{k=1}^n \big[x_k y_k - \frac{\varepsilon}{2}\big]^+ + \varepsilon \big[\log(x_ky_k) - \log \frac{\varepsilon}{2} \big]^+.
\end{align*}
Then $(\widetilde{x}_i)$ and $(\widetilde{y}_j)$ are also decreasing sequences.
\end{lemma} 
\begin{proof}[Proof of Lemma~\ref{lemma:order_eigenvalues}]
We have two proofs for this lemma. While the first proof will be presented here, the second one can be found in Appendix~\ref{sec:another_proof}.

Let us consider all permutations $(\widehat{x}_i)$ and $(\widehat{y}_j)$ of $(\widetilde{x}_i)$ and $(\widetilde{y}_j)$, respectively. The difference between
$\mathsf{D}\big( \{\widehat{x}_i\}, \{\widehat{y}_j\}; \{a_i\}, \{b_j\}\big)$
and $\mathsf{D}\big(\{\widetilde{x}_i \},\{\widetilde{y}_j\}, \{a_i\}, \{b_j\} \big)$ lies in the term
\begin{align*}
  -2 \sum_{k=1}^n \Big[x_k y_k- \frac{\varepsilon}{2}\Big]^+ + \varepsilon\Big[ \log(x_ky_k) - \log \frac{\varepsilon}{2}\Big]^+,
\end{align*}
which is actually the minimum value of
\begin{align*}
    -2 \tr\big(K_{xy}^{\top} K_{xy}\big) + 2\varepsilon\KL\big(\overline{\pi}\|\overline{\pi}_x \otimes \overline{\pi}_y\big),
\end{align*}
where its solution has to match order by magnitude   $\{\widetilde{x}_i\}$ to  $\{\widehat{x}_i\}$ and $\{ \widetilde{y}_j \}$ to $\{\widehat{y}_j \}$
of the largest eigenvalues. The part $\sum_{i=1}^m \frac{x_i}{a_i} + \sum_{j=1}^n \frac{y_j}{b_j}$ is minimized by the rearrangement inequality (cf. Lemma~\ref{lemma:compare_sorted_convolution}).  Therefore, we obtain the conclusion of the lemma. 
\end{proof}
\begin{lemma}\label{lemma:minimization_square_log}
Let $\{a_i\}_{i=1}^s $ and $\{b_i \}_{i=1}^s$ be positive sequences. The problem of  finding the maximum value of 
\begin{align}\label{eq:square_log}
    \Big\{\sum_{i=1}^s a_i x_i\Big\}^2 + \sum_{i=1}^s \frac{b_i}{2}\log(1-x_i), 
\end{align}
where $x_i\in [0,1]$, is equivalent to finding the set $S\subset [s]$ in the below problem:
\begin{align*}
    \max_{S\subset [s]} &\Big\{A_S - \frac{B_S}{2\big[A_S+ (A_S^2 -B_S)^{\frac{1}{2}}\big]} \Big\}^2 + \sum_{i\in S} b_i \log(b_i) - B_S\log\big(A_S + (A_S^2-B_S)^{\frac{1}{2}}\big),
\end{align*}
where $A_S = \sum_{i\in S} a_i$, $B_S = \sum_{i\in S} b_i$ and 
$A_S^2 \geq B_S$ and 
\begin{align*}
    b_i \leq 2a_i \big\{A_S + (A_S^2 - B_S)^{\frac{1}{2}} \big\}; \quad i\in S.
\end{align*}
\end{lemma}
\begin{proof}[Proof of Lemma \ref{lemma:minimization_square_log}] 
Taking the derivative of function $f(x) = (ax+b)^2 + c \log(1-x)$ in $[0,1]$, we have
\begin{align*}
    f^{\prime}(x) = 2a(ax+b) - \frac{c}{1-x}.
\end{align*}
The equation $f^{\prime}(x) = 0$ has either two solutions or no solution. It means that either $f(x)$ is monotone or it has one minimum and one maximum. When $x\rightarrow 1^{-}$, $f(x)$ diverges to infinity. Thus, the $x_{\max}$ where function $f$ attains its maximum is closer to $1$ than the $x_{\min}$, where $f$ attains its minimum. In the second case, $f$ is monotone, then $f$ is monotone decreasing, because $f(0) = 0$ and $f(1^{-}) = - \infty$. Overall, 
\begin{align*}
    x_{\max} = \begin{cases}
     0 \\
     \text{the larger solution of equation:} \quad 2a(ax+b) = \frac{c}{1-x}.
    \end{cases}
\end{align*}
Hence, if $\{\widetilde{x}_i\}$ is a maximizer of the objective function, then either $\widetilde{x}_i=0$ or $\widetilde{x}_i$ is the larger solution of the first derivative system of equations. 

We consider the case that all $\widetilde{x}_i$ are not equal to zero. By taking the derivatives of the function in equation~\eqref{eq:square_log} with respect to $x_i$, we obtain
\begin{align}
    2 a_i \Big\{\sum_{j=1}^s a_j \widetilde{x}_j \Big\} - \frac{b_i}{2} \frac{1}{1-\widetilde{x}_i} &= 0, \nonumber\\
    4\big\{a_i - a_i \widetilde{x}_i\big\} \times \Big\{\sum_{j=1}^s a_j \widetilde{x}_j \Big\} &= b_i, \label{eq:for_substitute} \\
    4\Big\{\sum_{j=1}^s a_j - \sum_{j=1}^s a_j, \widetilde{x}_j \Big\}\times \Big\{\sum_{j=1}^s a_j \widetilde{x}_j \Big\} &= \sum_{j=1}^s b_j.\nonumber
\end{align}
Denote $A = \sum_{j=1}^s a_j $ and $B = \sum_{j=1}^s b_j$. Then, by solving the quadratic equation $4X^2-4AX+B=0$, we have
\begin{align*}
    \sum_{j=1}^s a_j \widetilde{x}_j = \frac{1}{2} \big\{A + \sqrt{A^2-B} \big\}.
\end{align*}
Substituting it into equation~\eqref{eq:for_substitute}, we obtain
\begin{align*}
    4a_i(1 -  \widetilde{x}_i) &= \frac{2b_i}{A + \sqrt{A^2-B}}\\
    \widetilde{x}_i &= 1 - \frac{b_i}{2a_i (A + \sqrt{A^2-B})}.
\end{align*}
Then, the minimum value of the objective function in this case equals to
\begin{align*}
   \Big\{ \sum_{i=1}^s a_i - \frac{1}{2(A + \sqrt{A^2-B})}\sum_{i=1}^s b_i \Big\}^2 + \sum_{i=1}^s b_i \log(b_i) - \log\big(A+ \sqrt{A^2-B}\big)\sum_{i=1}^s b_i. 
\end{align*}
We have thus proved the stated result.
\end{proof}

\begin{lemma}[Optimizers of some functions]\label{lemma:maximizer_linear_log}
\text{}\\
(a) For positive constants $a$ and $c$, the function $f(x) = ax + c \log(1-x)$ for $x\in [0,1)$ attains its minimum at $x^*=\Big[1- \frac{c}{a} \Big]^+$.\\
(b) For positive constants $a$ and $c$, the function $f(x) = 2ax^{\frac{1}{2}} + c\log(1-x)$ for $x\in [0,1)$ attains its maximizer at $x^*=\Big[\sqrt{\frac{c^2}{4a^2}+1} - \frac{c}{2a}\Big]^2$.\\
(c) For positive constants $a,b,c$ and $\Upsilon$, the function $f(x) = \Upsilon x +\tau \KL(x\|a) + \tau \KL(x\|b) + \varepsilon \KL(x\|c)$ attains its minimum at $x^* = a^{\frac{\tau}{2\tau + \varepsilon}} b^{\frac{\tau}{2\tau + \varepsilon}}  c^{\frac{\varepsilon}{2\tau + \varepsilon}}\exp\Big\{\frac{-\Upsilon}{2\tau + \varepsilon} \Big\}$.
\end{lemma}
\begin{proof}[Proof of Lemma~\ref{lemma:maximizer_linear_log}]
\text{}

For part (a), we have $f^{\prime}(x) = a - \frac{c}{1-x}$ which is a decreasing function when $x\in [0,1)$. It means that the function is convex. Hence its maximum attains at the equation $f^{\prime}(x) = 0$ or at the boundary. In fact, we have $f^{\prime}(1-c/a) = 0$. Therefore, the maximizer attains at  $\max(0,1-c/a)$.

For part (b), we have $f^{\prime}(x) = ax^{-\frac{1}{2}} - \frac{c}{1-x}$, which is also a decreasing function. Hence, $f$ is concave, it means that the maximum attains at the solution of $f^{\prime}(x) = 0$ or the boundary. Solving the equation $f^{\prime}(x) =0$ we obtain $x^* = \Big[\sqrt{\frac{c^2}{4a^2}+1} - \frac{c}{2a}\Big]^2$, which belongs to $(0,1)$. 

For part (c), by taking the derivative of $f(x)$, we have
\begin{align*}
    f^{\prime}(x) = \Upsilon + \tau \big\{2\log(x) - \log(ab) \big\} + \varepsilon\big\{\log(x) -\log(c) \big\}
\end{align*}
Solving the equation $f^{\prime}(x) = 0$, we obtain
\begin{align*}
    x^* = \exp\Big\{\frac{\tau \log(a) + \tau \log(b) + \varepsilon \log(c) - \Upsilon}{2\tau + \varepsilon} \Big\} = a^{\frac{\tau}{2\tau + \varepsilon}} b^{\frac{\tau}{2\tau + \varepsilon}}  c^{\frac{\varepsilon}{2\tau + \varepsilon}}\exp\Big\{\frac{-\Upsilon}{2\tau + \varepsilon} \Big\}.
\end{align*}
\end{proof}
\begin{lemma}\label{lemma:expectation_factorize}
Let $(X,Y)$ and $(X^{\prime},Y^{\prime})$ be two i.i.d random vectors in $\br^m\times\br^n$ and follow a probability distribution $\pi$ which have a mean vector $(u,v)\in\br^{m\times n}$ and a covariance matrix $\begin{pmatrix}
\Sigma_x & \Sigma_{xy}\\
\Sigma^{\top}_{xy} & \Sigma_y
\end{pmatrix}$ where $\Sigma_x$ and $\Sigma_y$. Denote $Z=X-u,Z^{\prime}=X^{\prime}-u$ and $T=Y-v,T^{\prime}=Y^{\prime}-v$, then
\begin{align*}
    \E_{\pi\otimes\pi}\Big[\langle X,X^{\prime}\rangle-\langle Y,Y^{\prime}\rangle\Big]^2 = \E_{\pi\otimes\pi}\Big[\langle Z,Z^{\prime}\rangle-\langle T,T^{\prime}\rangle\Big]^2+2u^{\top}\Sigma_x u + 2u^{\top}\Sigma_{xy}v+2v^{\top}\Sigma_y v + \Big[\|u\|^2-\|v\|^2\Big]^2.
\end{align*}
\begin{proof}[Proof of Lemma~\ref{lemma:expectation_factorize}]
Firstly, we have
\begin{align*}
    \langle X ,X^{\prime} \rangle - \langle Y, Y^{\prime} \rangle &=\langle Z+u ,Z^{\prime} + u \rangle - \langle T+ v, T^{\prime} + v\rangle\\
    &= \langle Z, Z^{\prime}\rangle - \langle T, T^{\prime}\rangle + \langle u, Z+Z^{\prime}\rangle - \langle v, T+ T^{\prime} \rangle + \|u\|^2 - \|v\|^2.
\end{align*}
Note that
\begin{align*}
    \mathbb{E}_{\pi\otimes\pi} \langle u, Z\rangle \langle Z, Z^{\prime}\rangle &= \mathbb{E}_{\pi\otimes\pi}\Big[\sum_{i=1}^m Z_i Z_i^{\prime} \sum_{j=1}^m  u_j Z_j\Big] = \sum_{i=1,j=1}^m u_j \mathbb{E}_{\pi\otimes\pi}\big[Z_i Z_i^{\prime} Z_j \big] = \sum_{i=1,j=1}^m u_j \mathbb{E}_{\pi\otimes\pi}\big[Z_i Z_j \big]\mathbb{E}_{\pi\otimes\pi}[Z_i^{\prime}] = 0\\
    \mathbb{E}_{\pi\otimes\pi}\langle Z, Z^{\prime}\rangle &= \mathbb{E}_{\pi\otimes\pi}\Big[\sum_{i=1}^m Z_i Z_i^{\prime} \Big]= 0 \\
    \mathbb{E}_{\pi\otimes\pi} \langle Z,Z^{\prime}\rangle \langle v,T\rangle  &=\mathbb{E}_{\pi\otimes\pi}\Big[\sum_{i=1}^m Z_i Z_i^{\prime} \sum_{j=1}^n v_j T_j \Big] = \sum_{i=1}^m \sum_{j=1}^n v_j\mathbb{E}_{\pi\otimes\pi}\big[Z_i Z_i^{\prime} T_j \big] = \sum_{i=1}^m\sum_{j=1}^n \mathbb{E}_{\pi\otimes\pi}[Z_i T_j]\mathbb{E}_{\pi\otimes\pi}[Z_i^{\prime}]=0.
\end{align*}
Similarly, we get
\begin{align*}
    \mathbb{E}_{\pi\otimes\pi}\langle v,T\rangle \langle T,T^{\prime}\rangle = \mathbb{E}_{\pi\otimes\pi} \langle T,T^{\prime}\rangle  = \mathbb{E}_{\pi\otimes\pi}\langle T,T^{\prime}\rangle \langle u,Z\rangle = 0.
\end{align*}
As a consequence, we can deduce that
\begin{align*}
    \mathbb{E}_{\pi\otimes\pi}\Big[\langle X ,X^{\prime} \rangle - \langle Y, Y^{\prime}\rangle\Big]^2 = \mathbb{E}_{\pi\otimes\pi}\Big[\langle Z, Z^{\prime} \rangle - \langle T, T^{\prime} \rangle \Big]^2 + \mathbb{E}_{\pi\otimes\pi} \Big[\langle u, Z + Z^{\prime} \rangle - \langle v, T + T^{\prime} \rangle\Big]^2 + \Big[\|u\|^2 - \|v\|^2\Big]^2.
\end{align*}
Next, it is sufficient to show that
\begin{align*}
    \E_{\pi\otimes\pi}\Big[\langle u,Z+Z^{\prime}\rangle - \langle v,T+T^{\prime}\rangle \Big]^2 = 2u^{\top}\Sigma_x u + 2v^{\top}\Sigma_y v+ 2u^{\top}\Sigma_{xy}v.
\end{align*}
Indeed, we have
\begin{align*}
    \E_{\pi\otimes\pi}\Big[\langle u, Z+Z^{\prime}\rangle\Big]^2 &=\E_{\pi\otimes\pi}\Big[u^{\top}(Z+Z^{\prime})(Z+Z^{\prime})^{\top}u\Big] = u^{\top}\mathbb{V}ar(Z+Z^{\prime})u=2u^{\top}\Sigma_x u\\
    \E_{\pi\otimes\pi}\Big[\langle v, T+T^{\prime}\rangle\Big]^2 &=\E_{\pi\otimes\pi}\Big[v^{\top}(T+T^{\prime})(T+T^{\prime})^{\top}v\Big] = v^{\top}\mathbb{V}ar(T+T^{\prime})v=2v^{\top}\Sigma_y v\\
    \E_{\pi\otimes\pi}\Big[\langle u,Z+Z^{\prime}\rangle\langle v,T+T^{\prime}\rangle\Big] &= u^{\top}\E_{\pi\otimes\pi}\Big[(Z+Z^{\prime})(T+T^{\prime})^{\top}\Big]v = u^{\top}\E_{\pi\otimes\pi}\Big[ZT^{\top}+Z^{\prime}(T^{\prime})^{\top}\Big]v=2u^{\top}\Sigma_{xy}v.
\end{align*}
Hence, the proof is completed.
\end{proof}
\end{lemma}
\section{Another proof of Lemma~\ref{lemma:order_eigenvalues}}
\label{sec:another_proof}
In this appendix, we introduce another proof of Lemma \ref{lemma:order_eigenvalues} using Karamata's inequality. Firstly, we introduce an inequality which can be considered to be a generalization of Karamata's inequality in a special case. 

\begin{lemma}[A quasi-Karamata inequality]
\label{lemma:quasi_karamata}
Let $a_1\geq a_2\geq \ldots\geq a_n\geq 0$ and $b_1\geq b_2\geq\ldots\geq b_n \geq 0$ which satisfy
\begin{equation*}
    \begin{split}
        a_1&\geq b_1\\
        a_1+a_2&\geq b_1+b_2\\
        &\ldots\\
        a_1+a_2+\ldots+a_{n-1} &\geq b_1+b_2+\ldots+b_{n-1}\\
        a_1+a_2+\ldots+a_{n} &\geq b_1+b_2+\ldots+b_{n}
    \end{split}
\end{equation*}
Let $f$ be a concave and non-increasing function in $[0,+\infty)$. Then, we have
\begin{equation*}
    f(a_1)+f(a_2)+\ldots+f(a_n) \leq f(b_1)+f(b_2)+\ldots+f(b_n).
\end{equation*}
\end{lemma}
\begin{proof}[Proof of Lemma \ref{lemma:quasi_karamata}]
We prove the Lemma by using induction. The Lemma is obvious when $n=1$. 

Suppose that the problem is true for $n-1$. For $i \in \{1,2,\ldots,n\}$, let 
\begin{equation*}
c_i = a_1+a_2+\ldots+a_i - b_1-b_2-\ldots-b_i.
\end{equation*}
Let $k$ ($1\leq k \leq n$) be the smallest index such that $c_k = \min_{1\leq i\leq n} c_i$. Let $b'_1 = b_1+c_k$, then, we have $b'_1\geq b_2\geq \ldots\geq b_n$. We consider two cases:
\begin{itemize}
    \item If $k = n$, we have 
    \begin{equation*}
        \begin{split}
            a_1&\geq b'_1\\
            a_1+a_2&\geq b'_1+b_2\\
            &\ldots\\
            a_1+a_2+\ldots+a_n&= b'_1+b_2+\ldots+b_n
        \end{split}
    \end{equation*}
    Applying Karamata's inequality, we have $$f(a_1)+f(a_2)+\ldots+f(a_n)\leq f(b'_1)+f(b_2)+\ldots +f(b_n) \leq f(b_1)+f(b_2)+\ldots+f(b_n)$$.
    \item If $k < n$. We have
    \begin{equation*}
        \begin{split}
            a_1&\geq b'_1\\
            a_1+a_2&\geq b'_1+b_2\\
            &\ldots\\
            a_1+a_2+\ldots+a_k&= b'_1+b_2+\ldots+b_k.
        \end{split}
    \end{equation*}
    Applying Karamata's inequality, we have $$f(a_1)+f(a_2)+\ldots+f(a_k)\leq f(b'_1)+f(b_2)+\ldots +f(b_n) \leq f(b_1)+f(b_2)+\ldots+f(b_k)$$.
    
    Moreover, when $i > k$, we have
    \begin{equation*}
        a_{k+1}+\ldots+a_i - b_{k+1} -\ldots-b_i = c_i-c_k \geq 0.
    \end{equation*}
    Applying the inductive assumption for two arrays $a_{k+1},\ldots,a_n$ and $b_{k+1},\ldots,b_{n}$, we have
    \begin{equation*}
        f(a_{k+1})+f(a_{k+2})+\ldots+f(a_n) \leq f(b_{k+1})+f(b_{k+2})+\ldots+f(b_n).
    \end{equation*}
    which implies 
    \begin{equation*}
    f(a_1)+f(a_2)+\ldots+f(a_n) \leq f(b_1)+f(b_2)+\ldots+f(b_n).
\end{equation*}
\end{itemize}
As a consequence, we obtain the conclusion of the Lemma.
\end{proof}

\begin{lemma}[Extension of rearrangement inequality]
\label{lemma:compare_sorted_convolution}
Let $a_1\geq a_2\ldots\geq a_n\geq 0$, $b_1\geq b_2\geq \ldots\geq b_n\geq 0$ be two decreasing array. Consider $\alpha$ and $\beta$ are two injections from the set $\{1,2,\ldots,k\}$ to $\{1,2\ldots,n\}$ in the set $\{1,2,\ldots,n\}$, where $k$ is a positive integer such that $k\leq n$, we have 
\begin{equation*}
    \max a_{\alpha(1)}b_{\beta(1)}+\ldots+a_{\alpha(k)}b_{\beta(k)} = a_1b_1+a_2b_2+\ldots+a_kb_k,
\end{equation*}
and the maximum value is achieved when $\alpha,\beta$ are two permutation which satisfy
\begin{equation*}
    a_{\alpha(i)} = a_i, \quad b_{\beta(i)} = b_i \quad \forall i:  1\leq i \leq k.
\end{equation*}
\end{lemma}
\begin{proof}[Proof of Lemma \ref{lemma:compare_sorted_convolution}]
Let $I = \{i_1,i_2,\ldots,i_t\}$ be the indices which are not greater than $k$ such that $\alpha(i_r) > k,\forall r: 1 \leq r \leq t$. Let $J = \{j_1,j_2\ldots,j_t\}$ be the indices which are not greater than $k$ such that $\alpha(r)\notin J, \forall r: 1\leq r\leq k$. Consider the permutation $\Tilde{\alpha}:\{1,2,\ldots,k\} \to \{1,2,\ldots,k\}$ which is defined by 
\begin{equation*}
    \tilde{\alpha}(j) = \begin{cases}
    \alpha(j),\quad &\text{if } j \notin I\\
    j_r \quad & \text{if } j = i_r
    \end{cases}.
\end{equation*}
It is obviously that
\begin{equation*}
    a_{\alpha(1)}b_{\beta(1)}+\ldots+a_{\alpha(k)}b_{\beta(k)} \leq a_{\tilde{\alpha}(1)}b_{\beta(1)}+\ldots+a_{\tilde{\alpha}(k)}b_{\beta(k)}.
\end{equation*}
Similarly, there exist a permutation $\tilde{\beta}:\{1,2,\ldots,k\} \to \{1,2,\ldots,k\}$ such that 
\begin{equation*}
    a_{\alpha(1)}b_{\beta(1)}+\ldots+a_{\alpha(k)}b_{\beta(k)} \leq a_{\tilde{\alpha}(1)}b_{\beta(1)}+\ldots+a_{\tilde{\alpha}(k)}b_{\beta(k)} \leq a_{\tilde{\alpha}(1)}b_{\tilde{\beta}(1)}+\ldots+a_{\tilde{\alpha}(k)}b_{\tilde{\beta}(k)}.
\end{equation*}
Applying the rearrangement inequality, we have
\begin{equation*}
    a_{\tilde{\alpha}(1)}b_{\tilde{\beta}(1)}+\ldots+a_{\tilde{\alpha}(k)}b_{\tilde{\beta}(k)} \leq a_1b_1+a_2b_2+\ldots+a_kb_k,
\end{equation*}
which implies
\begin{equation*}
    a_{\alpha(1)}b_{\beta(1)}+\ldots+a_{\alpha(k)}b_{\beta(k)} \leq a_1b_1+a_2b_2+\ldots+a_kb_k.
\end{equation*}
From this, we can conclude our proof.
\end{proof}
Given the above lemmas, we have another proof of Lemma \ref{lemma:order_eigenvalues}.
\begin{proof}[Another Proof of Lemma~\ref{lemma:order_eigenvalues}]
Let us consider a permutation $(\widehat{x}_i)$ and $(\widehat{y}_i)$ of $(\tilde{x}_i)$ and $(\tilde{y}_i)$, respectively, such that $(\widehat{x}_i)$ and $(\widehat{y}_i)$ are decreasing sequence. The difference between
$\mathsf{D}\big( \{\widehat{x}_i\}, \{\widehat{y}_j\}; \{a_i\}, \{b_j\}\big)$
and $\mathsf{D}\big(\{\widetilde{x}_i \},\{\widetilde{y}_j\}, \{a_i\}, \{b_j\} \big)$ lies in the part
\begin{align*}
  -2 \sum_{k=1} \Big[x_k y_k- \frac{\varepsilon}{2}\Big]^+ + \varepsilon\Big[ \log(x_ky_k) - \log \frac{\varepsilon}{2}\Big]^+.
\end{align*}
Let $f(x) = -2\left(x-\dfrac{\varepsilon}{2}\right)^++\varepsilon\left(\log(x) -\log\left(\dfrac{\varepsilon}{2}\right)\right)$. This function has non-positive and decreasing derivative
\begin{equation*}
    f'(x) = \begin{cases}
     0, \quad &\text{when } x \leq \frac{\varepsilon}{2}\\
     \dfrac{\varepsilon}{x} - 2,\quad &\text{when } x \geq \frac{\varepsilon}{2}
    \end{cases}.
\end{equation*}
Thus, $f$ is concave function in $[0.\infty)$. Moreover, according to Lemma \ref{lemma:compare_sorted_convolution}, we have $(\widehat{x}_i\widehat{y}_i)$ and $(\tilde{x}_i\tilde{y}_i))$ satisfy the condition of Lemma \ref{lemma:quasi_karamata}. Thus, applying this Lemma with the function $f$, we have the term
\begin{align*}
  -2 \sum_{k=1} \Big[x_k y_k- \frac{\varepsilon}{2}\Big]^+ + \varepsilon\Big[ \log(x_ky_k) - \log \frac{\varepsilon}{2}\Big]^+.
\end{align*}
is minimized when $x_k$ and $y_k$ are decreasing sequence. The part $\sum_{i=1}^m \frac{x_i}{a_i} + \sum_{j=1}^n \frac{y_j}{b_j}$ is minimized by the rearrangement inequality.  Therefore, we obtain the conclusion of the lemma.  
\end{proof}
\section{Entropic $\mathsf{IGW}$ between Gaussian distributions with non-zero means}
\label{sec:nonzero_means}
In this appendix, by considering the problem of solving entropic Gromov-Wasserstein between balanced Gaussians in a special case, we show the difficulty to solve this problem in the general case. 

Let $m$ be a positive integer. Let $\mu = \mathcal{N}(\theta_\mu,\Sigma_\mu)$ and $\nu = \mathcal{N}(1,1)$ be two Gaussian measures in $\mathbb{R}^m$ and $\mathbb{R}$, respectively, where
\begin{align*}
    \Sigma_{\mu} &= \diag\big(a_1^2,\ldots, a_m^2 \big),\\
    \theta_\mu &= (b_1, \ldots, b_m)^{\top}.
\end{align*}
To find $\mathsf{IGW}_\varepsilon(\mu,\nu)$, we must find the joint distribution $\pi$ between $\mu$ and $\nu$ which has the covariance matrix
\begin{align*}
    \Sigma_{\pi} = \begin{pmatrix}\Sigma_{\mu} & K_{\mu\nu} \\
    K_{\mu\nu}^{\top} & 1
    \end{pmatrix} \label{eq:plan_structure_balanced}
\end{align*}
that minimizes the following term
\begin{equation*}
    \E_{\pi\otimes\pi} \Big\{ \big[ \langle X,X^{\prime}\rangle - \langle Y,Y^{\prime}\rangle \big]^2 \Big\} + \varepsilon \KL\big(\pi \| \mu\otimes \nu \big).
\end{equation*}
Note that, in this case, $K_{\mu\nu}$ is an $m\times 1$ matrix, thus, we can denote that $K_{\mu\nu}=(x_1, \ldots, x_m)^\top$.
Firstly, we prove the following equality
\begin{align}
\E_{\pi\otimes\pi} \Big\{ \big[ \langle X,X^{\prime}\rangle - \langle Y,Y^{\prime}\rangle \big]^2 \Big\}  = \sum_{i=1}^ma_i^2+\sum_{i=1}^mb_i^2 + 2 - 2 \sum_{i=1}^m(x_i+b_i)^2. \label{eq:first_equality_balanced_non_zero_mean}
\end{align}
Indeed, we have
\begin{align*}
    \E_{\pi\otimes\pi} \Big\{ \big[ \langle X,X^{\prime}\rangle - \langle Y,Y^{\prime}\rangle \big]^2 \Big\}  = \E_{\pi\otimes\pi} \langle X,X^{\prime}\rangle^2 + \E_{\pi\otimes\pi} \langle Y, Y^{\prime}\rangle^2 - 2 \E_{\pi\otimes\pi} \langle X, X^{\prime}\rangle \langle Y, Y^{\prime}\rangle.
\end{align*}
Note that the independence between $X$ and $X^{\prime}$ and $Y$ and $Y^{\prime}$ leads to
\begin{align*}
    \E_{\pi\otimes\pi} \langle X,X^{\prime}\rangle^2 &=[\mathbb{E}_{\pi\otimes\pi}\langle X,X^\prime\rangle]^2+\text{Var}_{\pi}\langle X,X^\prime\rangle\\
    &=[\langle \theta_\mu, \theta_\mu\rangle]^2+\sum_{i=1}^m\text{Var}_{\pi}[X_i]\text{Var}_{\pi}[X^\prime_i]\\ 
    &=\left(\sum_{i=1}^mb_i^2\right)^2+\sum_{i=1}^ma_i^4.
\end{align*}
Similarly, we have $ \E_{\pi\otimes\pi} \langle Y,Y^{\prime}\rangle^2 = 2$.
Meanwhile,
\begin{align*}
    \E_{\pi\otimes\pi} \langle X, X^{\prime}\rangle \langle Y, Y^{\prime}\rangle &= \E_{\pi\otimes\pi} \Big\{\sum_{i=1}^{m} \big[X_iX^\prime_{i}YY^{\prime}\big]\Big\} = \sum_{i=1}^m \E_{\pi\otimes\pi}\{X_i Y\}\E_{\pi}\{X^\prime_i Y'\} = \sum_{i=1}^m(x_i+b_i)^2.
\end{align*}
Putting the above results together, we obtain the desired equality~\eqref{eq:first_equality_balanced_non_zero_mean}. By Lemma~\ref{lemma:min_KL_Gaussian}, the optimal transportation plan $\pi^{*}$ is a Gaussian distribution, namely, $\pi^*=\mathcal{N}(\theta_{\pi^*},\Sigma_{\pi^*})$, where $\theta_{\pi^*}=(b_1, \ldots, b_n, 1)^\top$. Thus, according to Lemma \ref{lemma:KL_calculate}, the entropic term in the objective function reads
\begin{align*}
   \varepsilon \KL(\pi^{*}\|\mu \otimes \nu)& = \frac{1}{2}\varepsilon\Big\{\tr\big(\Sigma_{\pi^{*}} \Sigma^{-1}_{\mu\otimes \nu}\big) - (m+n) + \log\Big( \frac{\det(\Sigma_{\mu\otimes \nu})}{\det(\Sigma_{\pi^{*}})}\Big) \Big\} \\
   & = -\frac{\varepsilon}{2}\log\left(1-\sum_{i=1}^m\dfrac{x_i^2}{a_i^2}\right).
\end{align*}
From the above calculation, our problem turns out to find the non-negative real numbers $x_1, \ldots, x_m$ which maximize the following function
\begin{equation*}
    \psi(x_1,\ldots,x_m) = 2\sum_{i=1}^m(x_i+b_i)^2+\frac{\varepsilon}{2} \log\left(1-\sum_{i=1}^m\dfrac{x_i^2}{a_i^2}\right).
\end{equation*}
Taking the derivatives of the above function with respect to $x_i$, $i\in [m]$, we get
\begin{equation*}
    \frac{\partial \psi}{\partial x_i} = 4(x_i+b_i)-\dfrac{\varepsilon x_i/a^2_i}{1-\sum_{i=1}^m\frac{x_i^2}{a_i^2}}, \quad \forall i\in[m].
\end{equation*}
To solve the system of equations $\dfrac{\partial\psi}{\partial x_i}=0$ for all $i\in[m]$, we denote $t=\dfrac{\varepsilon}{1-\sum_{i=1}^m\frac{x_i^2}{a_i^2}}$. 
Then, $t = \dfrac{4(x_i+b_i)}{a_i^2x_i}$, which implies that $x_i = \dfrac{4a_i^2b_i}{t-4a_i^2}$. By substituting this representation of $x_i$ to the expression $t=\dfrac{\varepsilon}{1-\sum_{i=1}^m\frac{x_i^2}{a_i^2}}$, we obtain
\begin{equation*}
    t\left(1-\sum_{i=1}^m \dfrac{16a_i^2b_i^2}{(t-4a_i^2)^2}\right)=\varepsilon.
\end{equation*}
After some transformations, we achieve the equation of degree $2m$, which is difficult to solve in close-form.

\section{Revisiting the entropic optimal transport between Gaussian measures}
\label{sec:revisit_entropic_Gaussian}
In this appendix, via the technique that we use to study the entropic Gromov Wasserstein between two Gaussians, we provide a simpler proof for deriving the close-form expressions of entropic optimal transport (OT) between (unbalanced) Gaussian measures than the proof that was presented in~\citep{janati2020entropic}. 
\subsection{Entropic optimal transport between balanced Gaussian measures}
Before moving to the main theorem of this section, we review some backgrounds on the entropic optimal transport problem between balanced Gaussian measures.
\paragraph{Entropic Optimal Transport.} Let $\alpha$ and $\beta$ be two probability measures in $\mathbb{R}^d$. Then, the entropic OT between them is defined as follows:
\begin{align}\label{eq:entropic_OT_formulation}
    \mathsf{OT}_\varepsilon(\alpha,\beta):= \min_{\pi\in\Pi(\alpha,\beta)}~\mathbb{E}_{\pi}\|X-Y\|^2 + \varepsilon\KL(\pi\|\alpha\otimes\beta),
\end{align}
where $X\sim\alpha$ and $Y\sim\beta$ are independent random vectors in $\br^d$ such that $(X,Y)\sim\pi$, and $\varepsilon>0$ denotes the entropy-regularized parameter.
Now, we are ready to show the closed-form expression of entropic OT between Gaussians in the following theorem.
\begin{theorem}\label{theorem:entropic_OT}
Let $\mu=\mathcal{N}(\uU,\Sigma_\mu)$ and $\nu=\mathcal{N}(\vV,\Sigma_\nu)$ be two Gaussian measures in $\mathbb{R}^d$, where $\Sigma_\mu$ and $\Sigma_\nu$ are positive definite matrices. Then, the entropic optimal transport between $\mu$ and $\nu$ is given by
\begin{align*}
    \mathsf{OT}_{\varepsilon}(\mu,\nu) = \|\uU-\vV\|_2^2+\tr(\Sigma_\mu)+\tr(\Sigma_\nu)-\tr(D_\varepsilon)+\frac{d\varepsilon}{2}(1-\log\varepsilon)+\log\det\left(D_\varepsilon+\frac{\varepsilon}{2}\id\right),
\end{align*}
where $D_\varepsilon=\left(4\Sigma^{\frac{1}{2}}_{\mu}\Sigma_\nu\Sigma^{\frac{1}{2}}_\nu+\frac{\varepsilon^2}{4}\id\right)^\frac{1}{2}$. Moreover, the optimal transportation plan admits the form $\pi^*=\mathcal{N}\big((\uU^*,\vV^*)^{\top},\Sigma_{\pi^*}\big)$ where
\begin{align*}
    \Sigma_{\pi^*}:=\begin{pmatrix}
     \Sigma_\mu & K^*_{\mu\nu}\\
     (K^*_{\mu\nu})^{\top} & \Sigma_\nu
    \end{pmatrix},
\end{align*}
with $K^*_{\mu\nu}=\frac{1}{2}\Sigma^\frac{1}{2}_\mu\left(\frac{\varepsilon^2}{4}~\id+4~\Sigma^\frac{1}{2}_\mu\Sigma_\nu\Sigma^\frac{1}{2}_\mu\right)^\frac{1}{2}\Sigma^{-\frac{1}{2}}_\mu-\frac{\varepsilon}{4}\id$.
\end{theorem}
\begin{remark}
Note that if we replace $\varepsilon$ by $2\sigma^2$ in the above result, we obtain Theorem 1 in \citep{janati2020entropic}.
\end{remark}
\begin{proof}[Proof of Theorem~\ref{theorem:entropic_OT}]
Note that if $\bar{\mu}$ and $\bar{\nu}$ are respective centered transformations of $\mu$ and $\nu$, then 
\begin{align*}
    \mathsf{OT}_{\varepsilon}(\mu,\nu)=\mathsf{OT}_{\varepsilon}(\bar{\mu},\bar{\nu})+\|\uU-\vV\|_2^2.
\end{align*}
Therefore, it is sufficient to solve the case when $\uU=\vV=\zeros_d$. Under this assumption, we have
\begin{align*}
    \mathbb{E}_{\pi}\|X-Y\|^2 &= \mathbb{E}_{\pi}\|X\|^2+\mathbb{E}_{\pi}\|Y\|^2 - 2\mathbb{E}_{\pi}\langle X,Y\rangle\\
    &= \tr(\Sigma_\mu) + \tr(\Sigma_\nu) - 2\tr(K_{\mu\nu}).
\end{align*}
Thus, the entropic optimal transport between $\mu$ and $\nu$ can be rewritten as
\begin{align}
    \label{eq:entropic_OT_1}
    \mathsf{OT}_{\varepsilon}(\mu,\nu)=\tr(\Sigma_\mu)+\tr(\Sigma_\nu)+\min_{\pi\in\Pi(\mu,\nu)}\Big\{\varepsilon\KL(\pi\|\mu\otimes\nu)-2\tr(K_{\mu\nu})\Big\}.
\end{align}
By fixing the covariance matrix $\Sigma_\pi=\begin{pmatrix}
\Sigma_\mu & K_{\mu\nu}\\
K^{\top}_{\mu\nu} & \Sigma_\nu
\end{pmatrix}$, Lemma~\ref{lemma:min_KL_Gaussian} indicates that the optimal solution must be a Gaussian measure with the form $\pi=\mathcal{N}(\zeros_{2d},\Sigma_\pi)$. Let $U_{\mu\nu}\Lambda^{\frac{1}{2}}_{\mu\nu}V^{\top}_{\mu\nu}$ be the SVD of matrix $\Sigma^{-\frac{1}{2}}_{\mu}K_{\mu\nu}\Sigma^{-\frac{1}{2}}_{\nu}$ with $\Lambda^{\frac{1}{2}}_{\mu\nu}=\diag(\kappa^{\frac{1}{2}}_{\mu\nu,i})_{i=1}^d$. It follows from Lemma~\ref{lemma:kl_divergence} and Lemma~\ref{lemma:entropic_gw:inner_product:closed_form:supporting} part (b) that
\begin{align*}
    \varepsilon\KL(\pi\|\mu\otimes\nu)&=\frac{\varepsilon}{2}\left[\tr(\Sigma_{\pi}\Sigma^{-1}_{\mu\otimes\nu})-2d+\log\left(\frac{\det(\Sigma_{\mu\otimes\nu})}{\det(\Sigma_{\pi})}\right)\right] = -\frac{\varepsilon}{2}\sum_{i=1}^d\log(1-\kappa_{\mu\nu,i}).
\end{align*}
Let $U\Lambda^\frac{1}{2} V^{\top}$ be the SVD of $\Sigma^{\frac{1}{2}}_\nu\Sigma^{\frac{1}{2}}_\mu$, where $\Lambda^{\frac{1}{2}}=\diag(\lambda^{\frac{1}{2}}_{\mu\nu,i})_{i=1}^d$. By using von Neumann's inequality \citep{kristof1969neumann,horn_johnson_1991}, we have
 \begin{align*}
     \tr(K_{\mu\nu})&=\tr(\Sigma^{\frac{1}{2}}_\mu U_{\mu\nu}\Lambda^{\frac{1}{2}}_{\mu\nu}V^{\top}_{\mu\nu}\Sigma^{\frac{1}{2}}_\nu)=\tr(U_{\mu\nu}\Lambda^{\frac{1}{2}}_{\mu\nu}V^{\top}_{\mu\nu}\Sigma^{\frac{1}{2}}_\nu\Sigma^{\frac{1}{2}}_\mu)= \tr([V^{\top}U_{\mu\nu}\Lambda^{\frac{1}{2}}_{\mu\nu}][V^{\top}_{\mu\nu}U\Lambda^{\frac{1}{2}}])\\
     &\leq \sum_{i=1}^d\lambda^{\frac{1}{2}}_{\mu\nu,i}\kappa^{\frac{1}{2}}_{\mu\nu,i}.
 \end{align*}
The equality occurs when $U_{\mu\nu}=V$ and $V_{\mu\nu}=U$. Collecting the above two results, the problem in equation \eqref{eq:entropic_OT_1} is reduced to
\begin{align*}
    \mathsf{OT}_{\varepsilon}(\mu,\nu)=\tr(\Sigma_\mu)+\tr(\Sigma_\nu)-\max_{\kappa_{\mu\nu,i}\in(0,1)}\sum_{i=1}^d\left[2\lambda^{\frac{1}{2}}_{\mu\nu,i}\kappa^{\frac{1}{2}}_{\mu\nu,i}+\frac{\varepsilon}{2}\log(1-\kappa_{\mu\nu})\right].
\end{align*}
Note that the function $f(x)=2ax^{\frac{1}{2}}+\frac{\varepsilon}{2}\log(1-x)$ attains its maximum at $x=\left(\frac{-\varepsilon+\sqrt{\varepsilon^2+16a^2}}{4a}\right)^2$. Then, the optimal solution of the above problem is $\kappa^*_{\mu\nu,i}=\frac{\left(-\varepsilon+\sqrt{\varepsilon^2+16\lambda_{\mu\nu,i}}\right)^2}{16\lambda_{\mu\nu,i}}$, which leads to 
\begin{align*}
    \mathsf{OT}_{\varepsilon}(\mu,\nu)&=\tr(\Sigma_\mu)+\tr(\Sigma_\nu)-\dfrac{1}{2}\sum_{i=1}^d\left[-\varepsilon+\sqrt{\varepsilon^2+16\lambda_{\mu\nu,i}}+\varepsilon\log\left(\dfrac{\varepsilon(-\varepsilon+\sqrt{\varepsilon^2+16\lambda_{\mu\nu,i}})}{8\lambda_{\mu\nu,i}}\right)\right]\\
    &=\tr(\Sigma_\mu)+\tr(\Sigma_\nu)+\dfrac{d\varepsilon}{2}(1-\log\varepsilon)-\dfrac{1}{2}\sum_{i=1}^d\sqrt{\varepsilon^2+16\lambda_{\mu\nu,i}}-\dfrac{\varepsilon}{2}\sum_{i=1}^d\log\left(\dfrac{-\varepsilon+\sqrt{\varepsilon^2+16\lambda_{\mu\nu,i}}}{8\lambda_{\mu\nu,i}}\right)\\
    &=\tr(\Sigma_\mu)+\tr(\Sigma_\nu)+\dfrac{d\varepsilon}{2}(1-\log\varepsilon)-\sum_{i=1}^d\sqrt{\frac{\varepsilon^2}{4}+4\lambda_{\mu\nu,i}}+\dfrac{\varepsilon}{2}\log\left(\prod_{i=1}^d\left[\dfrac{\varepsilon}{2}+\sqrt{\dfrac{\varepsilon^2}{4}+4\lambda_{\mu\nu,i}}~\right]\right)\\
    &=\tr(\Sigma_\mu)+\tr(\Sigma_\nu) +\dfrac{d\varepsilon}{2}(1-\log\varepsilon)-\tr\left(\Big[4\Sigma^{\frac{1}{2}}_{\mu}\Sigma_\nu\Sigma^{\frac{1}{2}}_\mu+\frac{\varepsilon^2}{4}\id\Big]^{\frac{1}{2}}\right)+\log\det\left(\Big[4\Sigma^{\frac{1}{2}}_{\mu}\Sigma_\nu\Sigma^{\frac{1}{2}}_\mu+\frac{\varepsilon^2}{4}\id\Big]^{\frac{1}{2}}+\frac{\varepsilon}{2}\id\right).
\end{align*}
The last equality results from the fact that $(\lambda_{\mu\nu,i})_{i=1}^d$ are eigenvalues of $\Sigma^{\frac{1}{2}}_{\mu}\Sigma_\nu\Sigma^{\frac{1}{2}}_\mu$. Additionally, 
\begin{align*}
    K^*_{\mu\nu}&=\Sigma^{\frac{1}{2}}_\mu V\diag\left(\frac{-\varepsilon+\sqrt{\varepsilon^2+16\lambda_{\mu\nu,i}}}{4\lambda^{\frac{1}{2}}_{\mu\nu,i}}\right)_{i=1}^dU^{\top}\Sigma^{\frac{1}{2}}_\nu\\
    &=\frac{1}{2}~\Sigma^{\frac{1}{2}}_\mu V\diag\left(\Big[\frac{\varepsilon^2}{4\lambda_{\mu\nu,i}}+4\Big]^\frac{1}{2}\right)_{i=1}^dU^{\top}\Sigma^{\frac{1}{2}}_\nu-   \frac{\varepsilon}{4}~\Sigma^{\frac{1}{2}}_\mu V\diag\left(\lambda^{-\frac{1}{2}}_{\mu\nu,i}\right)_{i=1}^dU^{\top}\Sigma^{\frac{1}{2}}_\nu.
\end{align*}
Recall that $\Sigma^{\frac{1}{2}}_\nu\Sigma^{\frac{1}{2}}_\mu=U\Lambda^\frac{1}{2}V^{\top}$, or equivalently $U^{\top}\Sigma^\frac{1}{2}_\nu=\Lambda^\frac{1}{2}V^{\top}\Sigma^{-\frac{1}{2}}_\mu$. Then, we have
\begin{align*}
    K^*_{\mu\nu}&=\frac{1}{2}~\Sigma^{\frac{1}{2}}_\mu V\diag\left(\Big[\frac{\varepsilon^2}{4\lambda_{\mu\nu,i}}+4\Big]^\frac{1}{2}\right)_{i=1}^d\diag\left(\lambda^{\frac{1}{2}}_{\mu\nu,i}\right)_{i=1}^dV^{\top}\Sigma^{-\frac{1}{2}}_\mu-\frac{\varepsilon}{4}~\Sigma^{\frac{1}{2}}_\mu (U\Lambda^{\frac{1}{2}}V^{\top})^{-1}\Sigma^{\frac{1}{2}}_\nu\\
    &=\frac{1}{2}~\Sigma^{\frac{1}{2}}_\mu V\diag\left(\Big[\frac{\varepsilon^2}{4}+4\lambda_{\mu\nu,i}\Big]^\frac{1}{2}\right)_{i=1}^dV^{\top}\Sigma^{-\frac{1}{2}}_\mu-\frac{\varepsilon}{4}~\id.
\end{align*}
As a result of SVD property, we get $V\diag(\lambda_{\mu\nu,i})_{i=1}^dV^{\top}=\Sigma^\frac{1}{2}_\mu\Sigma_\nu\Sigma^\frac{1}{2}_\mu$. Thus, $V\diag\left(\frac{\varepsilon^2}{4}+4\lambda_{\mu\nu,i}\right)_{i=1}^dV^{\top}=\frac{\varepsilon^2}{4}\id+4\Sigma^\frac{1}{2}_\mu\Sigma_\nu\Sigma^\frac{1}{2}_\mu$, which follows that
\begin{align*}
    V\diag\left(\Big[\frac{\varepsilon^2}{4}+4\lambda_{\mu\nu,i}\Big]^\frac{1}{2}\right)_{i=1}^dV^{\top}=\left(\frac{\varepsilon^2}{4}~\id+4~\Sigma^\frac{1}{2}_\mu\Sigma_\nu\Sigma^\frac{1}{2}_\mu\right)^\frac{1}{2}.
\end{align*}
Consequently, we obtain
\begin{align*}K^*_{\mu\nu}=\frac{1}{2}\Sigma^\frac{1}{2}_\mu\left(\frac{\varepsilon^2}{4}~\id+4~\Sigma^\frac{1}{2}_\mu\Sigma_\nu\Sigma^\frac{1}{2}_\mu\right)^\frac{1}{2}\Sigma^{-\frac{1}{2}}_\mu-\frac{\varepsilon}{4}\id = \Sigma^\frac{1}{2}_\mu\left(\frac{\varepsilon^2}{16}~\id+~\Sigma^\frac{1}{2}_\mu\Sigma_\nu\Sigma^\frac{1}{2}_\mu\right)^\frac{1}{2}\Sigma^{-\frac{1}{2}}_\mu-\frac{\varepsilon}{4}\id.
\end{align*}
 Hence, the proof is completed.
\end{proof}
\subsection{Entropic optimal transport between unbalanced Gaussian measures}
Prior to presenting the closed-form expression of entropic optimal transport between unbalanced Gaussian measures, let us review the formulation of entropic unbalanced optimal transport (UOT).
\paragraph{Entropic Unbalanced Optimal Transport.}When either $\alpha$ or $\beta$ is not a probability measure, the formulation in equation~\eqref{eq:entropic_OT_formulation} is no longer valid. One solution to deal with this issue is using entropic UOT, which is given by
\begin{align*}
    \mathsf{UOT}_{\varepsilon,\tau}(\alpha,\beta):=\min_{\pi\in\mathcal{M}^+(\mathbb{R}^d\times\mathbb{R}^d)}~\mathbb{E}_{\pi}\|X-Y\|^2 + \tau\KL(\pi_x\|\alpha)+\tau\KL(\pi_y\|\beta)+\varepsilon\KL(\pi\|\alpha\otimes\beta),
\end{align*}
where $X,Y$ are independent random vectors in $\br^d$ such that $(X,Y)\sim\pi$ while $\varepsilon,\tau>0$ are regularized parameters, and $\pi_x,\pi_y$ are the marginal distributions of the coupling $\pi$ corresponding to $\alpha$ and $\beta$, respectively.

Consider two unbalanced Gaussian measures in $\br^d$:
\begin{align}\label{eq:unbalanced_Gaussians_form}
    \mu=m_\mu\mathcal{N}(\uU,\Sigma_\mu) \quad \text{and} \quad \nu=m_\nu\mathcal{N}(\vV,\Sigma_\nu),
\end{align}
where $m_\mu,m_\nu>0$ are their masses and $\Sigma_\mu,\Sigma_\nu$ are positive definite matrices. To state our main result, we need to define some quantities which are necessary for our analysis. Firstly, let us denote   
\begin{align*}
    \Sigma_{\mu,\eta} := \eta \Sigma_{\mu}^{-1} + \id; \qquad \Sigma_{\nu,\eta} := \eta \Sigma_{\nu}^{-1} + \id,
    \end{align*}
where $\eta :=\frac{\tau + \varepsilon}{2}$. Next, we define
\begin{align*}
    A&:=\frac{1}{\tau} \left\{\Big[ \frac{\varepsilon^2}{4} \id + \tau (2\eta) B^{\top}B \Big]^{\frac{1}{2}} - \frac{\varepsilon}{2}\id\right\}, \text{ where } B:= \Sigma_{\nu,\eta}^{-\frac{1}{2}} \Sigma_{\mu,\eta}^{-\frac{1}{2}}, \\
    \Upsilon_{\uU^*,\vV^*} &:=  (\uU - \vV)^{\top}\left\{\id - (\Sigma_{\mu}+\Sigma_{\nu})\Big[\eta \id + \Sigma_{\mu}^{-1} + \Sigma_{\nu}^{-1} \Big]^{-1} \right\} (\uU -\vV), \\
    \Upsilon_{\Sigma^*} &:=2\eta\Big\{1-\log(\eta) +\log \det(A) - \log \det(B) -\frac{1}{2} \log \det(\Sigma_{\mu}\Sigma_{\nu}) \Big\} -\frac{\varepsilon}{2} \log\det\big(\id -A^2(B^{\top}B)^{-1}\big), \\
    \Upsilon^* &:= \Upsilon_{\uU^*,\vV^*} + \Upsilon_{\Sigma^*} + 2\eta d.
\end{align*}
\begin{theorem}
\label{theorem:entropic_UOT}
Let $\mu$ and $\nu$ be two unbalanced Gaussian measures given in equation~\eqref{eq:unbalanced_Gaussians_form}, the entropic UOT between $\mu$ and $\nu$ can be computed as
\begin{align*}
    \mathsf{UOT}_{\varepsilon,\tau}(\mu,\nu)=m_{\pi^*}\Upsilon^* + \varepsilon\KL(m_{\pi^*}\|m_\mu m_\nu) + \tau[\KL(m_{\pi^*}\|m_\mu)+\KL(m_{\pi^*}\|m_\nu)].
\end{align*}
where $m_{\pi^*} =  (m_{\mu}m_{\nu})^{\frac{\tau+\varepsilon}{2\tau+\varepsilon}} \exp\Big\{\frac{-\Upsilon^*}{2\tau + \varepsilon}\Big\}$. Moreover, the optimal transport plan is also an unbalanced Gaussian measure which admits the form $\pi^*=m_{\pi^*}\mathcal{N}\big( (\uU^*,\vV^*)^{\top}, \Sigma^*_{xy}\big)$, where 
\begin{align*}
    \uU^*&= \uU-\Sigma_{\mu}\Big[\eta\id + \Sigma_{\mu} + \Sigma_{\nu} \Big]^{-1} (\uU-\vV);\\
    \vV^* &= \vV + \Sigma_{\nu}\Big[\eta\id + \Sigma_{\mu}^{-1} + \Sigma_{\nu}^{-1} \Big]^{-1}(\uU - \vV);\\
    \Sigma_{xy}^* &= \begin{pmatrix}\Sigma_x^* & K_{xy}^* \\
    (K_{xy}^*)^{\top} &\Sigma_y^*
    \end{pmatrix},
\end{align*}
with 
\begin{align*}
&\Sigma_x^*=\eta \Sigma_{\mu,\eta}^{-\frac{1}{2}} A(B^{\top}B)^{-\frac{1}{2}} \Sigma_{\mu,\eta}^{-\frac{1}{2}}
;\qquad \Sigma_y^* = \eta \Sigma_{\nu,\eta}^{-\frac{1}{2}} A(B^{\top}B)^{-\frac{1}{2}}\Sigma_{\nu,\eta}^{-\frac{1}{2}};\\ 
& K_{xy}^* = \eta \Sigma_{\mu,\eta}^{-\frac{1}{2}}  A(\id-A)^{-1} (B^{\top}B)^{-\frac{1}{2}} B^{-1}\Sigma_{\nu,\eta}^{-\frac{1}{2}}.
\end{align*}
\end{theorem}

\begin{proof}[Proof of Theorem~\ref{theorem:entropic_UOT}]
We assume that $\pi$ is a positive measure such that $\pi = m_{\pi} \overline{\pi}$ with $\overline{\pi}$ is a probability measure with mean $(\uU_x, \vV_y)$ and covariance 
\begin{align*}
    \Sigma_{\pi} = \begin{pmatrix} \Sigma_x & K_{xy} \\
    K_{xy}^{\top} & \Sigma_y
    \end{pmatrix}.
\end{align*}
Here the two marginals of $\overline{\pi}$ are denoted by $\overline{\pi}_x$ and $\overline{\pi}_y$ where $\overline{\pi}_x$ has mean $\uU_x$ and variance $\Sigma_x$ while $\overline{\pi}_y$ has mean $\vV_y$ and variance $\Sigma_y$.
Let $\pi_x$ and $\pi_y$ be two marginals of $\pi$, then $\pi_x = m_{\pi} \overline{\pi}_x$ and $\pi_y = m_{\pi}\overline{\pi}_y$.

By Lemma \ref{lemma:min_KL_Gaussian}, $\pi$ needs to be Gaussian.  We also have 
\begin{align*}
    \mathbb{E}_{\pi}\|X-Y\|^2&= m_{\pi} \Big\{ \|\uU_x- \vV_y\|^2 +\tr(\Sigma_x)+\tr(\Sigma_y)-2\tr(K_{xy})\Big\}.
\end{align*}
According to Lemma~\ref{lemma:kl_divergence},
\begin{align*}
    \KL(\pi_x\|\mu)&=m_\pi\KL(\overline{\pi}_x\|\overline{\mu})+\KL(m_\pi\|m_\mu),\\
    \KL(\pi_y\|\nu)&=m_\pi\KL(\overline{\pi}_y\|\bar{\nu})+\KL(m_\pi\|m_\nu),\\
    \KL(\pi\|\mu\otimes\nu)&=m_\pi\KL(\overline{\pi}\|\overline{\mu}\otimes\overline{\nu})+\KL(m_\pi\|m_\mu m_\nu).
\end{align*}
Combining the above results, the entropic optimal transport between $\mu$ and $\nu$ reads as
\begin{align*}
    \mathsf{UOT}_{\varepsilon,\tau}(\mu,\nu)=\min_{\pi\in\mathcal{M}^+(\mathbb{R}^d\times\mathbb{R}^d)}\Big\{m_\pi\Upsilon + \varepsilon\KL(m_\pi\|m_\mu m_\nu) + \tau[\KL(m_\pi\|m_\mu)+\KL(m_\pi\|m_\nu)]\Big\},
\end{align*}
where $\Upsilon:=\|\uU_x - \vV_y\|^2 + \tr(\Sigma_x)+\tr(\Sigma_y)-2\tr(K_{xy})+\varepsilon\KL(\overline{\pi}\|\bar{\mu}\otimes\overline{\nu})+\tau\big[\KL(\overline{\pi}_x\|\overline{\mu})+\KL(\overline{\pi}_y\|\overline{\nu})\big]$.
We divide our proof into two main steps. The first step is to minimize $\Upsilon$, and the second step is to find the optimal scale $m_{\pi^*}$.

\paragraph{Step 1: Minimization of $\Upsilon$} For the three KL terms, we have
\begin{align*}
    \mathsf{KL}(\overline{\pi}_x \|\overline{\mu}) &= \frac{1}{2} \Big\{\tr(\Sigma_x \Sigma_{\mu}^{-1}) - d - \log\det(\Sigma_x) + \log\det(\Sigma_{\mu}) \Big\} + \frac{1}{2} (\uU_x - \uU)^{\top}\Sigma_\mu^{-1} (\uU_x - \uU), \\
    \mathsf{KL}(\overline{\pi}_y \|\overline{\nu}) &= \frac{1}{2} \Big\{\tr(\Sigma_y \Sigma_{\nu}^{-1}) - d - \log\det(\Sigma_y) + \log\det(\Sigma_{\nu}) \Big\} + \frac{1}{2}(\vV_y - \vV)^{\top}\Sigma_\nu^{-1} (\vV_y -\vV ), \\
    \mathsf{KL}(\overline{\pi} \| \overline{\mu} \otimes \overline{\nu} ) &=\frac{1}{2} \Big\{\tr(\Sigma_x \Sigma_{\mu}^{-1}) - d - \log\det(\Sigma_x) + \log\det(\Sigma_{\mu}) \Big\} + \frac{1}{2} \Big\{\tr(\Sigma_y \Sigma_{\nu}^{-1}) - d - \log\det(\Sigma_y) + \log\det(\Sigma_{\nu}) \Big\} \\
    &\qquad +\frac{1}{2} (\uU_x-\uU, \vV_y- \vV)^{\top}\Sigma_{\mu\otimes \nu}^{-1} (\uU_x-\uU,\vV_y - \vV)   - 
     \frac{1}{2} 
    \sum_{i=1}^d \log(1-\kappa_{xy,i}),
\end{align*}
where 
\begin{align*}
\Sigma_{\mu\otimes \nu} = \begin{pmatrix} \Sigma_{\mu} & \mathbf{0}_{d\times d} \\
\mathbf{0}_{d\times d} & \Sigma_{\nu}
\end{pmatrix},
\end{align*}
and $\kappa^{\frac{1}{2}}_{xy,i}$ is the $i$-th largest singular value of $\Sigma_x^{-\frac{1}{2}}K_{xy}\Sigma^{-\frac{1}{2}}_y$ for all $i\in [d]$. Here we use the equation $\tr(\Sigma_{\pi}\Sigma_{\mu\otimes \nu}^{-1}) = \tr(\Sigma_x\Sigma_{\mu}^{-1}) + \tr(\Sigma_y\Sigma_{\nu}^{-1})$ and results from part (b) of Lemma \ref{lemma:entropic_gw:inner_product:closed_form:supporting}. Put the results together, we get
\begin{align*}
    \Upsilon &= \tr(\Sigma_x)+\tr(\Sigma_y)-2\tr(K_{xy})+\eta \Big\{ \tr(\Sigma_x \Sigma_{\mu}^{-1}) -\log \det(\Sigma_x) + \tr(\Sigma_y \Sigma_{\nu}^{-1}) -\log \det(\Sigma_y)  \Big\} -\frac{\varepsilon}{2}\sum_{i=1}^d\log(1-\kappa_{xy,i})\\ 
    & +\|\uU_x - \vV_y\|^2+ \frac{1}{2} \Big\{\tau (\uU_x - \uU)^{\top}\Sigma_\mu^{-1} (\uU_x - \uU) + \tau(\vV_y - \vV)^{\top}\Sigma_\nu^{-1} (\vV_y -\vV ) + \varepsilon (\uU_x-\uU, \vV_y- \vV)^{\top}\Sigma_{\mu\otimes\nu}^{-1} (\uU_x-\uU,\vV_y - \vV)  \Big\}\\
    &+ 2\eta d.
\end{align*}

\paragraph{The means part:}
We first work with the terms involving $\uU_x$ and $\vV_y$. Let $\uU_x - \uU = \widetilde{\uU}_{x}$ and $\vV_y - \vV = \widetilde{\vV}_{y}$, then
\begin{align*}
    (\widetilde{\uU}_{x}, \widetilde{\vV}_{y})^{\top} \Sigma_{\mu\otimes \nu}^{-1} (\widetilde{\uU}_{x}, \widetilde{\vV}_{y}) &= \widetilde{\uU}_{x}^{\top}\Sigma_{\mu}^{-1}\widetilde{\uU}_{x} + \widetilde{\vV}_{y}\Sigma_{\nu}^{-1}\widetilde{\vV}_{y} \nonumber\\
    \|\uU_x - \vV_y\|^2 &= \|\widetilde{\uU}_{x} - \widetilde{\vV}_{y} + \uU - \vV\|^2 \\
    &= \|\widetilde{\uU}_{x}\|^2 + \|\widetilde{\vV}_{y}\|^2 - 2\widetilde{\uU}_{x}^{\top}\widetilde{\vV}_{y} + 2\widetilde{\uU}_{x}^{\top}(\uU - \vV) - 2\widetilde{\vV}^{\top}_y(\uU-\vV) + \|\uU-\vV\|^2.
\end{align*}
Hence, sum of all terms which include $\uU_x$ and $\vV_y$ is equal to
\begin{align} \label{equation:Upsilon_uv}
    \Upsilon_{\uU,\vV}:=& \widetilde{\uU}_x^{\top}\Big(\eta\Sigma_{\mu}^{-1}+\id \Big)\widetilde{\uU}_x + \widetilde{\vV}_y^{\top}\Big(\eta\Sigma_{\nu}^{-1}+\id \Big)\widetilde{\vV}_y- 2\widetilde{\uU}_x^{\top}\widetilde{\vV}_y + 2\widetilde{\uU}_x^{\top}(\uU-\vV) - 2\widetilde{\vV}_y^{\top}(\uU-\vV) + \|\uU-\vV\|^2 \nonumber \\
=&\widetilde{\uU}_x^{\top}\Sigma_{\mu,\eta}\widetilde{\uU}_x + \widetilde{\vV}_y^{\top}\Sigma_{\nu,\eta}\widetilde{\vV}_y- 2\widetilde{\uU}_x^{\top}\widetilde{\vV}_y + 2\widetilde{\uU}_x^{\top}(\uU-\vV) - 2\widetilde{\vV}_y^{\top}(\uU-\vV) + \|\uU-\vV\|^2
.\end{align}
Taking the derivative with respect to $\widetilde{\uU}_x$ and $\widetilde{\vV}_y$ of $\Upsilon_{\uU,\vV}$, and set the equation to zero we obtain
\begin{align*}
    \Sigma_{\mu,\eta}\widetilde{\uU}_x   - \widetilde{\vV}_y + (\uU - \vV) &= \mathbf{0}, \\
      \Sigma_{\nu,\eta}\widetilde{\vV}_y - \widetilde{\uU}_x - (\uU - \vV) &= \mathbf{0}.
\end{align*}
Adding two equations we get 
\begin{align*}
    (\eta\Sigma_{\mu}^{-1} + \id)\widetilde{\uU}_x - \widetilde{\vV}_y + (\eta\Sigma_{\nu}^{-1}+\id) \widetilde{\vV}_y - \widetilde{\uU}_x &= \mathbf{0}, \\
    \Sigma_{\mu}^{-1}\widetilde{\uU}_x  &=- \Sigma_{\nu}^{-1}\widetilde{\vV}_y.
\end{align*}
Replace them into the above equation, we obtain
\begin{align*}
    \Sigma_{\mu,\eta}\widetilde{\uU}_x  + \Sigma_{\nu}\Sigma_{\mu}^{-1} \widetilde{\uU}_x &= -(\uU - \vV) \\
    \widetilde{\uU}_x &=- \Big[\Sigma_{\mu,\eta} + \Sigma_{\nu}\Sigma_{\mu}^{-1}\Big]^{-1} (\uU-\vV) \\
    &= -\Sigma_{\mu}\Big[\eta\id + \Sigma_{\mu} + \Sigma_{\nu} \Big]^{-1} (\uU-\vV).
    \end{align*}
Similarly, we obtain
   $\widetilde{\vV}_y = \Sigma_{\nu}\Big[\eta\id + \Sigma_{\mu}^{-1} + \Sigma_{\nu}^{-1} \Big]^{-1}(\uU - \vV)$. The value of $\Upsilon_{\uU,\vV}$ at the minimizer is computed as follow
   \begin{align*}
\widetilde{\uU}_x^{\top}\Sigma_{\mu,\eta} \widetilde{\uU}_x - \widetilde{\uU}_x^{\top}\widetilde{\vV}_y &= - \widetilde{\uU}_x^{\top}(\uU-\vV)\\
\widetilde{\vV}_y^{\top}\Sigma_{\nu,\eta} \widetilde{\vV}_y - \widetilde{\uU}_x^{\top}\widetilde{\vV}_y &=  \widetilde{\vV}_y^{\top}(\uU-\vV).
   \end{align*}
Adding them together
\begin{align*}
    \widetilde{\uU}_x^{\top}\Sigma_{\mu,\eta} \widetilde{\uU}_x + \widetilde{\vV}_y^{\top}\Sigma_{\nu,\eta} \widetilde{\vV}_y- 2 \widetilde{\uU}_x^{\top}\widetilde{\vV}_y = -(\widetilde{\uU}_x -\widetilde{\vV}_y)^{\top}(\uU-\vV).
\end{align*}
Put it back to equation \eqref{equation:Upsilon_uv}
\begin{align*}
    \Upsilon_{\uU^*,\vV^*} = (\widetilde{\uU}_x - \widetilde{\vV}_y)^{\top}(\uU-\vV) + \|\uU-\vV\|^2 = (\uU - \vV)^{\top}\left\{\id - (\Sigma_{\mu}+\Sigma_{\nu})\Big[\eta \id + \Sigma_{\mu}^{-1} + \Sigma_{\nu}^{-1} \Big]^{-1} \right\} (\uU -\vV).
\end{align*}

\paragraph{The covariance part:}
The second part is to group all factors of $\Sigma_x$, $\Sigma_y$ and $\kappa_{xy,i}$, which is
\begin{align}
  \Upsilon_{\Sigma}:=& \tr(\Sigma_x) + \tr(\Sigma_y) - 2\tr(K_{xy}) + \eta\Big\{ \tr(\Sigma_x\Sigma_{\mu}^{-1}) - \log\Big(\frac{\det(\Sigma_x)}{\det(\Sigma_{\mu})}\Big) + \tr(\Sigma_y\Sigma_{\nu}^{-1}) - \log \Big(\frac{\det(\Sigma_y)}{\det(\Sigma_{\nu})} \Big)  \Big\} \nonumber \\
  &- \frac{\varepsilon}{2}\log \prod_{i=1}^d (1-\kappa_{xy,i}). \label{equation:Upsilon_Sigma}
\end{align}
Grouping terms which have the same factors, the above quantity is rewritten as  
\begin{align} \label{equation:Upsilon_Sigma_simple}
    \Upsilon_{\Sigma}=&\tr(\Sigma_x \Sigma_{\mu,\eta}) + \tr(\Sigma_y \Sigma_{\nu,\eta}) - 2 \tr(K_{xy}) - \eta \Big[\log\Big(\frac{\det(\Sigma_x)}{\det(\Sigma_{\mu})} \Big) + \log \Big(\frac{\det(\Sigma_y)}{ \det(\Sigma_{\nu}) }\Big)\Big] - \frac{\varepsilon}{2}\sum_{i=1}^d\log(1-\kappa_{xy,i}).
\end{align}
Let   $U_{xy} \Lambda_{xy} V_{xy}^{\top}$ be SVD of  $\Sigma_x^{-\frac{1}{2}}K_{xy}\Sigma_y^{-\frac{1}{2}}$ and denote
\begin{align} \label{equation:SVD_XY}
    \Sigma_x^{\frac{1}{2}} U_{xy} = X; \qquad V_{xy}^{\top}\Sigma_y^{\frac{1}{2}}= Y. 
\end{align}
Then we get
\begin{align*}
    XX^{\top} = \Sigma_x; \qquad Y^{\top}Y = \Sigma_y;\qquad \Lambda_{xy} = \mathsf{diag}\big(\kappa_{xy,i}^{\frac{1}{2}}\big).
\end{align*}
Note that if we fix $\Sigma_x$, $\Sigma_y$, $\Lambda_{xy}$, then we only need to maximize the term $\tr(K_{xy})$, which is equal to
$\tr(\Lambda_{xy} YX)$. We could choose $U_{xy}$ and $V_{xy}$ such that $YX = \mathsf{diag}(\lambda_{xy,i})$ where $\lambda_{xy,i}$ is the $i$th singular value of $\Sigma_y^{\frac{1}{2}}\Sigma_x^{\frac{1}{2}}$.  By von Neumann's trace inequality \citep{kristof1969neumann,horn_johnson_1991}, it is where the term $\tr(K_{xy})$ attains its maximum. Hence, $YX$ is a diagonal matrix when the objective function $\Upsilon_{\Sigma}$  attains its minimum. 

We also note that the form $\Sigma_x^{\frac{1}{2}}U_{xy}$ and $V_{xy}^{\top} \Sigma_y^{\frac{1}{2}}$ are the QR decompositions of $X$ and $Y$, respectively. Hence there is no condition on $X$ and $Y$, except that they are invertible.  Then $\Upsilon_{\Sigma}$ now is equal to 
\begin{align*}
  \Upsilon_{\Sigma}=&  \tr(XX^{\top} \Sigma_{\mu,\eta}) + \tr(Y^{\top}Y\Sigma_{\nu,\eta}) - 2\tr(X \Lambda_{xy}Y) - 2\eta\Big[\log\big( \det(X) \det(Y)\big) \Big] -\frac{\varepsilon}{2}\sum_{i=1}^d \log(1-\kappa_{xy,i}) + \mathsf{const}.
\end{align*}
Taking the derivative with respect to $X$ and $Y$ (since in a small local neighbourhood, they are still invertible), we get the following system of equations
\begin{align}
   2 \Sigma_{\mu,\eta} X - 2Y^{\top}\Lambda_{xy}^{\top} - 2\eta (X^{-1})^{\top} &= \mathbf{0} \label{equation:X_Lambda_xy}\\
   2Y \Sigma_{\nu,\eta} - 2\Lambda_{xy}^{\top} X^{\top} - 2\eta (Y^{-1})^{\top} &= \mathbf{0} \label{equation:Y_Lambda_xy}.
\end{align}
We multiply  the first equation with $X^{\top}$ on the left side and the second equation with $Y^{\top}$ on the right side,
\begin{align}
    X^{\top}\Sigma_{\mu,\eta} X &= (YX)^{\top}\Lambda_{xy} + \eta \id \label{equation:X_Lambda_tau_varepsilon} \\
    Y\Sigma_{\nu,\eta}Y^{\top} &= \Lambda_{xy}^{\top}(YX)^{\top} + \eta\id \label{equation:Y_Lambda_tau_varepsilon}.
\end{align}
The right hand side of \eqref{equation:X_Lambda_tau_varepsilon} and \eqref{equation:Y_Lambda_tau_varepsilon} are diagonal matrices, since $\Lambda_{xy}$ and $YX$ are diagonal matrices. It means that $X^{\top}\Sigma_{\mu,\eta}X$ is also a diagonal matrix.
Denote 
\begin{align}
    (YX)^{\top} \Lambda_{xy} + \eta\id = \Lambda_{xy,\eta}:= \mathsf{diag}(\lambda_{xy,\eta,i}). \label{equation:YX_Lambda_Diag}
\end{align}
Then
\begin{align*}
   X^{\top} \Sigma_{\mu,\eta}X =X^{\top} U_{\mu,\eta}\Lambda_{\mu,\eta}U^{\top}_{\mu,\eta} X   &= \Lambda_{xy,\eta} \\
     \Big[\Lambda_{\mu,\eta}^{\frac{1}{2}} \Lambda_{xy,\eta}^{-\frac{1}{2}} X^{\top} U_{\mu,\eta}\Big]\Lambda_{\mu,\eta} \Big[U^{\top}_{\mu,\eta} X \Lambda_{xy,\eta}^{-\frac{1}{2}}\Lambda_{\mu,\eta}^{\frac{1}{2}} \Big] &=\Lambda_{\mu,\eta},
\end{align*}
where $U_{\mu,\eta}\Lambda_{\mu,\eta}U^{\top}_{\mu,\eta}$ and $U_{\nu,\eta}\Lambda_{\nu,\eta}U^{\top}_{\nu,\eta}$ are the spectral decomposition of $\Sigma_{\mu,\eta}$ and $\Sigma_{\nu,\eta}$, respectively. 
The equation has the form $A\Lambda A^{\top} = \Lambda$, with $\Lambda$ is a diagonal matrix. Let $A\Lambda^{\frac{1}{2}} = \Lambda^{\frac{1}{2}}U$, then $\Lambda^{\frac{1}{2}}U U^{\top} \Lambda^{\frac{1}{2}} = A \Lambda A^{\top} =\Lambda$. It means that $UU^{\top} =\id$, then $U$ is unitary matrix. Thus we obtain $A = \Lambda^{\frac{1}{2}} U \Lambda^{-\frac{1}{2}}$. Replace $A$ by 
 $\Big[\Lambda_{\mu,\eta}^{\frac{1}{2}} \Lambda_{xy,\eta}^{-\frac{1}{2}} X^{\top} U_{\mu,\eta}\Big]$, there exists unitary matrix $U$ such that
 \begin{align}
     \Lambda_{\mu,\eta}^{\frac{1}{2}} \Lambda_{xy,\eta}^{-\frac{1}{2}} X^{\top} U_{\mu,\eta} &= \Lambda^{\frac{1}{2}}_{\mu,\eta} U \Lambda_{\mu,\eta}^{-\frac{1}{2}} \nonumber \\
      \Lambda_{xy,\eta}^{-\frac{1}{2}} X^{\top} U_{\mu,\eta} &= U \Lambda_{\mu,\eta}^{-\frac{1}{2}} \nonumber \\
     X^{\top}&= \Lambda_{xy,\eta}^{\frac{1}{2}}U \Lambda_{\mu,\eta}^{-\frac{1}{2}} U^{\top}_{\mu,\eta} \nonumber \\
     X&= U_{\mu,\eta} \Lambda_{\mu,\eta}^{-\frac{1}{2}} U^{\top}\Lambda_{xy,\eta}^{\frac{1}{2}} \label{equation:X}.
 \end{align}
 Similarly, there exists an unitary matrix $V$ such that
 \begin{align} \label{equation:Y}
     Y = \Lambda_{xy,\eta}^{\frac{1}{2}} V \Lambda_{\nu,\eta}^{-\frac{1}{2}} U_{\nu,\eta}^{\top}.
 \end{align}
 Combine the results with a note that $U_{\mu,\eta} = U_{\mu}$ and $U_{\nu,\eta} = U_{\nu}$, we obtain
 \begin{align}
     YX &= \Lambda_{xy,\eta}^{\frac{1}{2}} V \Lambda_{\nu,\eta}^{-\frac{1}{2}} U_{\nu,\eta}^{\top} U_{\mu,\eta} \Lambda_{\mu,\eta}^{-\frac{1}{2}} U^{\top}\Lambda_{xy,\eta}^{\frac{1}{2}} \nonumber\\
     (YX)\Lambda^{-1}_{xy,\eta} &= V\Lambda_{\nu,\eta}^{-\frac{1}{2}} U_{\nu}^{\top}U_{\mu}\Lambda_{\mu,\eta}^{-\frac{1}{2}} U^{\top}:=\mathsf{diag}(\alpha_i)_{i=1}^d. \label{eq:SVD_XY}
 \end{align}
 Here, $\alpha_i$ is the $i$th singular values of $U_{\nu}\Lambda_{\nu,\tau,\varepsilon}^{-\frac{1}{2}} U_{\nu}^{\top}U_{\mu}\Lambda_{\mu,\eta}^{-\frac{1}{2}} U_{\mu} $, where $U$ and $V$ could be chosen such that 
 $ V\Lambda_{\nu,\eta}^{-\frac{1}{2}} U_{\nu}^{\top}U_{\mu}\Lambda_{\mu,\tau,\varepsilon}^{-\frac{1}{2}} U$ is a diagonal matrix $\mathsf{diag}(\alpha_i)_{i=1}^d$. By the von Neumann's trace inequality \citep{kristof1969neumann,horn_johnson_1991},
 \begin{align*}
   \tr(K_{xy}) =  \tr(\Lambda_{xy}YX) \leq \sum_{i=1}^d \alpha_i\kappa_{xy,i}^{\frac{1}{2}} \lambda_{xy,\eta,i}. 
 \end{align*}
 The equality happens when $(\kappa_{xy,i}) $, $(\alpha_i)$ and $(\lambda_{xy,\eta,i})$ are decreasing sequences with $ V\Lambda_{\nu,\eta}^{-\frac{1}{2}} U_{\nu}^{\top}U_{\mu}\Lambda_{\mu,\eta}^{-\frac{1}{2}} U$  is diagonal.
 
 Substitute it into the equation \eqref{equation:YX_Lambda_Diag} with $\Lambda_{xy,\eta} := \mathsf{diag}(\lambda_{xy,\eta,i})_{i=1}^d$
 \begin{align}
     \lambda_{xy,\eta,i} &= \eta + \alpha_i \kappa_{xy,i}^{\frac{1}{2}} \lambda_{xy,\eta,i}\nonumber\\
     \lambda_{xy,\eta,i} &= \eta \Big[1-\alpha_i \kappa_{xy,i}^{\frac{1}{2}}\Big]^{-1}. \label{equation:lambda_xy_eta}
 \end{align}
 Combine it with equations \eqref{equation:X_Lambda_tau_varepsilon} and \eqref{equation:YX_Lambda_Diag}, we get
 \begin{align*}
     \det(X^{\top}\Sigma_{\mu,\eta}X) &= \det(X)^2 \det(\Sigma_{\mu,\eta})=  \eta^d \prod_{i=1}^d  \big(1-\alpha_i \kappa_{xy,i}^{\frac{1}{2}} \big)^{-1}. 
 \end{align*}
 Therefore,
 \begin{align}\label{equation:log_det_X}
     \log \det(X) = -\frac{1}{2} \sum_{i=1}^d\log\big(1-\alpha_i \kappa_{xy,i}^{\frac{1}{2}}\big) + \mathsf{const}.
 \end{align}
 Similarly, we also have
 \begin{align}\label{equation:log_det_Y}
     \log \det(Y) = - \frac{1}{2}\sum_{i=1}^d\log\big(1-\alpha_i \kappa_{xy,i}^{\frac{1}{2}}\big) + \mathsf{const}.
 \end{align}
 Put them \eqref{equation:X_Lambda_tau_varepsilon}, \eqref{equation:Y_Lambda_tau_varepsilon}, \eqref{equation:log_det_X}, \eqref{equation:log_det_Y} together, with $\Sigma^*$ to be the optimal solution of $X,Y$, etc,  we obtain
 \begin{align*}
     \Upsilon_{\Sigma^*}= 2\eta\sum_{i=1}^d \log\big(1-\alpha_i\kappa_{xy,i}^{\frac{1}{2}}) - \frac{\varepsilon}{2}\sum_{i=1}^d\log(1-\kappa_{xy,i}) + \mathsf{const}.
 \end{align*}
Let us consider the function for $\alpha, t \in (0,1)$ 
 \begin{align*}
     f(t) =2\eta \log\big(1-\alpha t^{\frac{1}{2}} \big) -\frac{\varepsilon}{2}\log(1-t).
 \end{align*}
 Taking derivative with respect to $t$ we get
 \begin{align*}
     f^{\prime}(t) = \frac{-2\eta\alpha}{2(1-\alpha t^{\frac{1}{2}}) t^{\frac{1}{2}}}   + \frac{\varepsilon}{2}\frac{1}{1-t} = \frac{\varepsilon (1-\alpha t^{\frac{1}{2}}) t^{\frac{1}{2}} -2\eta\alpha(1-t)}{2(1-\alpha t^{\frac{1}{2}}) t^{\frac{1}{2}} (1-t)} = \frac{(2\eta  - \varepsilon )\alpha t+ \varepsilon t^{\frac{1}{2}} -2\eta \alpha}{2(1-\alpha t^{\frac{1}{2}}) t^{\frac{1}{2}} (1-t)} 
 \end{align*}
 The nominator is a quadratic function of $t^{\frac{1}{2}}$, which is increasing function taking negative value with $t=0$ and positive value when $t = 1$, since $0<\alpha <1$. Thus, it has only one solution
 \begin{align*}
     \sqrt{t^*} = \frac{\sqrt{\varepsilon^2+ 8 \tau \eta \alpha^2 } -\varepsilon}{2\tau\alpha} = \frac{\sqrt{\varepsilon^2/4 +2\tau \eta \alpha^2} -\varepsilon/2}{\tau \alpha}.
 \end{align*}
 It also deduces that if $\alpha$ increases, the $t^*$ also increases.  Substituting the result into the equation \eqref{equation:lambda_xy_eta}, we obtain
 \begin{align}
     \lambda_{xy,\eta,i} = \frac{\eta}{1-\alpha_i \kappa_{xy,i}^{\frac{1}{2}}}= \frac{\eta}{1- \frac{\sqrt{\varepsilon^2/4 + \tau (2\eta) \alpha_i^2} -\varepsilon/2}{\tau} }. \label{formula:lambda_xy_eta}
 \end{align}
Due to equation \eqref{eq:SVD_XY}, we get
  $\alpha_i \lambda_{xy,\eta,i} = (YX)_{i,i}$. Combine with equation \eqref{equation:YX_Lambda_Diag},  we get
\begin{align}
    \lambda_{xy,i} = \frac{\lambda_{xy,\eta,i} - \eta}{\alpha_i \lambda_{xy,\eta,i}}. \label{formula:Lambda_xy}
\end{align}
 
 \paragraph{Covariance matrix:} Recall from equations \eqref{equation:X} and \eqref{equation:Y}, we have
 \begin{align*}
     X = U_{\mu,\eta} \Lambda_{\mu,\eta}^{-\frac{1}{2}} U^{\top}\Lambda_{xy,\eta}^{\frac{1}{2}}; \qquad  Y= \Lambda_{xy,\eta}^{\frac{1}{2}} V \Lambda_{\nu,\eta}^{-\frac{1}{2}} U_{\nu,\eta}^{\top}.
 \end{align*}
Since equation \eqref{equation:SVD_XY}, $V\Lambda_{\nu,\eta}^{-\frac{1}{2}} U_{\nu}^{\top}U_{\mu}\Lambda_{\mu,\eta}^{-\frac{1}{2}} U^{\top}$ is a diagonal matrix. It means that 
\begin{align*}
(VU_{\nu}^{\top} ) U_{\nu}\Lambda_{\nu,\eta}^{-\frac{1}{2}} U_{\nu}^{\top}U_{\mu}\Lambda_{\mu,\eta}^{-\frac{1}{2}}U_{\mu}^{\top}(U_{\mu} U) = (VU_{\nu}^{\top} ) (\Sigma_{\nu,\eta}^{-\frac{1}{2}}\Sigma_{\mu,\eta}^{-\frac{1}{2}}) (U_{\mu} U^{\top})= \mathsf{diag}(\alpha_i)_{i=1}^d,
\end{align*}
 which means that the $\alpha_i$ are the singular values of  $\Sigma_{\nu,\eta}^{-\frac{1}{2}}\Sigma_{\mu,\eta}^{-\frac{1}{2}}$. Then
 \begin{align*}
     \Sigma_{\nu,\eta}^{-\frac{1}{2}}\Sigma_{\mu,\eta}^{-\frac{1}{2}} = \big[U_{\nu} V^{\top} \big] \mathsf{diag}(\alpha_i) \big[UU_{\mu}^{\top}\big] = U_{\mu\nu,\eta} \mathsf{diag}(\alpha_i) V_{\mu\nu,\eta}^{\top}
 \end{align*}
which is a SVD of $\Sigma_{\nu,\eta}^{-\frac{1}{2}}\Sigma_{\mu,\eta}^{-\frac{1}{2}}$, where 
$U_{\mu}U^{\top} = V_{\mu\nu,\eta}$ and $VU_{\nu}^{\top} =U_{\mu\nu,\eta}^{\top}$. 
Put them into the equation \eqref{equation:SVD_XY}
\begin{align}
    K_{xy}&=X\Lambda_{xy}Y = U_{\mu} \Lambda_{\mu,\eta}^{-\frac{1}{2}} U^{\top} \Lambda_{xy,\eta}^{\frac{1}{2}} \Lambda_{xy} \Lambda_{xy,\eta}^{\frac{1}{2}} V \Lambda_{\nu,\eta}^{-\frac{1}{2}} U_{\nu}^{\top}= \Sigma_{\mu,\eta}^{-\frac{1}{2}} \big[U_{\mu} U^{\top}\big] \big[\Lambda_{xy,\eta}^{\frac{1}{2}}\Lambda_{xy} \Lambda_{xy,\eta}^{\frac{1}{2}} \big] \big[VU_{\nu}^{\top}\big]\Sigma_{\nu,\eta}^{-\frac{1}{2}} \nonumber \\
    &= \Sigma_{\mu,\eta}^{-\frac{1}{2}}\left[V_{\mu\nu,\eta} \mathsf{diag}\Big(\frac{\lambda_{xy,\eta,i}-\eta}{\alpha_i}\Big)_{i=1}^d U^{\top}_{\mu\nu,\eta} \right]\Sigma_{\nu,\eta}^{-\frac{1}{2}} \nonumber \\
    &= \eta \Sigma_{\mu,\eta}^{-\frac{1}{2}}  A(\id-A)^{-1} (B^{\top}B)^{-\frac{1}{2}} B^{-1}\Sigma_{\nu,\eta}^{-\frac{1}{2}}. \label{formula:K_xy}
\end{align}
We could obtain it from
\begin{align*}
    B &= U_{\mu\nu,\eta} \mathsf{diag}(\alpha_i)_{i=1}^d V_{\mu\nu,\eta}^{\top} \Rightarrow B^{\top}B = V_{\mu\nu,\eta}\mathsf{diag}(\alpha_i^2)_{i=1}^d V_{\mu\nu,\eta}^{\top} \\
     A &= \frac{1}{\tau}V_{\mu\nu,\eta} \mathsf{diag}\Big(\sqrt{\varepsilon^2/4 + \tau(2\eta) \alpha_i^2} -\varepsilon/2 \Big)_{i=1}^d V_{\mu\nu,\eta}^{\top} \\
     &=\frac{1}{\tau} \left\{\Big[ \frac{\varepsilon^2}{4} \id + \tau (2\eta) B^{\top}B \Big]^{\frac{1}{2}} - \frac{\varepsilon}{2}\id\right\} \\
     \eta(\id-A)^{-1} & = V_{\mu\nu,\eta} \mathsf{diag}(\lambda_{xy,\eta,i})_{i=1}^d V_{\mu\nu,\eta}^{\top}\\
 \eta \big\{(\id-A)^{-1} - \id\big\} (B^{\top}B)^{\frac{1}{2}}&=  V_{\mu\nu,\eta} \mathsf{diag}\left(\frac{\lambda_{xy,\eta,i}-\eta}{\alpha_i} \right)_{i=1}^d V_{\mu\nu,\tau}^{\top}.
\end{align*}
Multiply both side with $B^{-1} = V_{\mu\nu,\eta}\mathsf{diag}(\alpha_i) U_{\mu\nu,\eta}^{\top}$, we obtain the formula in \eqref{formula:K_xy}.

For the $\Sigma_x$, multiply with $X^{\top}$ to the right side of equation \eqref{equation:X_Lambda_xy}, we have
 \begin{align}
      \Sigma_{\mu,\eta} XX^{\top} - Y^{\top}\Lambda_{xy}^{\top}X^{\top} - \eta (X^{-1})^{\top} X^{\top} &= \mathbf{0} \nonumber \\
      \Sigma_{\mu,\eta} \Sigma_x &= (X\Lambda_{xy}Y)^{\top} + \eta \id \nonumber\\
      \Sigma_x &= \Sigma_{\mu,\eta}^{-1}(K_{xy}^{\top} + \eta \id) \label{formula:Sigma_x_K_xy}.
 \end{align}
 Another way to derive $\Sigma_x$ is to use equation \eqref{equation:X} of $X$ and equation \eqref{formula:lambda_xy_eta} of $\lambda_{xy,\eta}$
 \begin{align}
    \Sigma_x&= XX^{\top}= U_{\mu} \Lambda_{\mu,\eta}^{-\frac{1}{2}} U^{\top} \Lambda_{xy,\eta} U \Lambda_{\mu,\eta}^{-\frac{1}{2}} U_{\mu}^{\top} = \Sigma_{\mu,\eta}^{-\frac{1}{2}} U_{\mu}U^{\top}\Lambda_{xy,\eta} U U_{\mu}^{\top}\Sigma_{\mu,\eta}^{-\frac{1}{2}} \nonumber \\
    &=\Sigma_{\mu,\eta}^{-\frac{1}{2}}V_{\mu\nu,\eta}\Lambda_{xy,\eta}V_{\mu\nu,\eta}^{\top}\Sigma_{\mu,\eta}^{-\frac{1}{2}} \nonumber\\
    &= \eta \Sigma_{\mu,\eta}^{-\frac{1}{2}} A(B^{\top}B)^{-\frac{1}{2}} \Sigma_{\mu,\eta}^{-\frac{1}{2}},
 \end{align}
 where the last equation is obtained from the equation \eqref{formula:Lambda_xy}
 \begin{align*}
     V_{\mu\nu,\eta} \mathsf{diag}(\lambda_{xy,i})_{i=1}^d V_{\mu\nu,\eta}^{\top} &= \eta \big\{(\id -A)^{-1} - \id \big\} (\id -A) (B^{\top}B)^{-\frac{1}{2}} = \eta \big\{\id - (\id-A) \big\} (B^{\top}B)^{-\frac{1}{2}} \\
     &= \eta A(B^{\top}B)^{-\frac{1}{2}}.
 \end{align*}
 Similarly for equation \eqref{equation:Y_Lambda_xy}
 \begin{align}
    Y^{\top} Y \Sigma_{\nu,\eta} - Y^{\top}\Lambda_{xy}^{\top} X^{\top} - \eta \id &= \mathbf{0}  \nonumber\\
    \Sigma_y = (K_{xy}^{\top}+\eta \id) \Sigma_{\nu,\eta}^{-1}. \label{formula:Sigma_y_K_xy}
 \end{align}
 or we could derive another way to obtain $\Sigma_y = \eta \Sigma_{\nu,\eta}^{-\frac{1}{2}} A(B^{\top}B)^{-\frac{1}{2}}\Sigma_{\nu,\eta}^{-\frac{1}{2}}$.
 
Next we calculate the $\Upsilon_{\Sigma^*}$.  Since equations  \eqref{formula:Sigma_x_K_xy} and \eqref{formula:Sigma_y_K_xy}, we have
\begin{align*}
   \tr(\Sigma_x \Sigma_{\mu,\eta}) + \tr(\Sigma_y \Sigma_{\nu,\eta}) - 2 \tr(K_{xy}) = 2\eta.
\end{align*}
We also have
\begin{align*}
    \prod_{i=1}^d (1-\kappa_{xy,i}) &= \det\big(\id - \mathsf{diag}(\kappa_{xy,i})_{i=1}^d\big)\\
    V_{\mu\nu,\eta} \mathsf{diag}(\kappa_{xy,i})_{i=1}^d V_{\mu\nu,\eta}^{\top
} &= A^2 (B^{\top}B)^{-1}. \end{align*}
It means that $\sum_{i=1}^d\log(1-\kappa_{xy,i})_{i=1}^d = \log\det\big(\id - A^2(B^{\top}B)^{-1}\big)$. Put all the results into equation
\eqref{equation:Upsilon_Sigma_simple}, we obtain
 \begin{align*}
 \Upsilon_{\Sigma}&=\tr(\Sigma_x \Sigma_{\mu,\eta}) + \tr(\Sigma_y \Sigma_{\nu,\eta}) - 2 \tr(K_{xy}) - \eta \Big[\log\Big(\frac{\det(\Sigma_x)}{\det(\Sigma_{\mu})} \Big) + \log \Big(\frac{\det(\Sigma_y)}{ \det(\Sigma_{\nu}) }\Big)\Big] - \frac{\varepsilon}{2}\sum_{i=1}^d\log(1-\kappa_{xy,i})\\
     \Upsilon_{\Sigma^*} &= 2\eta - 2\eta \Big\{\log(\eta) +\log \det(A) - \log \det(B) - \frac{1}{2}\log \det(\Sigma_{\mu}) -\frac{1}{2}\log \det(\Sigma_{\nu}) \Big\} -\frac{\varepsilon}{2} \log \det\big(\id - A^2 (B^{\top}B)^{-1}\big) \\
     &= 2\eta\Big\{1-\log(\eta) +\log \det(A) - \log \det(B) -\frac{1}{2} \log \det(\Sigma_{\mu}\Sigma_{\nu}) \Big\} -\frac{\varepsilon}{2} \log\det\big(\id -A^2(B^{\top}B)^{-1}\big).
 \end{align*}
 \paragraph{Step 2: Calculation of $m_{\pi^*}$} We need to find the minimizer of the following problem:
 \begin{align*}
     \min_{m_\pi>0}~m_{\pi} \Upsilon^* + \tau \KL(m_{\pi}\|m_{\mu})
+\tau \KL(m_{\pi}\|m_{\nu}) +\varepsilon \KL(m_{\pi}\|m_{\mu}m_{\nu}).
 \end{align*}
 By part (c) of Lemma \ref{lemma:maximizer_linear_log}, we obtain
 \begin{align*}
     m_{\pi^*}= m_{\mu}^{\frac{\tau}{2\tau+\varepsilon}} m_{\nu}^{\frac{\tau}{2\tau + \varepsilon}} (m_{\mu}m_{\nu})^{\frac{\varepsilon}{2\tau + \varepsilon}} \exp\Big\{\frac{-\Upsilon^*}{2\tau + \varepsilon} \Big\}.
 \end{align*}
Hence, we have thus proved our claims.
\end{proof}

\end{document}